\documentclass[a4paper]{scrartcl}

\usepackage[utf8]{inputenc}
\usepackage[T1]{fontenc}

\usepackage{amsmath, amsthm, amssymb}
\usepackage{mathtools}
\usepackage{refcount}
\usepackage{mathrsfs}
\usepackage{bbm} % For extension of mathbb
\usepackage{stmaryrd} % For double brackets (integer intervals)
\usepackage{array,booktabs,multirow}
\usepackage{systeme} % Systems of equations with aligned coefficients
\usepackage{bigdelim} % "Big delimiters" (?)
\usepackage{makecell} % For line breaks inside cells
\usepackage{url}
\usepackage{subcaption} % For subfigures
\DeclareCaptionSubType*[arabic]{figure}

\makeatletter
\let\c@table\c@figure
\let\ftype@table\ftype@figure
\makeatother
\usepackage{adjustbox}
\usepackage{needspace} % Very convenient way to prevent line breaks

\usepackage{enumitem}
\newlist{hypothenum}{enumerate}{3}
\setlist[hypothenum,1]{label=(\roman*)}

\usepackage[table]{xcolor}
\newcommand{\faded}{\color{gray}}

\usepackage[referable]{threeparttablex}
\renewlist{tablenotes}{enumerate}{1}
\makeatletter
\setlist[tablenotes]{label=\tnote{\alph*},ref=\alph*,itemsep=\z@,topsep=\z@skip,partopsep=\z@skip,parsep=\z@,itemindent=\z@,labelindent=\tabcolsep,labelsep=.2em,leftmargin=*,align=left,before={\footnotesize}}
\makeatother
\usepackage{genyoungtabtikz}

\theoremstyle{plain}
\newtheorem*{mainthm}{Main Theorem}
\newtheorem{theorem}{Theorem}[section]
\newtheorem{proposition}[theorem]{Proposition}

\newtheorem{lemma}[theorem]{Lemma}
\newtheorem{corollary}[theorem]{Corollary}

\theoremstyle{definition}
\newtheorem{definition}[theorem]{Definition}
\newtheorem{example}[theorem]{Example}

\theoremstyle{remark}
\newtheorem{remark}[theorem]{Remark}

\numberwithin{equation}{section}

\newcommand{\eps}{\varepsilon}
\newcommand{\setsuch}[2]{\left\{ #1 \; \middle| \; #2 \right\}}
\newcommand{\restr}[2]{{\left. #1 \right|}_{#2}}
\newcommand{\NN}{\mathbb{N}}
\newcommand{\ZZ}{\mathbb{Z}}
\newcommand{\RR}{\mathbb{R}}
\newcommand{\CC}{\mathbb{C}}
\newcommand{\HH}{\mathbb{H}}

\newcommand{\lie}{\mathfrak}
\newcommand{\twice}{{\scriptstyle 2 \cdot}}
\newcommand{\linv}{\textnormal{$\lie{l}$-inv}}
\newcommand{\theor}{\operatorname{Table}}

\DeclareMathOperator{\sgn}{sgn}
\DeclareMathOperator{\charact}{char}

\DeclareMathOperator{\Id}{Id}

\DeclareMathOperator{\SU}{SU}

\usepackage[all,cmtip]{xy}
\newcommand{\fundef}[5]{
\entrymodifiers={+!!<0pt,\fontdimen22\textfont2>}
\xymatrix@R=3pt{\llap{$#1$\;\;} {#2} \ar@{->}[r] & {#3} \\ {#4} \ar@{|->}[r] & {#5}}
} %Formats function definitions nicely. For more complicated things (bijections with two arrows etc.), redo it by hand. To explain the "\entrymodifiers" line: see https://www.math.lsu.edu/~aperlis/publications/axisalignment/perlis_axisalignment_24Jun2003.pdf

\usepackage{xparse}
\makeatletter
\newenvironment{smallermatrix}[1][c]
{\null\,\vcenter\bgroup
  \Let@\restore@math@cr\default@tag
  \baselineskip0pt \lineskip0.4pt \lineskiplimit0pt
  \ialign\bgroup\if#1l\else\hfil\fi$\m@th\scriptstyle##$\if#1r\else\hfil\fi&&\thickspace\hfil
  $\m@th\scriptstyle##$\hfil\crcr
}{%
  \crcr\egroup\egroup\,%
}
\makeatother %Source: https://tex.stackexchange.com/questions/512608/how-to-arbitrarily-combine-subscript-superscript-and-midscript/512609?noredirect=1#comment1296563_512609

\newcommand\poption[1]{\def\temp{#1}\ifx\temp\empty {} \else \text{\fbox{$\temp$}} \fi}
\newcommand{\threeindwithspace}[3]{
  \begin{smallermatrix}[c]
  \mathstrut\IfValueT{#1}{#1} \\
  \IfValueT{#2}{\poption{#2}}{} \\
  \mathstrut\IfValueT{#3}{#3}
  \end{smallermatrix}%
}
\newcommand{\threeind}[3]{
  \def\haut{#1}
  \def\bas{#3}
  \def\matrix{\threeindwithspace{#1}{#2}{#3}}
  \!
  \ifx\haut\empty
    \ifx\bas\empty \smash{\matrix} \else \smash[t]{\matrix} \fi
  \else
    \ifx\bas\empty \smash[b]{\matrix} \else \matrix \fi
  \fi
  \!%
}

\newtagform{argumenttag}{(}{.\tagargument)}% define new form for tags
\newcommand\labelWithArgument[2]{%
    \usetagform{argumenttag}% use the new tag form
    \def\tagargument{#2}\label{#1}% set tag arguments and label
}
\newcommand\eqrefWithArgument[2]{%
\def\myref{\getrefnumber{#1}}% extract the reference number
{\textup{(\myref.\mbox{#2})}}% print the reference number with argument
}
\newcommand{\ie}{i.e.\ }
\newcommand{\eg}{e.g.\ }

\newcommand\diagramfontsize\footnotesize

\begin{document}

\title{Representations having vectors fixed by a Levi subgroup}
\author{Ilia Smilga \thanks{The author is supported by the European Research Council (ERC) under the European Union Horizon 2020 research and innovation programme, under the grants ICHAOS, No. 647133 (which covered the bulk of the work on this paper) and the ERC starting grant DiGGeS, No. 715982 (which covered only the revision phase).}}
              
\maketitle

\begin{abstract}
For any semisimple real Lie algebra~$\mathfrak{g}_\mathbb{R}$, we classify the representations of~$\mathfrak{g}_\mathbb{R}$ that have at least one nonzero vector on which the centralizer of a Cartan subspace, also known as the centralizer of a maximal split torus, acts trivially.
In the process, we revisit the notion of $\mathfrak{g}$-standard Young tableaux, introduced by Lakshmibai and studied by Littelmann, that provides a combinatorial model for the characters of the irreducible representations of any classical semisimple Lie algebra~$\mathfrak{g}$. We construct a new version of these objects, which differs from the old one for $\mathfrak{g} = \mathfrak{so}(2r)$ and seems, in some sense, simpler and more natural.
\end{abstract}            

\section{Introduction}
\label{sec:intro}

\subsection*{Version information}

This is the arXiv version of this paper. Another version has been published in the \emph{Journal of Algebra}, with a slightly different title as per the reviewer's request: \emph{Representations having vectors fixed by a Levi subgroup associated to a real form}. I chose to keep the old title on arXiv for continuity reasons only; both versions contain no more than a hint to a possible generalization to Levi subgroups not coming from real forms.

Note that the published version is also much more terse than the preprint, by a factor of about 2/3. Spacing constraints forced me to throw out a lot of details of the proofs there, and more generally to do away with a lot of niceties to the reader.

Also note that the theorem and equation numbering is not consistent across the two versions.

\subsection{Background and motivation}
\label{sec:background}

The present work is motivated by the following geometrical result, proved by the author earlier. Recall that, for a semisimple real Lie group~$G_\RR$, the \emph{restricted Weyl group} is the group $W := N_{G_\RR}(\lie{a})/Z_{G_\RR}(\lie{a})$, where $\lie{a}$ is the Cartan subspace (or maximal split torus) of~$G_\RR$; the \emph{longest element} of~$W$ is the unique element that maps all positive restricted roots to negative restricted roots. Also let $L$ denote the centralizer $Z_{G_\RR}(\lie{a})$ (sometimes also known as $MA$, where $M$ is the centralizer of~$A$ in the maximal compact subgroup of~$G_\RR$).
\begin{theorem}[\cite{Smi16b}]
Let $G_\RR$ be a semisimple real Lie group, $\rho$ a representation of~$G_\RR$ on a finite-dimensional real vector space~$V$. Assume that $\rho$ satisfies the following algebraic condition:
\begin{itemize}
\item[(*)] the longest element $w_0$ of the restricted Weyl group $W$ of~$G_\RR$ acts (via $\rho$) nontrivially on the subspace $V^L$ of vectors of~$V$ that are fixed by all elements of~$L$.
\end{itemize}
Then the representation $\rho$ has the following geometric property:
\begin{itemize}
\item[(**)] The affine group $G_\RR \ltimes_\rho V$ contains a subgroup~$\Gamma$ that is free (of rank at least $2$), has linear part Zariski-dense in~$G_\RR$, and acts properly discontinuously on the affine space corresponding to~$V$.
\end{itemize}
\end{theorem}
Moreover, it is conjectured that the converse is true, \ie that every representation with the property (**) satisfies the condition (*). The author has found a partial proof~\cite{Smi18} of this converse.

Representations that have property~(**) are called \emph{non-Milnor}, since they provide counterexamples to a conjecture by Milnor~\cite{Mil77}. A weaker version of this conjecture, due to Auslander~\cite{Aus64}, remains open to this day, and provides the main motivation to a large body of work that includes, besides the author's two papers cited above, \cite{AMS02,AMS11,AMS12,DGK20,Dru92,FG83,GT,Mar83,Tom16}
and many others. For a concise statement of these two conjectures and a brief overview of this background, see the introduction to~\cite{Smi16b}. For a more detailed exposition, see the surveys \cite{AbSur} or~\cite{DDGS}.

This theorem naturally raises the problem of explicitly classifying the representations that satisfy (*). In an earlier paper~\cite{LFlSm}, we did this in the special case where $G_\RR$ is split: then (by definition) the Cartan subspace (or maximal split torus) $\lie{a}$ is also a Cartan subalgebra (or maximal torus) of~$G_\RR$; hence its centralizer~$L$ coincides with the Lie group~$A$ that it generates, and $V^L = V^A$ is just the weight space in~$V$ corresponding to the weight~$0$. So \cite{LFlSm} consisted in classifying the representations where $w_0$ acts nontrivially on this zero-weight space.

In the general case, a natural first step towards this problem consists in classifying the representations for which at least the subspace~$V^L$ itself is nontrivial. This latter classification is the goal of the present paper.

% Reduction steps:
%  group -> algebra.
%  arbitrary -> irreducible representation.
%  representation of g_R -> representation of g.
%  complex -> real representation. Remove item 12 from the notations? remove some more items from the notations?
%  semisimple -> simple.

\subsection{Basic notations}
\label{sec:notations}

We now introduce some notations necessary to formulate the main theorem, and used throughout the paper. Most of them are standard; we have highlighted with asterisks those that are neither universally accepted nor easily guessable from context, even for a reader familiar with the theory of semisimple Lie algebras and their representations.

\begin{enumerate}
\item Let $\lie{g}$ be a semisimple complex Lie algebra, $\lie{g}_\RR$ some real form of~$\lie{g}$ (so that $\lie{g} = (\lie{g}_\RR)^\CC$).
\item We choose in~$\lie{g}_\RR$ a Cartan subspace~$\lie{a}_\RR$ (an abelian subalgebra of $\lie{g}_\RR$ whose elements are diagonalizable over~$\RR$ and which is maximal for these properties); we set $\lie{a} := (\lie{a}_\RR)^\CC$.
\item We choose in~$\lie{g}$ a Cartan subalgebra~$\lie{h}$ (an abelian subalgebra of $\lie{g}$ whose elements are diagonalizable and which is maximal for these properties) that contains~$\lie{a}$.
\item We denote by~$\lie{l}(\lie{g}_\RR)$, or simply $\lie{l}$ when clear from context, the centralizer of~$\lie{a}$ in~$\lie{g}$. %, and by $\lie{l}_\RR$ [...] We then have $\lie{l} = (\lie{l}_\RR)^\CC = \lie{l}_\RR \oplus i \lie{l}_\RR$.
\item Let $\Delta$~be the set of roots of~$\lie{g}$ in~$\lie{h}^*$. We shall identify $\lie{h}^*$ with~$\lie{h}$ via the Killing form. We call $\lie{h}_{(\RR)}$ the $\RR$-linear span of~$\Delta$; it is given by the formula $\lie{h}_{(\RR)} = \lie{a}_\RR \oplus i \lie{t}_\RR$, where $\lie{t}_\RR$ is the orthogonal complement of $\lie{a}_\RR$ in~$\lie{h} \cap \lie{g}_\RR$.
\item \label{itm:lex_ord_choice} We choose on~$\lie{h}_{(\RR)}$ a lexicographical ordering that ``puts $\lie{a}_\RR$ first'', \ie such that every vector whose orthogonal projection onto~$\lie{a}_\RR$ is positive is itself positive. (This condition is necessary to make Proposition~\ref{Levi_is_Levi} work.) We call $\Delta^+$ the set of roots in~$\Delta$ that are positive with respect to this ordering, and we let $\Pi = \{\alpha_1, \ldots, \alpha_r\}$ be the set of simple roots in~$\Delta^+$. Let~$\varpi_1, \ldots, \varpi_r$ be the corresponding fundamental weights.
\bgroup \let\oldenum\labelenumi \renewcommand\labelenumi{\oldenum *}
  \item We call $P$ (resp. $Q$) the weight lattice (resp. root lattice), \ie the abelian subgroup of~$\lie{h}^*$ generated by $\varpi_1, \ldots, \varpi_r$ (resp. by~$\Delta$). Elements of~$P$ are called \emph{integral weights}.
\egroup
\item We introduce the dominant Weyl chamber:
\[\lie{h}^+ := \setsuch{X \in \lie{h}_{(\RR)}}{\forall \alpha \in \Pi,\quad \alpha(X) \geq 0}.\]
\bgroup \let\oldenum\labelenumi \renewcommand\labelenumi{\oldenum *}
  \item When $\lie{g}$ is simple, we call $(e_1, \ldots, e_n)$ the vectors called $(\eps_1, \ldots, \eps_n)$ in the appendix to \cite{BouGAL456}, which form a convenient basis of a vector space containing $\lie{h}_{(\RR)}^*$. Throughout the paper, we use the Bourbaki conventions \cite{BouGAL456} for the numbering of simple roots and their expressions in the coordinates $e_i$.
  \item Given an integral weight $\lambda \in P$, we always denote $\lambda_1, \ldots, \lambda_n$ its coordinates in this last basis: %and $x_1, \ldots, x_r$ its coordinates in the basis formed by fundamental weights:
\begin{equation}
\label{eq:coordinate_definition}
\lambda =: \sum_{i=1}^n \lambda_i e_i.
% =: \sum_{i=1}^r x_i \varpi_i.
\end{equation}
\egroup
\item In the sequel, all representations are supposed to be finite-dimensional and (except for a brief discussion at the beginning of the next subsection) complex. Recall (\cite[Theorem~5.5]{Kna96} or \cite[Theorems 9.4 and~9.5]{Hall15}) that to every irreducible representation of~$\lie{g}$, we may associate, in a bijective way, a vector $\lambda \in P \cap \lie{h}^+$ called its \emph{highest weight}. We denote by $\rho_\lambda(\lie{g})$ the irreducible representation of~$\lie{g}$ with highest weight~$\lambda$, and by $V_\lambda(\lie{g})$ the space on which it acts. When clear from context, we will shorten $V_\lambda(\lie{g})$ to~$V_\lambda$.
%\item Recall (\cite[Proposition~7.15]{Kna96}) that the restriction map $\rho \mapsto \restr{\rho}{\lie{g}_\RR}$ induces a bijection between irreducible representations of~$\lie{g}$ and those of~$\lie{g}_\RR$. We denote by $\rho_\lambda(\lie{g}_\RR)$ the restriction of~$\rho_\lambda(\lie{g})$ to~$\lie{g}_\RR$, still acting on~$V_\lambda(\lie{g})$.
\item Given a representation $V$ of~$\lie{g}$, we denote by $V^\lie{l} := \setsuch{v \in V}{\forall l \in \lie{l},\;\; l \cdot v = 0}$ the $\lie{l}$-invariant subspace of~$V$. %Clearly this coincides with the $\lie{l}_\RR$-invariant subspace of~$V$.
\bgroup \let\oldenum\labelenumi \renewcommand\labelenumi{\oldenum *}
  \item We denote by $\lie{sp}_{\twice n} (\CC)$, $\lie{sp}_{\twice n} (\RR)$ and $\lie{sp}_\twice (p,q)$ some Lie algebras that have rank $n$ (or~$p+q$) and a standard representation of dimension $2n$ (or~$2p+2q$). Some authors, such as Bourbaki~\cite{BouGAL456}, denote them respectively by $\lie{sp}_{2n} (\CC)$, $\lie{sp}_{2n} (\RR)$ and $\lie{sp}(2p,2q)$; while other authors, such as Knapp~\cite{Kna96}, denote them respectively by $\lie{sp}_n (\CC)$, $\lie{sp}_n (\RR)$ and $\lie{sp}(p,q)$.
\egroup
\end{enumerate}

\subsection{Statement of main result}
\label{sec:main_result}

Here is the main theorem of this paper. In the process of proving it, we develop some other results that may be interesting on their own: we will describe these in Subsubsection~\ref{sec:intro_dYt}.

This theorem solves the problem raised at the end of Subsection~\ref{sec:background}, after doing the following reductions:
\begin{enumerate}
\item \label{itm:group_to_alg} from representations of the Lie group~$G_\RR$ to representations of its Lie algebra~$\lie{g}_\RR$;
\item \label{itm:Rrep_to_Crep} from real representations to complex representations;
\item \label{itm:Rgrp_to_Cgrp} from representations of~$\lie{g}_\RR$ to representations of its complexification~$\lie{g}$;
\item \label{itm:arb_to_irrep} from arbitrary to irreducible representations;
\item \label{itm:semisimple_to_simple} from the case where $\lie{g}_\RR$ is semisimple to the case where $\lie{g}_\RR$ is simple.
\end{enumerate}
These reductions rely on classical results of Lie theory and representation theory, and happen without any surprises. See also \cite{Smi20} for more details about the steps \ref{itm:group_to_alg} and \ref{itm:Rgrp_to_Cgrp}--\ref{itm:semisimple_to_simple}.

\begin{mainthm}
Let $\lie{g}_\RR$ be a simple real Lie algebra. Then the set
\[\mathcal{M}_{\linv}(\lie{g}_\RR) := \setsuch{\lambda \in P \cap \lie{h}^+}{V_\lambda^{\lie{l}(\lie{g}_\RR)} \neq 0}\]
is equal to the set $\mathcal{M}_{\theor}(\lie{g}_\RR)$, defined as follows:
\begin{hypothenum}
\item \label{itm:main_classical} $\mathcal{M}_{\theor}$ is listed in Table~\ref{tab:conditions_for_classical_algebras} when the complexification~$\lie{g}$ is a classical simple Lie algebra;
\item \label{itm:main_excep_compact} $\mathcal{M}_{\theor} = \{0\}$ when $\lie{g}$ is an exceptional simple Lie algebra, and $\lie{g}_\RR$ is its compact real form;
\item \label{itm:main_excep_noncompact} $\mathcal{M}_{\theor} = Q \cap \lie{h}^+$ when $\lie{g}$ is an exceptional simple Lie algebra, and $\lie{g}_\RR$ is any noncompact real form;
\item \label{itm:main_complex} $\mathcal{M}_{\theor} = Q \cap \lie{h}^+$ if $\lie{g}$ is not simple.
\end{hypothenum}
\end{mainthm}

We recall (see \eg \cite[Theorem~6.94]{Kna96}) that this last case occurs if and only if $\lie{g}_\RR$ is ``already'' complex, \ie if it is obtained from some complex Lie algebra $\lie{g}_\CC$ by restriction of scalars; and in this case we have $\lie{g} = (\lie{g}_\RR)^\CC \simeq \lie{g}_\CC \oplus \lie{g}_\CC$.

\begin{table}[p]
  \caption{\label{tab:conditions_for_classical_algebras} Conditions for $V_\lambda^\lie{l} \neq 0$ for real forms of classical simple Lie algebras. Here we decompose $\lambda = \sum_{i=1}^n \lambda_i e_i$; when $\lie{g}$ is of type $B_r$, $C_r$ or~$D_r$, we adopt the additional convention that $\lambda_i = 0$ for all $i > r$. See also Table~\ref{tab:root_lattice_congruences} for a reminder of the (well-known) explicit equations for $\lie{h}^+$, $P$ and~$Q$ in these coordinates.}
  \centering\bigskip
  \makebox[\textwidth][c]{
  \begin{tabular}[t]{llll}
    $\lie{g}$ & $\lie{g}_\RR$ & Parameter range & $\lambda \in \mathcal{M}_{\theor}$ iff... \\
    \midrule
    \multirow{4}{*}{$\underset{r \geq 1}{A_r}$}
    & $\lie{su}(p,r+1-p)$ & $0 \leq p < \frac{r+1}{2}$ & $\lambda \in Q \cap \lie{h}^+$ and $\lambda_{r+1-2p} \geq 0 \geq \lambda_{2p+1}$ \\
    & $\lie{su}(p,p)$ & $p = \frac{r+1}{2}$ ($r$ odd) & $\lambda \in Q \cap \lie{h}^+$ \\
    & $\lie{sl}_{r+1}(\RR)$ &  & $\lambda \in Q \cap \lie{h}^+$ \\
    & $\lie{sl}_{m}(\HH)$ & $m = \frac{r+1}{2}$ ($r$ odd) & $\lambda \in Q \cap \lie{h}^+$ and $\sum_{i=2}^{m+1} \lambda_i \geq 0 \geq \sum_{i=m}^{2m-1} \lambda_i$ \\
    \midrule
    \multirow{1}{*}{$\underset{r \geq 1}{B_r}$}
    & $\lie{so}(p,2r+1-p)$ & $0 \leq p \leq r$ &
$\begin{cases}
\lambda \in Q \cap \lie{h}^+,\; \lambda_{2p+1} = 0 \text{ and} \\
\text{if } \sum_{i=1}^r \lambda_i \equiv 1 \pmod{2}, \text{ then } \lambda_{2r-2p+1} > 0
\end{cases}$ \\
    \midrule
    \multirow{4}{*}{$\underset{r \geq 1}{C_r}$} % it should logically be \multirow{3} but for some reason this does not look nice
    & $\lie{sp}_\twice (1, 1)$ & & $\lambda \in Q \cap \lie{h}^+$ and $\lambda_2 \in 2\ZZ$ \\
    & $\lie{sp}_\twice (p, r-p)$ & $\begin{cases} 0 \leq p \leq \frac{r}{2} \\ (p,r) \neq (1,2) \end{cases}$ & $\lambda \in Q \cap \lie{h}^+$ and $\lambda_{4p+1} = 0$ \\
    & $\lie{sp}_{\twice r}(\RR)$ & & $\lambda \in Q \cap \lie{h}^+$ \\
    \midrule
    \multirow{3}{*}{$\underset{r \geq 3}{D_r}$}
    & $\lie{so}(p,2r-p)$ & $0 \leq p \leq r$ & $\lambda \in Q \cap \lie{h}^+$ and $\lambda_{2p+1} = 0$ \\
    & $\lie{so}^*(6)$ & & $\lambda \in Q \cap \lie{h}^+$ and $|\lambda_3| \leq \lambda_1 - \lambda_2$ \\
    & $\lie{so}^*(2r)$ & $r \geq 4$ & $\lambda \in Q \cap \lie{h}^+$ \\
    \bottomrule
  \end{tabular}
  }
\end{table}

\begin{table}[p]
  \caption{\label{tab:root_lattice_congruences}Dominance, integrality and radicality conditions for classical Lie algebras}
  \centering\bigskip
\makebox[\textwidth][c]{
\begin{tabular}{c>{\hspace{1em}}l>{\hspace{1em}}l>{\hspace{1em}}l}
$\lie{g}$ & $\lambda \in \lie{h}^+$ iff... & $\lambda \in P$ iff... & $\lambda \in Q$ iff... \\[1ex] \hline
$A_{n-1}$ & $\begin{cases} \lambda_1 \geq \cdots \geq \lambda_n \\ \sum_{i=1}^n \lambda_i = 0 \end{cases}$
& $\forall i,j,\;\; \lambda_i - \lambda_j \in \ZZ$ & $\forall i,\;\; \lambda_i \in \ZZ$ \\[1ex]
$B_n$ & $\lambda_1 \geq \cdots \geq \lambda_n \geq 0$ & $\forall i,\;\; \lambda_i \in \ZZ$ \;or\; $\forall i,\;\; \lambda_i \in \ZZ + \frac{1}{2}$ & $\forall i,\;\; \lambda_i \in \ZZ$ \\[1ex]
$C_n$ & $\lambda_1 \geq \cdots \geq \lambda_n \geq 0$ & $\forall i,\;\; \lambda_i \in \ZZ$ & $\forall i,\;\; \lambda_i \in \ZZ$ \;and\; $\sum_{i=1}^n \lambda_i \in 2\ZZ$
\\[1ex]
$D_n$ & $\lambda_1 \geq \cdots \geq \lambda_{n-1} \geq |\lambda_n|$ & $\forall i,\;\; \lambda_i \in \ZZ$ \;or\; $\forall i,\;\; \lambda_i \in \ZZ + \frac{1}{2}$ & $\forall i,\;\; \lambda_i \in \ZZ$ \;and\; $\sum_{i=1}^n \lambda_i \in 2\ZZ$
\end{tabular}
}
\end{table}

\begin{remark}
Note that, for algebras $\lie{su}(p,q)$, $\lie{so}(p,q)$ and $\lie{sp}_\twice(p,q)$, the behaviour is qualitatively different depending on whether $p$ lies in the lower or upper half of its range (the whole range, given the assumption $p \leq q$, goes from $0$ to~$\frac{p+q}{2}$). Table~\ref{tab:conditions_for_sx_p_q} shows how the condition for having $V_\lambda^\lie{l} \neq 0$ simplifies in these two subcases.
\end{remark}

\begin{table}
  \caption{\label{tab:conditions_for_sx_p_q}Necessary and sufficient condition on $\lambda \in Q \cap \lie{h}^+$ to have $V_\lambda^\lie{l} \neq 0$, for algebras $\lie{su}(p,q)$, $\lie{so}(p,q)$ and $\lie{sp}_\twice(p,q)$ with $p \leq q$. (This is just a reformulation of part of Table~\ref{tab:conditions_for_classical_algebras}.)}
  \centering\bigskip
  \begin{threeparttable}
  \begin{tabular}[t]{rc@{}c}
    $\lie{g} \quad$ & $p < \frac{p+q}{4}$ & $p \geq \frac{p+q}{4}$ \\
    \midrule
    $\lie{su}(p,q)$ & $\lambda_{2p+1} = \lambda_{2p+2} = \cdots = \lambda_{q-p} = 0$ & $\lambda_{q-p} \geq 0 \geq \lambda_{2p+1}$\tnotex{tnote:su_pp} \\
    \midrule
    $\underset{p+q \text{ odd}}{\lie{so}(p,q)}$ & $\begin{cases}\lambda_{2p+1} = 0\tnotex{tnote:so_p_3p_plus_1} \\ \sum_{i=1}^{2p} \lambda_i \in 2\ZZ\end{cases}$
    & either $\lambda_{q-p} > 0$, or $\sum_{i=1}^{q-p-1} \lambda_i \in 2\ZZ$ \\
    \midrule
    $\lie{sp}_\twice(p,q)$ & $\lambda_{4p+1} = 0$ & true for all $\lambda \in Q \cap \lie{h}^+$\tnotex{tnote:sp_11} \\
    \midrule
    $\underset{p+q \text{ even}}{\lie{so}(p,q)}$ & $\lambda_{2p+1} = 0$ & true for all $\lambda \in Q \cap \lie{h}^+$ \\
    \bottomrule
  \end{tabular}
  \vspace{2mm}
  \begin{tablenotes}
    \item\label{tnote:su_pp} This formula only makes sense for $p < q$. For $q = p$, we have $V_\lambda^\lie{l} \neq 0$ for all $\lambda \in Q \cap \lie{h}^+$. Note that for $q = p+1$, this condition reduces to $\lambda_1 \geq 0 \geq \lambda_n$, which is also tautologically true.
    \item\label{tnote:so_p_3p_plus_1} When $p = \frac{p+q-1}{4}$ (\ie $q = 3p+1$), we have $r = 2p$ and the condition $\lambda_{2p+1} = 0$ becomes tautologically true; only the parity condition remains.
    \item\label{tnote:sp_11} Except for $\lie{sp}_\twice(1,1)$, where the condition is $\lambda_2 \in 2\ZZ$.
  \end{tablenotes}
  \end{threeparttable}
\end{table}

%\begin{table}
%  \rotatebox{90}{
%  \begin{tabular}[t]{r|c|c|c|c|c|}
%    $\lie{g}$ & $p+q \geq 4p+2$ & $p+q = 4p+1$ & $p+q = 4p$ & $4p-1 \geq p+q \geq 2p+1$ & $p+q = 2p$ \\
%    \midrule
%    $\lie{su}(p,q)$ & $\lambda_{2p+1} = \lambda_{2p+2} = \cdots = \lambda_{q-p} = 0$ & $\lambda_{2p+1} = 0$ & \multicolumn{2}{|c|}{$\lambda_{q-p} \geq 0 \geq \lambda_{2p+1}$} & $\top$ \\
%    \midrule
%    $\lie{so}(p,q)$, $p+q$ odd & $\begin{cases}\lambda_{2p+1} = 0 \\ \sum_{i=1}^{2p} \lambda_i \in 2\ZZ\end{cases}$ & $\sum_{i=1}^{2p} \lambda_i \in 2\ZZ$ & \notapplicable & \multicolumn{2}{|c|}{either $\lambda_{q-p} > 0$, or $\sum_{i=1}^{q-p-1} \lambda_i \in 2\ZZ$} \\
%    \midrule
%    $\lie{sp}_\twice(p,q)$ & \multicolumn{2}{|c|}{$\lambda_{4p+1} = 0$} & \multicolumn{2}{|c|}{$\top$} & $\top$ (Footnote: except for $\lie{sp}_\twice(1,1)$: true iff $\lambda_2 \in 2\ZZ$)
%    \midrule
%    $\lie{so}(p,q)$, $p+q$ even & $\lambda_{2p+1} = 0$ & \notapplicable & \multicolumn{3}{|c|}{$\top$} \\
%    \bottomrule
%  \end{tabular}
%  }
%\end{table}

\subsection{Strategy of the proof}
\label{sec:strategy}

We start by treating some trivial cases in Subsection~\ref{sec:trivial}. In Subsection~\ref{sec:Levi}, we make some preliminary remarks about a class of subalgebras that includes~$\lie{l}$: the Levi subalgebras.

The first key idea, suggested to the author by E. B. Vinberg and presented in Subsection~\ref{sec:additivity}, is the following. By using the properties of the so-called Cartan product, we prove (in three lines) that the set $\mathcal{M}_\linv$ is closed under addition. For any given value of~$\lie{g}_\RR$, this reduces the problem to a finite number of computations.

Now in order to compute the dimension of~$V_\lambda^\lie{l}$ given $\lambda$ and~$\lie{l}$, we need so-called \emph{branching rules}, \ie rules giving the decomposition into irreducible representations of the restriction of~$V_\lambda$ from~$\lie{g}$ to~$\lie{l}$. For exceptional Lie algebras, of which there are a finite number, we only need (thanks to the previous paragraph) to do a finite number of computations; so all we need is an algorithmic implementation of branching rules, which is readily available. The case of exceptional Lie algebras is treated in Section~\ref{sec:exceptional}.

For classical Lie algebras, since there are an infinite number of them, we need a conceptual description of branching rules. Such descriptions do exist in the special case of restrictions to Levi subalgebras. Note that the methods of this paper could probably be adapted to obtain a generalization of the Main Theorem to all Levi subalgebras~$\lie{l}$ of~$\lie{g}$, not just those that come from some real form.

The easier case, treated in Section~\ref{sec:Ar}, is when $\lie{g}_\RR$ has a complexification $\lie{g}$ of type~$A_r$ (\ie $\lie{g} = \lie{sl}_{r+1}(\CC)$). Then there is a classical theory (that we recall in Subsection~\ref{sec:Young}) that identifies, for every~$\lambda$, the character of the representation~$V_\lambda(\lie{g})$ with the set of semistandard Young tableaux of shape~$\lambda$. Then the irreducible components of the restriction of~$V_\lambda(\lie{g})$ to a Levi subalgebra $\lie{l} \subset \lie{g}$ are identified with an easily-described subset of these Young tableaux. This reduces the proof of the main theorem for real forms of $A_r$ to a problem of combinatorics, that we solve in Subsections \ref{sec:supq} and~\ref{sec:slmH}.

\subsubsection{Doubled Young tableaux}
\label{sec:intro_dYt}

In order to treat the remaining classical algebras, we need a similar description of the characters of representations of~$\lie{g}$, when $\lie{g}$ is of type~$B_r$, $C_r$ or~$D_r$. Such a description has been found by Littelmann, Lakshmibai and Seshadri, although, as far as the author knows, the details of their proof have never been published. We provide a complete proof, and also present a slight improvement of their construction in the case $D_r$; this proof, and this improvement, may be of interest independently of the remainder of the paper.

More precisely, Littelmann found~\cite{Lit95} a description of the characters of representations (and a Levi branching rule) for an arbitrary semisimple Lie algebra~$\lie{g}$. His construction, that we recall in Subsection~\ref{sec:Littelmann}, starts with the data of some path~$\pi^+$ that connects $0$ to the weight~$\lambda$, and lies entirely within the dominant Weyl chamber~$\lie{h}^+$. He then identifies the character of the representation $V_\lambda$ with a certain finite set of paths $B_{\pi^+}$ (called the \emph{Littelmann path model}) obtained from~$\pi^+$ by iterating some finite set of operations; and the irreducible components of the restriction of~$V_\lambda(\lie{g})$ to a Levi subalgebra $\lie{l} \subset \lie{g}$ to some easily-described subset of these paths.

We need however an even more explicit characterization of the path model, at least for the algebras $B_r$, $C_r$ and~$D_r$. Lakshmibai hinted at such a characterization in~\cite{Lak86}; Littelmann fully developed it in the appendix to~\cite{Lit90}, and hinted at a proof. They gave, for representations of classical Lie algebras, a character formula involving Young tableaux on the alphabet $\{1, \ldots, n, \overline{n}, \ldots, \overline{1}\}$ satisfying some additional properties. Although these results predate the introduction of the path model, the set of these Young tableaux can be reinterpreted as the description of the path model $B_{\pi^+_{0, \mathrm{L}}(\lambda)}$ obtained from the starting path
\begin{equation}
\label{eq:Littelmann_starting_path}
\pi^+_{0, \mathrm{L}}(\lambda) = x_1 \varpi_1 * \cdots * x_r \varpi_r,
\end{equation}
defined (see Definition~\ref{path_notations}) as the concatenation of $r$ linear segments, each of them equal to (a suitable translation of) the linear path going from $0$ to~$x_i \varpi_i$, where $x_1, \ldots, x_r$ denote the coordinates of~$\lambda$ in the basis $\varpi_1, \ldots, \varpi_r$.

However, we slightly modify this construction. After some preliminary work in Subsections \ref{sec:Bruhat} to~\ref{sec:admissible}, we give (and prove) in Subsection~\ref{sec:dYt} a description of the path model $B_{\pi^+_0(\lambda)}$ obtained from the path
\begin{equation}
\label{eq:basic_path_definition}
\pi^+_0(\lambda) := (\lambda_2 - \lambda_1) c_1 * \cdots * (\lambda_r - \lambda_{r-1}) c_{r-1} * |\lambda_r| c_r^{\sgn(\lambda_r)},
\end{equation}
where we set $c_k = e_1 + \cdots + e_k$ for $k = 1, \ldots, r$, $c_r^+ = c_r$ and $c_r^- = c_{r-1} - e_r$. This description also involves some Young tableaux on the alphabet $\{1, \ldots, n, \overline{n}, \ldots, \overline{1}\}$, that we call \emph{$\lie{g}$-standard doubled Young tableaux}: see Definition~\ref{tableau_definition}. (We chose this terminology because every column is ``split into two'' in some sense: see Remark~\ref{why_doubled}). This reduces the proof of the main theorem for $\lie{g}$ of types $B_r$, $C_r$ and~$D_r$ to a problem of combinatorics, which we solve in Section~\ref{sec:BCD_proof}.

When $\lie{g}$ is of type $B_r$ or~$C_r$, or when $\lie{g}$ is of type $D_r$ but we have $x_{r-1} = x_r = 0$, then the two starting paths $\pi^+_{0, \mathrm{L}}(\lambda)$ and $\pi^+_0(\lambda)$ actually coincide; and indeed, in this case our ``$\lie{g}$-standard doubled Young tableaux'' coincide (up to some cosmetic differences) with the ``$G$-standard Young tableaux'' defined in the appendix of~\cite{Lit90}.

However when $\lie{g}$ is of type $D_r$ (and $x_{r-1} > 0$ or $x_r > 0$), the two starting paths differ, and so do the two notions of standard Young tableaux. Our notion is simpler: namely the complicated condition~(3) from \cite[A.3]{Lit90} is replaced by the much simpler condition~\ref{itm:dsYt_Young} from our Definition~\ref{tableau_definition}. The key result that makes this simplification possible is Proposition~\ref{Bruhat_vs_Young}, that allows us to circumvent the problem of the non-transitivity of the ``Bruhat order'' on weights (see the remark preceding Definition~\ref{Bruhat_tuple_definition}.)

In the future, the author has some hope of extending this improved construction to exceptional Lie algebras, by finding, for arbitrary $\lie{g}$ and~$\lambda$, a starting path $\pi^+_0(\lambda)$ that satisfies some sort of generalization of Proposition~\ref{Bruhat_vs_Young}.

\subsection{Acknowledgements}
\label{sec:acknowledgements}

I would like to thank Peter Littelmann for some helpful discussions. In particular, he greatly helped me in navigating the available literature; for this, I would also like to thank the anonymous referee. I also thank the tex.stackexchange community, for their help in typesetting this document; and the IH{\'E}S, where part of the work on this manuscript was carried out, for their amazing working conditions.

\section{Basic remarks}
\label{sec:basic}

We start by establishing some basic properties of the set~$\mathcal{M}_{\linv}$ of integral weights~$\lambda$ such that $V_\lambda^\lie{l} \neq 0$, which will allow us to prove the Main Theorem in some easy cases and lay the groundwork for the proof in the remaining cases.

\subsection{Trivial cases}
\label{sec:trivial}

We start with some trivial remarks.

\begin{proposition}
\label{basic_results}
Let $\lambda \in P \cap \lie{h}^+$ be a dominant integral weight of~$\lie{g}$.
\begin{hypothenum}
\item \label{itm:nontriv_implies_radical} If $V_\lambda^\lie{l} \neq 0$, then necessarily $\lambda \in Q$.
\item \label{itm:main_split} If $\lie{g}_\RR$ is split, quasi-split (\ie $\lie{l}$ is abelian), or complex, then, conversely, $V_\lambda^\lie{l} \neq 0$ for all $\lambda \in Q \cap \lie{h}^+$.
\item \label{itm:main_compact} If $\lie{g}_\RR$ is compact, then $V_\lambda^\lie{l} \neq 0$ if and only if $\lambda = 0$.
%\item \label{itm:semisimple} Assume that $\lie{g}_\RR = \lie{g}_{\RR, 1} \oplus \cdots \oplus \lie{g}_{\RR, k}$ is a semisimple real Lie algebra. Let $\lambda = \lambda_{(1)} + \cdots + \lambda_{(k)}$ be the decomposition of~$\lambda$ into components along the Cartan subalgebras of the $\lie{g}_i$. Then $V_\lambda^\lie{l} \neq 0$ if and only if, for every $i = 1, \ldots, k$, the irreducible representation $V_{i, \lambda_{(i)}}$ of $\lie{g}_i$ with highest weight $\lambda_i$ satisfies $V_{i, \lambda_{(i)}}^{\lie{l}_i} \neq 0$. %TODO Extend to reductive algebras?
\end{hypothenum}
\end{proposition}
This settles points \ref{itm:main_complex} and~\ref{itm:main_excep_compact} of the Main Theorem, as well as point~\ref{itm:main_classical} for the groups $\lie{sl}_{r+1}(\RR)$ and~$\lie{sp}_{\twice r}(\RR)$ (that are split) and $\lie{su}(p,p)$ (that is quasi-split).
\needspace{2\baselineskip}
\begin{proof}~
\begin{hypothenum}
\item We always have $\lie{a} \subset \lie{h} \subset \lie{l}$, hence
\[V_\lambda^\lie{l} \subset V_\lambda^\lie{h} = V_\lambda^0,\]
where $V_\lambda^0$ denotes the zero-weight space of~$V_\lambda$. The latter is nontrivial if and only if $0$ is a weight of~$V_\lambda$, which, by the well-known characterization of the set of weights of a representation (see \eg \cite[Theorem~10.1]{Hall15}), occurs if and only if $\lambda \in Q$.
\item We claim that in all three cases, we have $\lie{l} = \lie{h}$, so that the converse also holds.

If $\lie{g}_\RR$ is split, then we actually have $\lie{a} = \lie{h} = \lie{l}$. If $\lie{g}_\RR$ is quasi-split, then $\lie{l}$, being an abelian subalgebra containing $\lie{h}$, must be equal to $\lie{h}$ (by maximality of $\lie{h}$).

Finally, if $\lie{g}_\RR$ is ``complex'', or more precisely obtained by restriction of scalars (from $\CC$ to~$\RR$) from some simple complex Lie algebra $\lie{g}_\CC$, then $\lie{g}_\RR$ is in fact also quasi-split. Indeed, let $J: \lie{g}_\RR \to \lie{g}_\RR$ be the $\RR$-linear map corresponding to multiplication by~$i$ in~$\lie{g}_\CC$; it lifts to a $\CC$-linear map $\lie{g} \to \lie{g}$ that we shall also denote by~$J$. Then it is straightforward to check that $\lie{h} = \lie{a} \oplus J \lie{a}$. Now since $J$ commutes with everything, $\lie{l}$~is also the centralizer of~$\lie{h}$, so it must coincide with~$\lie{h}$.
\item If $\lie{g}_\RR$ is compact, then obviously $\lie{a} = 0$, $\lie{l} = \lie{g}$, and $V_\lambda^\lie{l} \neq 0$ if and only if $V_\lambda$~is the trivial representation, \ie if and only if $\lambda = 0$. \qedhere
%\item It is straightforward to check that in this setting, we have (with the obvious notations):
%\begin{align*}
%\lie{a} &= \lie{a}_1 \oplus \cdots \oplus \lie{a}_k; \\
%\lie{l} &= \lie{l}_1 \oplus \cdots \oplus \lie{l}_k; \\
%V_\lambda(\lie{g}) &= V_{\lambda_{(1)}}(\lie{g}_1) \otimes \cdots \otimes V_{\lambda_{(k)}}(\lie{g}_k); \\
%V_\lambda(\lie{g})^\lie{l} &= V_{\lambda_{(1)}}(\lie{g}_1)^{\lie{l}_1} \otimes \cdots \otimes V_{\lambda_{(k)}}(\lie{g}_k)^{\lie{l}_k}. \qedhere
%\end{align*}
\end{hypothenum}
\end{proof}

\begin{center}
\fbox{\begin{minipage}{0.87\textwidth}
For the remainder of the paper, we assume that both the real Lie algebra~$\lie{g}_\RR$ and its complexification~$\lie{g}$ are simple.
\end{minipage}}
\end{center}

\subsection{Levi subalgebras}
\label{sec:Levi}

The goal of this subsection is to describe $\lie{l}$ in purely complex terms, so that we will (almost) not need to care about $\lie{g}_\RR$ any more.

%As we already mentioned, complex representations of $\lie{g}_\RR$ are in bijection with complex representations of $\lie{g}$, and $\lie{l}_\RR$-invariant vectors coincide with $\lie{l}$-invariant vectors. This means that our problem is mostly a problem about $\lie{g}$, so that we will often fix $\lie{g}$ and let its real form $\lie{g}_\RR$ vary. In fact, the only information about~$\lie{g}_\RR$ that matters to us is the algebra~$\lie{l}$, that we shall denote as $\lie{l}(\lie{g}_\RR)$ when adopting this perspective. This is a so-called Levi subalgebra of~$\lie{g}$ (hence the letter~$\lie{l}$); in this short subsection, we recall some basic facts about Levi subalgebras.

\begin{definition}
Let $\Theta \subset \Pi$ be a set of simple roots. We define the \emph{Levi subalgebra of type~$\Theta$} in~$\lie{g}$ to be
\[\lie{l}(\Theta) := \lie{h} \oplus \bigoplus_{\alpha \in \Delta \cap \langle \Theta \rangle} \lie{g}^\alpha,\]
where $\langle \Theta \rangle$ denotes the linear span of $\Theta$.
\end{definition}

In other terms, $\lie{l}(\Theta)$ is a reductive Lie algebra whose Cartan subalgebra coincides with~$\lie{h}$, and whose root system is the subsystem of~$\Delta$ generated by~$\Theta$. Now the following result is straightforward (and well-known).

\begin{proposition}
\label{Levi_is_Levi}
For every real form~$\lie{g}_\RR$ of~$\lie{g}$, define $\Theta(\lie{g}_\RR) := \Pi \cap \lie{a}^\perp$. Then the subalgebra~$\lie{l}$ corresponding to~$\lie{g}_\RR$ is a Levi subalgebra of~$\lie{g}$, of type~$\Theta(\lie{g}_\RR)$:
\[\lie{l}(\lie{g}_\RR) = \lie{l}(\Theta(\lie{g}_\RR)).\]
\end{proposition}
\begin{proof}
It easly to see that
\[\lie{l}(\lie{g}_\RR) = \lie{h} \oplus \bigoplus_{\alpha \in \Delta \cap \lie{a}^\perp} \lie{g}^\alpha.\]
The proposition then follows, provided that $\Delta^+ \cap \lie{a}^\perp$ is a system of positive roots for the root system $\Delta \cap \lie{a}^\perp$. This works because of the way we have chosen the lexicographic order on $\lie{h}_{(\RR)}$ (see item~\ref{itm:lex_ord_choice} in Subsection~\ref{sec:notations}).
\end{proof}

Later we will obtain a purely combinatorial criterion for deciding whether $V^{\lie{l}(\Theta)}_\lambda = 0$ that depends only on the data of~$\Theta$  (Corollaries \ref{characterization_for_sln} and~\ref{combinatorial_characterization}). So the only information about~$\lie{g}_\RR$ that will matter to us is the data of the set $\Theta(\lie{g}_\RR)$.

This set $\Theta(\lie{g}_\RR)$ can be read off the Satake diagram of~$\lie{g}_\RR$ (see \cite{OV90}, Chapter~5, \S{}4, 3\textsuperscript{o} for a definition): it is precisely the set of blackened nodes (called~$\Pi_0$ in~\cite{OV90}). A table of the Satake diagrams of all the simple real Lie algebras is given in \cite{OV90}, Reference Chapter, Table~9.

\subsection{Additivity property}
\label{sec:additivity}

We finish this section by proving Proposition~\ref{closed_under_addition}, which will considerably simplify our task for proving that $\mathcal{M}_{\linv} \supset \mathcal{M}_{\theor}$. Indeed it will now suffice to check that $\mathcal{M}_{\linv}$ contains a basis of the monoid $\mathcal{M}_{\theor}$, which, for any given group, is only a finite computation.

Let $G$ be a simply-connected complex Lie group with Lie algebra $\lie{g}$ and $N$ a maximal unipotent subgroup of~$G$.
Let $\CC[G/N]$ denote the space of regular functions on the variety~$G/N$.
Pointwise multiplication of functions is $G$-equivariant and makes $\CC[G/N]$ into a $\CC$-algebra without zero divisors (because the variety~$G/N$ is irreducible).

\begin{theorem}[{\cite[(3.20)--(3.21)]{VinPopBook}}]\label{thm:Vinberg}
  Each finite-dimensional irreducible representation of~$G$ (or equivalently of its Lie algebra~$\lie{g}$) occurs exactly once as a direct summand of the representation $\CC[G/N]$.
  The $\CC$-algebra $\CC[G/N]$ is graded by the highest weight $\lambda$, in the sense that the product of a vector in $V_\lambda$ by a vector in $V_\mu$ lies in $V_{\lambda+\mu}$ (where $V_\lambda$ stands here for the subrepresentation of $\CC[G/N]$ with highest weight $\lambda$).
\end{theorem}

For given $\lambda$ and $\mu$, we call \emph{Cartan product} the induced bilinear map $\odot: V_\lambda \times V_\mu\to V_{\lambda+\mu}$.
Given $u\in V_\lambda$ and $v\in V_\mu$, this defines $u\odot v\in V_{\lambda+\mu}$ as the projection of $u\otimes v\in V_\lambda\otimes V_\mu=V_{\lambda+\mu}\oplus\dots$.
Since $\CC[G/N]$ has no zero divisor, $u\odot v\neq 0$ whenever $u\neq 0$ and $v\neq 0$.  We deduce the following.

\begin{proposition}
\label{closed_under_addition}
The set~$\mathcal{M}_{\linv}$ is a submonoid of the additive monoid~$Q \cap \lie{h}^+$, \ie is closed under addition.
\end{proposition}
\begin{proof}
Let $\lambda_1$ and~$\lambda_2$ be two elements of this set. Choose any two nonzero vectors $u_1$ and~$u_2$ in $V_{\lambda_1}^{\lie{l}}$ and~$V_{\lambda_2}^{\lie{l}}$ respectively. Then the vector $u_1 \odot u_2$ is in $V_{\lambda_1 + \lambda_2}$, is invariant by~$\lie{l}$, and is still nonzero.
\end{proof}

\section{Exceptional Lie algebras}
\label{sec:exceptional}

We are now ready to prove the Main Theorem for all exceptional simple real Lie algebras. Proposition~\ref{basic_results}~\ref{itm:main_compact} takes care of the compact real forms of $E_6$, $E_7$, $E_8$, $F_4$ and~$G_2$. Proposition~\ref{basic_results}~\ref{itm:main_split} takes care of their split real forms, namely $EI$, $EV$, $EVIII$, $FI$ and~$G$; and of the quasi-split real form $EII$.

For the remaining noncompact real forms of exceptional simple Lie algebras, Proposition~\ref{closed_under_addition} together with Proposition~\ref{basic_results}~\ref{itm:nontriv_implies_radical} show that it suffices to verify that $V_\lambda^{\lie{l}} \neq 0$ for all $\lambda \in Q \cap \lie{h}^+$ that are primitive (\ie not expressible as the sum of two nonzero elements of $Q \cap \lie{h}^+$).

The (finite) list of primitive elements of~$Q \cap \lie{h}^+$ for each exceptionnal simple~$\lie{g}$ can easily be deduced from the equations defining $Q$ in the basis $(\varpi_1, \ldots, \varpi_r)$, which are listed for example in Table~A.1 in the appendix to~\cite{LFlSm} (see the latest arXiv version, as the journal version was published without the appendix). For each of these weights~$\lambda$, we computed the dimension of $V_\lambda^{\lie{l}}$, or in other terms the multiplicity of the trivial representation in the restriction of~$V_\lambda$ to~$\lie{l}$, using branching rules implemented in the software LiE \cite{LiE}. The results are listed in Table~\ref{tab:dimensions_excep_noncompact}. (Note that we saved ourselves some work by taking advantage of the outer automorphism of~$E_6$.) We observe that all these dimensions are indeed nonzero.

\begin{table}[p] %COORD (simple) (fund)
\caption{\label{tab:dimensions_excep_noncompact} Dimensions of $V_\lambda^{\lie{l}}$ for primitive $\lambda \in Q \cap \lie{h}^+$, in real forms of exceptional simple Lie algebras that are neither compact, nor split, nor quasi-split.}
\centering\bigskip
\newcommand{\ditto}{\faded \hspace{2.5em} same \hspace{4em} as \hspace{4em} above \hspace{4.5em}}
\makebox[\textwidth][c]{
\begin{tabular}{cc>{\hspace{1em}}cccccc}
$\lie{g}$ & \multicolumn{1}{c}{\parbox{2.8cm}{\centering Coordinates of $\lambda$\\ in $(\varpi_i)$ basis}} & \multicolumn{1}{c}{\parbox{2.8cm}{\centering Coordinates of $\lambda$\\ in $(e_i)$ basis}} & $\dim V_\lambda$ && \multicolumn{1}{c}{$\mathclap{\dim V_\lambda^{\lie{l}(EIII)}}$} && \multicolumn{1}{c}{$\mathclap{\dim V_\lambda^{\lie{l}(EIV)}}$} \\ \midrule

\multirow{14}{*}{$E_6$} & $(0,0,0,0,0,3)$	& $(0,0,0,0,3,-1,-1,1)$	& 3\,003	&& 2 && 1 \\*
& \faded $(3,0,0,0,0,0)$	& \faded $(0,0,0,0,0,-2,-2,2)$	& \multicolumn{5}{r}{\ditto} \\ 
& $(0,0,0,0,1,1)$	& $(0,0,0,1,2,-1,-1,1)$	& 5\,824	&& 8 && 2 \\*
& \faded $(1,0,1,0,0,0)$	& \faded $\left(-\frac{1}{2},\frac{1}{2},\frac{1}{2},\frac{1}{2},\frac{1}{2},-\frac{3}{2},-\frac{3}{2},\frac{3}{2}\right)$	& \multicolumn{5}{r}{\ditto} \\ 
& $(0,0,0,0,3,0)$	& $(0,0,0,3,3,-2,-2,2)$	& 1\,559\,376	&& 25 && 1 \\*
& \faded $(0,0,3,0,0,0)$	& \faded $\left(-\frac{3}{2},\frac{3}{2},\frac{3}{2},\frac{3}{2},\frac{3}{2},-\frac{5}{2},-\frac{5}{2},\frac{5}{2}\right)$	& \multicolumn{5}{r}{\ditto} \\ 
& $(0,0,0,1,0,0)$	& $(0,0,1,1,1,-1,-1,1)$	& 2\,925	&& 8 && 2 \\ 
& $(0,0,1,0,0,2)$	& $\left(-\frac{1}{2},\frac{1}{2},\frac{1}{2},\frac{1}{2},\frac{5}{2},-\frac{3}{2},-\frac{3}{2},\frac{3}{2}\right)$	& 78\,975	&& 20 && 3 \\*
& \faded $(2,0,0,0,1,0)$	& \faded $(0,0,0,1,1,-2,-2,2)$	& \multicolumn{5}{r}{\ditto} \\ 
& $(0,0,1,0,1,0)$	& $\left(-\frac{1}{2},\frac{1}{2},\frac{1}{2},\frac{3}{2},\frac{3}{2},-\frac{3}{2},-\frac{3}{2},\frac{3}{2}\right)$	& 70\,070	&& 25 && 3 \\ 
& $(0,0,2,0,0,1)$	& $(-1,1,1,1,2,-2,-2,2)$	& 600\,600	&& 41 && 3 \\*
& \faded $(1,0,0,0,2,0)$	& \faded $(0,0,0,2,2,-2,-2,2)$	& \multicolumn{5}{r}{\ditto} \\
& $(0,1,0,0,0,0)$	& $\left(\frac{1}{2},\frac{1}{2},\frac{1}{2},\frac{1}{2},\frac{1}{2},-\frac{1}{2},-\frac{1}{2},\frac{1}{2}\right)$	& 78	&& 3 && 2 \\*
& $(1,0,0,0,0,1)$	& $(0,0,0,0,1,-1,-1,1)$	& 650	&& 6 && 3 \\

& & & & \multicolumn{1}{c}{$\mathclap{\dim V_\lambda^{\lie{l}(EVI)}}$} && \multicolumn{1}{c}{$\mathclap{\dim V_\lambda^{\lie{l}(EVII)}}$} \\ \cmidrule{5-8}

\multirow{10}{*}{$E_7$} & $(0,0,0,0,0,0,2)$	& $(0,0,0,0,0,2,-1,1)$	& 1\,463	& 8 && 4 \\*
& $(0,0,0,0,0,1,0)$	& $(0,0,0,0,1,1,-1,1)$	& 1\,539	& 12 && 6 \\
& $(0,0,0,0,1,0,1)$	& $(0,0,0,1,1,2,-2,2)$	& 980\,343	& 360 && 48 \\ 
& $(0,0,0,0,2,0,0)$	& $(0,0,0,2,2,2,-3,3)$	& 109\,120\,648	& 4\,900 && ~155~ \\ 
& $(0,0,0,1,0,0,0)$	& $(0,0,1,1,1,1,-2,2)$	& 365\,750	& 200 && 30 \\ 
& $(0,0,1,0,0,0,0)$	& $\left(-\frac{1}{2},\frac{1}{2},\frac{1}{2},\frac{1}{2},\frac{1}{2},\frac{1}{2},-\frac{3}{2},\frac{3}{2}\right)$	& 8\,645	& 26 && 9 \\ 
& $(0,1,0,0,0,0,1)$	& $\left(\frac{1}{2},\frac{1}{2},\frac{1}{2},\frac{1}{2},\frac{1}{2},\frac{3}{2},-\frac{3}{2},\frac{3}{2}\right)$	& 40\,755	& 60 && 16 \\ 
& $(0,1,0,0,1,0,0)$	& $\left(\frac{1}{2},\frac{1}{2},\frac{1}{2},\frac{3}{2},\frac{3}{2},\frac{3}{2},-\frac{5}{2},\frac{5}{2}\right)$	& 11\,316\,305	& 1\,553 && 103 \\ 
& $(0,2,0,0,0,0,0)$	& $(1,1,1,1,1,1,-2,2)$	& 253\,935	& 111 && 15 \\*
& $(1,0,0,0,0,0,0)$	& $(0,0,0,0,0,0,-1,1)$	& 133	& 4 && 3 \\

& & & && \multicolumn{1}{c}{$\mathclap{\dim V_\lambda^{\lie{l}(EIX)}}$} \\ \cmidrule{5-7}

\multirow{8}{*}{$E_8$} & $(0,0,0,0,0,0,0,1)$	& $(0,0,0,0,0,0,1,1)$	& 248	&& 4 \\*
& $(0,0,0,0,0,0,1,0)$	& $(0,0,0,0,0,1,1,2)$	& 30\,380	&& 26 \\ 
& $(0,0,0,0,0,1,0,0)$	& $(0,0,0,0,1,1,1,3)$	& 2\,450\,240	&& 188 \\ 
& $(0,0,0,0,1,0,0,0)$	& $(0,0,0,1,1,1,1,4)$	& 146\,325\,270	&& 1\,383 \\ 
& $(0,0,0,1,0,0,0,0)$	& $(0,0,1,1,1,1,1,5)$	& 6\,899\,079\,264	&& 10\,488 \\ 
& $(0,0,1,0,0,0,0,0)$	& $\left(-\frac{1}{2},\frac{1}{2},\frac{1}{2},\frac{1}{2},\frac{1}{2},\frac{1}{2},\frac{1}{2},\frac{7}{2}\right)$	& 6\,696\,000	&& 276 \\ 
& $(0,1,0,0,0,0,0,0)$	& $\left(\frac{1}{2},\frac{1}{2},\frac{1}{2},\frac{1}{2},\frac{1}{2},\frac{1}{2},\frac{1}{2},\frac{5}{2}\right)$	& 147\,250	&& 43 \\ 
& $(1,0,0,0,0,0,0,0)$	& $(0,0,0,0,0,0,0,2)$	& 3\,875	&& 10 \\

& & & & \multicolumn{1}{c}{$\mathclap{\dim V_\lambda^{\lie{l}(FII)}}$} \\ \cmidrule{5-5}

\multirow{4}{*}{$F_4$} & $(0,0,0,1)$	& $(1,0,0,0)$	& 26	& 1 \\*
& $(0,0,1,0)$	& $\left(\frac{3}{2},\frac{1}{2},\frac{1}{2},\frac{1}{2}\right)$	& 273	& 1 \\ 
& $(0,1,0,0)$	& $(2,1,1,0)$	& 1\,274	& 1 \\ 
& $(1,0,0,0)$	& $(1,1,0,0)$	& 52	& 1
\end{tabular}
}
\end{table}

\section{Type $A_r$}
\label{sec:Ar}

We now prove the Main Theorem for $\lie{g} = \lie{sl}_n(\CC)$, which has rank $r = n-1$; for the duration of this section, $n$ is some integer larger than or equal to~$2$.

We will start, in Subsection~\ref{sec:Young}, by establishing some notations and terminology about Young tableaux. We will then treat the case $\lie{g}_\RR = \lie{su}(p,q)$ in Subsection~\ref{sec:supq}, and the case $\lie{g}_\RR = \lie{sl}_m(\HH)$ in Subsection~\ref{sec:slmH}. (For the real form $\lie{g}_\RR = \lie{sl}_n(\RR)$, which is split, the Main Theorem follows from Proposition~\ref{basic_results}.\ref{itm:main_split}.) We will put the pieces together in the brief subsection~\ref{sec:An_conclusion}.

\subsection{Young tableaux: notations and definitions}
\label{sec:Young}

We start by establishing the conventions (Definition~\ref{Young_conventions}) and notations (Definition~\ref{Young_notations}) for the basic manipulation of Young tableaux and diagrams. They will also serve us in the Section~\ref{sec:BCD}.

\begin{definition}
\label{Young_conventions}
Let $n \geq 0$. A \emph{Young diagram of order $n$} is a top- and left-aligned Young diagram with at most $n$~rows. The \emph{shape} of a Young diagram~$\mathcal{P}$ is the $n$-tuple $(\#_1 \mathcal{P}, \ldots, \#_n \mathcal{P})$, where $\#_i \mathcal{P}$ stands for the length of the $i$-th row of~$\mathcal{P}$ (consistently with the next definition); we will often identify the diagram with this tuple.

Let $\mathcal{P}, \mathcal{Q}$ be two Young diagrams. We say that $\mathcal{Q}$ is \emph{contained} in~$\mathcal{P}$, denoted by $\mathcal{Q} \subset \mathcal{P}$, if we have $\#_i \mathcal{Q} \leq \#_i \mathcal{P}$ for all~$i$. If this is the case, we define the \emph{skew diagram} $\mathcal{P}/\mathcal{Q}$ to be the diagram comprising all the boxes that are in~$\mathcal{P}$ but not in~$\mathcal{Q}$.

Fix some ordered set~$\mathcal{A}$. A \emph{Young tableau on the alphabet~$\mathcal{A}$}, denoted for example by~$\mathcal{T}$, is a Young diagram~$\mathcal{P}$ in which each box is filled with an element of~$\mathcal{A}$; we then say that $\mathcal{T}$ is an \emph{$\mathcal{A}$-filling} of~$\mathcal{P}$. We define similarly a \emph{skew tableau on the alphabet~$\mathcal{A}$}. For $n \geq 0$, a \emph{Young tableau of order~$n$} is a $\{1, \ldots, n\}$-filling of a Young diagram of order~$n$.

We say that a Young tableau, or skew tableau, is \emph{semistandard} if the values written in its boxes form a strictly increasing sequence along each column (from top to bottom), and a nondecreasing sequence along each row (from left to right).
\end{definition}
\begin{definition}
\label{Young_notations}
If $\mathcal{P}$ is any diagram or tableau, we introduce the following notations (see Figure~\ref{fig:Young_tableau_example} for an illustration):
\begin{itemize}
\item For $i \in \NN$, we denote by $\threeind{}{}{i} \mathcal{P}$ the $i$-th row of~$\mathcal{P}$ (from the top). %TODO Introduce convention for $i$ out of bounds?
\item For $j \in \NN$, we denote by $\threeind{j}{}{} \mathcal{P}$ the $j$-th column of~$\mathcal{P}$ (from the left).
\item For $I \subset \NN$, we denote by $\threeind{}{}{I} \mathcal{P}$ (resp. $\threeind{I}{}{} \mathcal{P}$) the subtableau or subdiagram comprising all the rows (resp. columns) of~$\mathcal{P}$ indexed by~$I$.
\item We denote by $\# \mathcal{P}$ the total number of boxes in~$\mathcal{P}$.
\end{itemize}
If $\mathcal{T}$ is any tableau on the alphabet~$\mathcal{A}$, we introduce the following notations:
\begin{itemize}
\item For $s \in \mathcal{A}$ (resp. $S \subset \mathcal{A}$), we denote by $\threeind{}{s}{} \mathcal{T}$ (resp. $\threeind{}{S}{} \mathcal{T}$) the subtableau of~$\mathcal{T}$ comprising only the boxes containing the symbol~$s$ (resp. symbols from~$S$).
\item We denote by~$\square \mathcal{T}$ the underlying diagram of~$\mathcal{T}$, \ie the diagram obtained by erasing all the symbols from all the boxes.
\end{itemize}

\begin{figure}
\centering
\begin{align*}
\mathcal{T} &=
  \smash[b]{\begin{tikzpicture}[baseline={([yshift=-.5ex]current bounding box.center)}]
  \tgyoung(0cm,0cm,::::<\threeind{4}{}{} \mathcal{T}>,::::<\raisebox{.3ex}{$\downarrow$}>,112224,2333::<\longleftarrow \threeind{}{}{2} \mathcal{T}>,4,5,:~,:~)  
  \end{tikzpicture}.} &
\square \mathcal{T} &=
  \begin{tikzpicture}[baseline={([yshift=-.5ex]current bounding box.center)}]
  \tgyoung(0cm,0cm,;;;;;;,;;;;,;,;)
  \end{tikzpicture}
= (6, 4, 1, 1, 0); \\
\threeind{4}{}{2} \mathcal{T} &= 3; & \# \mathcal{T} &= 6+4+1+1 = 12. \\
\threeind{}{2}{} \mathcal{T} &= 
  \begin{tikzpicture}[baseline={([yshift=-.5ex]current bounding box.center)}]
  \tgyoung(0cm,0cm,::;222,2)
  \end{tikzpicture}; &
\# \threeind{}{2}{} \mathcal{T} &= 4;
\end{align*}
\DeclareRobustCommand\examplerowname{\threeind{}{}{2} \mathcal{T}}
\DeclareRobustCommand\examplecolname{\threeind{4}{}{} \mathcal{T}}
\DeclareRobustCommand\examplerow{
  \begin{tikzpicture}[baseline={([yshift=-.5ex]current bounding box.center)}]
  \tgyoung(0cm,0cm,2333)
  \end{tikzpicture}}
\DeclareRobustCommand\examplecol{
  \begin{tikzpicture}[baseline={([yshift=-.5ex]current bounding box.center)}]
  \tgyoung(0cm,0cm,2,3)
  \end{tikzpicture}}
\caption{\label{fig:Young_tableau_example} Example of a semistandard Young tableau~$\mathcal{T}$. It has some order $n \geq 5$; the shape of its underlying diagram $\square \mathcal{T}$ is given under the assumption that $n = 5$. The $2$-d row of the tableau~$\mathcal{T}$ is $\examplerowname = \examplerow$, and its $4$-th column is $\examplecolname = \examplecol$.}
\end{figure}

These notations can of course be combined \emph{ad libitum}. For example, $\# \threeind{[1,x]}{[p,q]}{[1,y]} \mathcal{T}$ stands for the total number of occurrences of symbols lying between $p$ and $q$ in the top left $x$-by-$y$ rectangle of~$\mathcal{T}$. We also convene that $\threeind{i}{}{j} \mathcal{T}$ stands (by slight notation abuse) for the symbol that fills the $(i, j)$-th box of~$\mathcal{T}$.
\end{definition}

The following simple (and well-known) trick provides a useful point of view for studying semistandard Young tableaux. Define the \emph{thickness} of a skew diagram $\mathcal{P}/\mathcal{Q}$ as the largest height of one of its columns:
\begin{equation}
\textnormal{thickness}(\mathcal{P}/\mathcal{Q}) := \max_j \#^j \left(\mathcal{P}/\mathcal{Q}\right) = \max_j \left(\#^j \mathcal{P} - \#^j \mathcal{Q}\right).
\end{equation}

\begin{proposition}[Horizontal strip decomposition]
\label{layering_trick}
The set of semistandard Young tableaux of order~$n$ with underlying diagram~$\mathcal{P}$ is in bijection with the set of nested sequences of Young diagrams
\[\emptyset = \mathcal{P}_0 \subset \mathcal{P}_1 \subset \cdots \subset \mathcal{P}_n = \mathcal{P}\]
with the property that, for each $s = 1, \ldots, n$, the skew diagram $\mathcal{P}_s/\mathcal{P}_{s-1}$ is a \emph{horizontal strip}, \ie has thickness at most~$1$.
\end{proposition}
\begin{proof}
Given a semistandard Young tableau $\mathcal{T}$, we associate to it the sequence whose $s$-th term is the diagram comprising the boxes of~$\mathcal{T}$ filled with symbols not exceeding~$s$, \ie formally we set $\mathcal{P}_s := \square \threeind{}{[1,s]}{} \mathcal{T}$.

Conversely, given a sequence with the required properties, we associate to it the tableau~$\mathcal{T}$ obtained by filling, for each~$s = 1, \ldots, n$, all the boxes that appear in the skew tableau $\mathcal{P}_s/\mathcal{P}_{s-1}$ with the symbol~$s$.

It is then straightforward to check that these two maps are well-defined and are reciprocal bijections.
\end{proof}

It is well-known that Young tableaux of order~$n$ are closely related with irreducible representations of~$\lie{g} = \lie{sl}_n(\CC)$: see Proposition~\ref{classical_character_formula} below. More specifically, we can describe the branching rule from $\lie{sl}_n(\CC)$ to a Levi subalgebra in terms of Young tableaux; this is the content of Proposition~\ref{An_Levi_branching_rule} below. In order to state these two propositions, we need a few more definitions.

\begin{definition}[Passing from diagrams and tableaux to weights, in type $A_r$.]
\label{sln_tableau_concepts}
Note that these definitions are only valid for this section. In Section~\ref{sec:BCD}, when $\lie{g}$ will be of type $B_r$, $C_r$ or~$D_r$, we will need to slightly modify them: see Definition~\ref{BCD_tableaux_and_weights}.
\begin{hypothenum}
\item \label{itm:offset_and_shape} Let $\mathcal{P} = (\#_1 \mathcal{P}, \ldots, \#_n \mathcal{P})$ be a Young diagram of order~$n$. We define its \emph{offset} $a(\mathcal{P})$ as its average row length:
\begin{equation}
a(\mathcal{P}) := \frac{1}{n}\# \mathcal{P} = \frac{1}{n} \sum_{i=1}^n \#_i \mathcal{P},
\end{equation}
and its \emph{$\lie{sl}_n$-shape}~$\lambda$ as the orthogonal projection of the vector $\sum \#_i \mathcal{P} e_i$ onto the Cartan subspace~$\lie{h}$ of~$\lie{sl}_n(\CC)$: in other terms, $\lambda = \sum \lambda_i e_i$ with
\begin{equation}
\label{eq:lambda_i_and_P_i}
\forall i = 1, \ldots, n,\quad \lambda_i := \#_i \mathcal{P} - a(\mathcal{P}).
\end{equation}
We observe (compare Table~\ref{tab:root_lattice_congruences}) that this $\lambda$ is always an element of $P \cap \lie{h}^+$. Given some $\lambda \in P \cap \lie{h}^+$, the \emph{reduced} Young diagram of $\lie{sl}_n$-shape~$\lambda$ is the one whose $n$-th row has length~$0$, or equivalently whose offset is equal to~$-\lambda_n$.
\item \label{itm:total_weight} We define the \emph{total weight} $\nu(\mathcal{T})$ of a Young tableau or skew tableau~$\mathcal{T}$ as
\begin{equation}
\label{eq:An_total_weight_definition}
\nu(\mathcal{T}) := \sum_{i, j} \nu\left(\threeind{j}{}{i} \mathcal{T}\right),
\end{equation}
where, for all $s = 1, \ldots, n$, we define $\nu(s)$ as the orthogonal projection of $e_s$ onto~$\lie{h}$:
\begin{equation}
\label{eq:An_symbol_interpretation}
\nu(s) := e_s - \frac{1}{n} \sum_{i=1}^n e_i.
\end{equation}
\item \label{itm:dominance} Given a linear form $\alpha \in \lie{h}^*$, we say that a Young tableau~$\mathcal{T}$ is \emph{$\alpha$-dominant} (resp. \emph{$\alpha$-codominant}) if, whenever we cut~$\mathcal{T}$ between two columns, the total weight of the right part (resp. of the left part) has nonnegative (resp. nonpositive) image by~$\alpha$. In other terms:
\begin{align*}
\mathcal{T} \text{ is $\alpha$-dominant} \quad&:\iff\quad \forall j \geq 0,\quad \alpha \left( \nu \left( \threeind{[j+1,N]}{}{} \mathcal{T} \right) \right) \geq 0; \\
\mathcal{T} \text{ is $\alpha$-codominant} \quad&:\iff\quad \forall j \geq 0,\quad \alpha \left( \nu \left( \threeind{[1,j]}{}{} \mathcal{T} \right) \right) \leq 0,
\end{align*}
where $N = \#_1 \mathcal{T}$ is the width of~$\mathcal{T}$.

In this paper, we will usually consider tableaux of total weight~$0$, for which these two properties are obviously equivalent. Dominance is the most natural property in general, but we will find it more convenient to use codominance.

For a subset $\Theta \subset \Pi$, we say that $\mathcal{T}$ is \emph{$\Theta$-(co)dominant} if it is $\alpha$-(co)dominant for all $\alpha \in \Theta$.
\end{hypothenum}
\end{definition}

We then have the following classical character formula. It is given only for general context; we will not use it directly in the sequel. Recall that the \emph{character} of a representation~$V$ is the formal sum
\begin{equation}
\label{eq:char_defn}
\charact(V) := \sum_{\mu \in \lie{h}^*} \left( \dim V^\mu \right) e^\mu,
\end{equation}
where $V^\mu$ stands for the weight space in~$V$ corresponding to the weight~$\mu$.
\begin{proposition}[Character formula with Young tableaux]
\label{classical_character_formula}
Let $\lambda \in P \cap \lie{h}^+$ be a dominant integral weight of~$\lie{g} = \lie{sl}_n(\CC)$. Then the character of the representation with highest weight~$\lambda$ is given by:
\[\charact(V_\lambda) = \sum_{\mathcal{T}} e^{\nu(\mathcal{T})},\]
where $\mathcal{T}$ runs over all reduced semistandard Young tableaux of order~$n$ and of $\lie{sl}_n$-shape~$\lambda$.
\end{proposition}
For a proof, see \eg \cite{FulHar}, Proposition~15.15 together with the discussion that follows its proof.

We also have the following (closely related) classical branching rule, on which we will rely in the sequel.
\begin{proposition}[Branching rule with Young tableaux]
\label{An_Levi_branching_rule}
Let $\Theta \subset \Pi$ be a set of simple roots of $\lie{g} = \lie{sl}_n(\CC)$, and let $\lambda \in P \cap \lie{h}^+$ be a dominant integral weight. Then we have
\[\restr{V_\lambda(\lie{g})}{\lie{l}(\Theta)} = \bigoplus_\mathcal{T} V_{\nu(\mathcal{T})}(\lie{l}(\Theta)),\]
where $\mathcal{T}$ runs over all reduced $\Theta$-dominant semistandard Young tableaux of order~$n$ and of $\lie{sl}_n$-shape $\lambda$.
\end{proposition}
This is stated in this form in \cite[Theorem~2.2.(b)]{Lit90}, and can be deduced from Littelmann's more general branching rule (\cite[Restriction Rule]{Lit95}, restated here as Proposition~\ref{Vl_in_terms_of_paths}) by using the $\lie{sl}_n$-analog of Proposition~\ref{BCD_path_model_description} (that links Littelmann paths with Young tableaux). It was however certainly known before Littelmann, as one of the multiple avatars of the Littlewood-Richardson rule; see~\cite{McDo}.

Now of course $V_\lambda^{\lie{l}}$ is obtained by selecting, in this decomposition, the summands isomorphic to the trivial representation, \ie such that $\nu(\mathcal{T}) = 0$. So we obtain a criterion for the nontriviality of~$V_\lambda^{\lie{l}}$, namely Corollary~\ref{characterization_for_sln} below. We will however start by introducing one more definition and a couple of remarks, so as to state this criterion in a purely combinatorial way.

\begin{definition}
We say that a Young tableau or skew tableau~$\mathcal{T}$ on an alphabet~$\mathcal{A}$ is \emph{balanced} (with respect to~$\mathcal{A}$) if each symbol from~$\mathcal{A}$ occurs the same number of times:
\begin{equation}
\forall s \in \mathcal{A},\quad \# \threeind{}{s}{} \mathcal{T} = \frac{1}{\# \mathcal{A}} \# \mathcal{T}.
\end{equation}
Clearly a Young tableau~$\mathcal{T}$ of order~$n$ is then balanced if and only if it has total weight~$0$. Moreover, by construction its total number of boxes is then $n a$, where $a$ is the offset of the diagram underlying~$\mathcal{T}$; so $\mathcal{T}$ is balanced if and only if each symbol occurs exactly $a$ times:
\begin{equation}
\forall s = 1, \ldots, n,\quad \# \threeind{}{s}{} \mathcal{T} = a.
\end{equation}
In particular all balanced Young tableaux have integer offset.
\end{definition}

\begin{remark}~
\label{Young_combinatorial_reformulation}
\begin{hypothenum}
\item \label{itm:full_columns} Note that any semistandard Young tableau~$\mathcal{T}$ of order~$n$ is obtained from a reduced Young tableau with the same $\lie{sl}_n$-shape by prepending some number of columns of height~$n$, and then there is no choice but to fill each of these columns with all the symbols from $1$ to~$n$ in order. These columns have in particular zero total weight, so that they are ``invisible'' when computing total weight or checking dominance. This explains why we no longer require $\mathcal{P}$ to be reduced in Corollary~\ref{characterization_for_sln}.
\item \label{itm:codominance} Every simple root $\alpha \in \Pi$ is of the form $\alpha = \alpha_i = e_i - e_{i+1}$, for some $i = 1, \ldots, n-1$. Then $\mathcal{T}$ is $\alpha_i$-codominant if and only if, for any~$j$, there are at least as many symbols $i+1$ as symbols $i$ among the first $j$ columns of~$\mathcal{T}$:
\begin{equation}
\label{eq:codominance_Ar}
\forall j \geq 0,\quad \# \threeind{[1,j]}{i+1}{} \mathcal{T} \geq \# \threeind{[1,j]}{i}{} \mathcal{T}.
\end{equation}
\end{hypothenum}
\end{remark}

\begin{corollary}
\label{characterization_for_sln}
Let $\Theta \subset \Pi$ be a set of simple roots of $\lie{g} = \lie{sl}_n(\CC)$. Let $\lambda \in P \cap \lie{h}^+$ be a dominant integral weight, and let $\mathcal{P}$ be any Young diagram with $\lie{sl}_n$-shape $\lambda$.

Then $V^{\lie{l}(\Theta)}_\lambda \neq 0$ if and only if $\mathcal{P}$ admits a $\Theta$-codominant balanced semistandard $\{1, \ldots, n\}$-filling.
\end{corollary}

Thus we have reduced the proof of the Main Theorem for $\lie{g} = \lie{sl}_n(\CC)$ to a purely combinatorial problem. We will now classify the diagrams that admit such a filling, first for $\Theta = \Theta(\lie{su}(p,q))$ and then for $\Theta = \Theta(\lie{sl}_m(\HH))$.

\subsection{The case $\lie{g}_\RR = \lie{su}(p, n-p)$}
\label{sec:supq}

For the duration of this subsection, we fix some $p \leq \frac{n}{2}$, and we assume that $\lie{g}_\RR = \lie{su}(p, n-p)$.

Let us then describe $\Theta(\lie{g}_\RR)$. We introduce, for the whole remaining duration of the paper, the following notation shortcuts:
\begin{align}
\label{eq:Theta_interval_definition} \Pi_{[x,y]} &:= \{\alpha_x,\; \alpha_{x+1},\; \ldots,\; \alpha_y\} \subset \Pi = \Pi_{[1,r]}; \\
\label{eq:Theta_odd_definition} \Pi_{\operatorname{odd}} &:= \setsuch{\alpha_i \in \Pi}{i \text{ is odd}},
\end{align}
with the convention $\Pi_{[x,x-1]} = \emptyset$ for all~$x$. (Recall that $r$ represents the rank of~$\lie{g}$; in this section, we have $\lie{g} = \lie{sl}_n(\CC)$ so $r = n-1$.) From \cite{OV90}, Reference Chapter, Table~9, we then get
\[\Theta(\lie{su}(p, n-p)) = \begin{cases}
\Pi_{[p+1,\; n-p-1]} &\text{if } p < \frac{n}{2}; \\
\emptyset &\text{if } p = \frac{n}{2}.
\end{cases}\]

We can in fact reduce ourselves to considering sets $\Theta$ of the form $\Pi_{[1,k-1]}$ (see the final proof in Subsection~\ref{sec:An_conclusion} for details). It remains to prove the following combinatorial result, which is the goal of this subsection.

\begin{proposition}
\label{supq_combinatorics}
Let $\mathcal{P}$ be a Young diagram of order~$n$, and let $k \in \{1, \ldots, n\}$. Then $\mathcal{P}$ has a $\Pi_{[1,k-1]}$-codominant balanced semistandard $\{1, \ldots, n\}$-filling if and only if the offset~$a$ of~$\mathcal{P}$ is integer, and satisfies the inequalities
\begin{equation}
\label{eq:su_pq_diagram_inequalities}
\#_k \mathcal{P} \geq a \geq \#_{n-k+1} \mathcal{P}.
\end{equation}
\end{proposition}

The proof relies on the following ``divide-and-conquer'' strategy, which is a straightforward application of the ``horizontal strip decomposition'' trick (Proposition~\ref{layering_trick}). It will also be useful in the next subsection.
\begin{lemma}
\label{dominance_decomposition}
Let $\Theta \subset \Pi$, and suppose that $k \in \{1, \ldots, n-1\}$ is such that $\alpha_k = e_k - e_{k+1} \not\in \Theta$. Then a Young diagram~$\mathcal{P}$ admits a $\Theta$-codominant, balanced, semistandard $\{1, \ldots, n\}$-filling if and only if there exists a diagram $\mathcal{Q} \subset \mathcal{P}$ such that:
\begin{itemize}
\item the Young diagram $\mathcal{Q}$ admits a $(\Theta \cap \Pi_{[1,k-1]})$-codominant, balanced, semistandard $\{1, \ldots, k\}$-filling;
\item the skew diagram $\mathcal{P}/\mathcal{Q}$ admits a $(\Theta \cap \Pi_{[k+1,n-1]})$-codominant, balanced, semistandard $\{k+1, \ldots, n\}$-filling;
\item the offset of~$\mathcal{Q}$ (as a diagram of order~$k$) coincides with the offset of~$\mathcal{P}$, \ie
\[\# \mathcal{Q} = \frac{k}{n} \# \mathcal{P}.\]
\end{itemize}
\end{lemma}

In our case, $\Theta \cap \Pi_{[1,k-1]}$ is the whole set $\Pi_{[1,k-1]}$ and $\Theta \cap \Pi_{[k+1,n-1]}$ is empty. It remains to characterize Young diagrams~$\mathcal{Q}$ and skew diagrams~$\mathcal{P}/\mathcal{Q}$ having these properties; this is respectively the object of the following two lemmas.

\begin{lemma}
\label{rectangle_characterization}
For $k \geq 0$ and $a \geq 0$, define the tableau $\mathcal{R}_k^a$ that is shaped like a rectangle with $k$ rows of length~$a$, with, for each $s = 1, \ldots, k$, the $s$-th row filled with the symbol~$s$.

Let $k \geq 1$. Then the only $\Pi_{[1,k-1]}$-codominant balanced semistandard Young tableaux of order~$k$ are the rectangles $\mathcal{R}_k^a$, for all (integer) offsets $a \geq 0$.
\end{lemma}
Note that, in the light of Corollary~\ref{characterization_for_sln}, this is equivalent to the (trivial) statement that $V^\lie{g}_\lambda(\lie{g}) \neq 0$ if and only if $\lambda = 0$ (for $\lie{g} = \lie{sl}_k(\CC)$). We nevertheless give the combinatorial proof.
\begin{proof}
We prove this by induction on~$k$. For $k = 1$, this is obvious. Now assume this is true for all values $k' < k$, and let $\mathcal{T}$ be a tableau satisfying these properties. Let $a$ be its offset, so that each symbol occurs exactly $a$ times.

The tableau $\threeind{}{[1,k-1]}{} \mathcal{T}$ is a Young tableau of order~$k-1$, is still balanced, and is $\Pi_{[1,k-2]}$-codominant; so by the induction hypothesis, it must be equal to~$\mathcal{R}_{k-1}^a$. This implies that the tableau $\threeind{[1,a]}{}{} \mathcal{T}$ (obtained by truncating~$\mathcal{T}$ after the $a$-th column) contains exactly $a$ times the symbol~$k-1$. In order to be $\alpha_{k-1}$-codominant, it must also contain at least $a$ times the symbol~$k$. This can only happen if the $k$-th row of $\threeind{[1,a]}{}{} \mathcal{T}$ has length at least~$a$, and is filled with the symbol~$k$. This forces $\threeind{[1,a]}{}{} \mathcal{T} = \mathcal{R}_k^a$, hence $\mathcal{T} = \mathcal{R}_k^a$ as well.
\end{proof}

\begin{lemma}
\label{thin_skew_tableaux_admit_nice_filling}
Let $m \geq 0$, and let $\mathcal{P}/\mathcal{Q}$ be a skew diagram. Then it admits a balanced semistandard $\{1, \ldots, m\}$-filling if and only if it has thickness at most~$m$ and its number of boxes is divisible by~$m$.
\end{lemma}
\begin{proof}
The ``only if'' part is obvious. Conversely, let $\mathcal{P}/\mathcal{Q}$ be a skew diagram of thickness at most~$m$ and containing $m a$ boxes, for some integer $a \geq 0$. By Proposition~\ref{layering_trick}, it suffices to find a Young diagram $\mathcal{P}'$ with the following properties:
\begin{hypothenum}
\item $\mathcal{Q} \subset \mathcal{P}' \subset \mathcal{P}$;
\item $\mathcal{P}/\mathcal{P}'$ contains exactly $a$ boxes, and has thickness at most~$1$;
\item $\mathcal{P}'/\mathcal{Q}$ contains exactly $(m-1)a$ boxes, and has thickness at most~$m-1$.
\end{hypothenum}
We may then conclude by induction on~$m$, filling all boxes of~$\mathcal{P}'/\mathcal{Q}$ with the symbols from $1$ to~$m-1$ in a balanced and semistandard way, and filling the remaining boxes, namely $\mathcal{P}/\mathcal{P}'$, with the symbol~$m$.

Denote by~$X$ (resp.~$Y$) the set of indices~$j$ such that the height of the $j$-th column of the skew diagram $\mathcal{P}/\mathcal{Q}$ is exactly~$m$ (resp. is nonzero). By the pigeonhole principle, we then have
\[\# X \leq a \leq \# Y.\]
We now define $\mathcal{P}'$ by specifying its column heights $\#^j \mathcal{P}'$:
\begin{itemize}
\item whenever $j$ is in~$X$ or is among the largest $(a - \# X)$ values in~$Y \setminus X$, we set $\#^j \mathcal{P}' := \#^j \mathcal{P} - 1$;
\item whenever $j$ is among the remaining values in~$Y \setminus X$ or outside of~$Y$, we set $\#^j \mathcal{P}' := \#^j \mathcal{P}$.
\end{itemize}
By case distinction, it is straightforward to verify that these column heights do indeed define a valid Young diagram, \ie that they form a nonincreasing sequence. As for the properties (i) through~(iii) above, $\mathcal{P}'$ then satisfies them by construction.
\end{proof}

We are now ready to prove the proposition.

\begin{proof}[Proof of Proposition~\ref{supq_combinatorics}]
Let $\mathcal{P}$ be any Young diagram with $\lie{sl}_n$-shape~$\lambda$, and let $a$ be its offset. Plugging Lemmas \ref{rectangle_characterization} and~\ref{thin_skew_tableaux_admit_nice_filling} into Lemma~\ref{dominance_decomposition}, we now see that $\mathcal{P}$ has a filling with the required properties if and only if $a$~is integer and:
\begin{equation}
\label{eq:slaloms_between_rectangles}
\begin{cases}
\text{the rectangular diagram $\square \mathcal{R}_k^a$ is contained in~$\mathcal{P}$;} \\
\text{the skew diagram $\mathcal{P}/\square \mathcal{R}_k^a$ has thickness at most $n-k$.}
\end{cases}
\end{equation}

It remains to check that the condition~\eqref{eq:slaloms_between_rectangles} is equivalent to the inequalities~\eqref{eq:su_pq_diagram_inequalities}, namely $\#_k \mathcal{P} \geq a \geq \#_{n-k+1}\mathcal{P}$. Indeed we have, on the one hand:
\[\square\mathcal{R}_k^a \subset \mathcal{P}
\;\iff\; \#^a \mathcal{P} \geq k
\;\iff\; \#_k \mathcal{P} \geq a,\]
and on the other hand:
\begin{align*}
\forall j,\quad \#^j \mathcal{P} - \#^j \square\mathcal{R}_k^a \leq n-k
&\;\iff\; \forall j > a,\quad \#^j \mathcal{P} \leq n-k \\
&\;\iff\; \#^{a+1}\mathcal{P} \leq n-k \\
&\;\iff\; \#_{n-k+1}\mathcal{P} \leq a. \qedhere
\end{align*}
\end{proof}

\subsection{The case $\lie{sl}_m(\HH)$}
\label{sec:slmH}

Fix some $m \geq 1$. For the duration of this subsection, we assume that $n = 2m$ and that $\lie{g}_\RR = \lie{sl}_m(\HH)$.

From \cite{OV90}, Reference Chapter, Table~9, we then get
\[\Theta(\lie{sl}_m(\HH)) = \Pi_{\operatorname{odd}} = \{e_1 - e_2,\; e_3 - e_4,\; \ldots,\; e_{2m-1} - e_{2m}\}.\]

The Main Theorem for this $\lie{g}_\RR$ then follows, by Corollary~\ref{characterization_for_sln} (see the final proof in Subsection~\ref{sec:An_conclusion} for details), from the following combinatorial result. This subsection is dedicated to proving it.

\begin{proposition}
\label{slmH_combinatorics}
Let $\mathcal{P}$ be a Young diagram of order~$n$. Then it admits a $\Pi_{\operatorname{odd}}$-codominant balanced semistandard $\{1, \ldots, n\}$-filling if and only if its offset~$a$ is integer, and it satisfies the two inequalities
\begin{equation}
\labelWithArgument{eq:slmH_lambda_condition}{$m$}
\begin{cases}
\displaystyle -p_1 + \sum_{i=2}^{m+1} p_i - \sum_{i=m+2}^{2m} p_i \geq 0 \\
\displaystyle -\sum_{i=1}^{m-1} p_i + \sum_{i=m}^{2m-1} p_i - p_{2m} \leq 0,
\end{cases}
\end{equation}
with the notation shortcut $p_i := \#_i \mathcal{P}$.
\end{proposition}

The proof, like the proof of Proposition~\ref{supq_combinatorics}, relies on Lemma~\ref{dominance_decomposition}. However now the situation is more complex: while for $\Theta = \Pi_{[1,k-1]}$ we had a single ``cutting point'' (namely~$k$), here we will ``cut'' at all the even indices at the same time. More rigorously, we will use $2m-2$ as the cutting point, and then proceed by induction on~$m$. Overall, the proof is much more technical than in the previous subsection.

The proof relies on two big lemmas:
\begin{itemize}
\item Subsubsection~\ref{sec:skew_thickness_two} is dedicated to proving Lemma~\ref{conditions_for_nice_filling_of_skew_diagram}, which, roughly, gives a condition for the existence of a suitable filling of the ``bottom'' skew tableau, namely $\threeind{}{\{2m-1,2m\}}{}\mathcal{T}$.
\item Subsubsection~\ref{sec:induction_step} is dedicated to proving Lemma~\ref{induction_step}, which, roughly, deduces the result from the induction hypothesis and from this characterization.
\end{itemize}

For a more detailed explanation of how these lemmas fit together, see the schematic given in the final proof (Subsubsection~\ref{sec:slmH_conclusion}).

\subsubsection{Skew tableaux of thickness 2}
\label{sec:skew_thickness_two}

We now give the criterion for the existence of an $\alpha_{2m-1}$-codominant balanced semistandard $\{2m-1, 2m\}$-filling of a skew tableau. Clearly we lose no generality by considering the alphabet $\{1, 2\}$ instead. In order to give this criterion, we first need a definition.

\begin{definition}
\label{bridge_definition}
Let $\mathcal{P}$ be a Young diagram of order~$n$, and let $\mathcal{Q}$ be a Young diagram contained in~$\mathcal{P}$. The \emph{bridge at height~$i$} in $\mathcal{P}/\mathcal{Q}$ is the rectangle formed by all columns~$j$ such that $\#^j \mathcal{Q} = i-1$ and $\#^j \mathcal{P} = i$ (see Figure~\ref{fig:bridges_example}). For each $i = 1, \ldots, n$, we then denote by $b_i$ the length of the bridge at height~$i$.
\end{definition}

\begin{figure}
\[\begin{tikzpicture}[x=13pt,y=13pt] % 13pt is the default Yboxdim
\Yfillcolor{black!20}
\tgyoung(0cm,0cm,:::::::::::::_2,,:::::::_3,::::_2,::_1);
\Yfillopacity{0}
\Ylinecolor{black!30}
\tgyoung(0cm,0cm,;;;;;;;;;;;;;;;,;;;;;;;;;;;;;,;;;;;;;;;;,;;;;;;,;;;;);
\Ylinecolor{black}
\tgyoung(0cm,0cm,::::::::::::;111,::::::::::::;2,:::::::;112,:::;122,1122,22);
\draw[thick, <->](13,-0.5)--(15,-0.5) node[midway,below] {\footnotesize $b_1 = 2$};
\node[below] at (12,-1) {\footnotesize $b_2 = 0$};
\draw[thick, <->](7,-2.5)--(10,-2.5) node[midway,below] {\footnotesize $b_3 = 3$};
\draw[thick, <->](4,-3.5)--(6,-3.5) node[midway,below] {\footnotesize $b_4 = 2$};
\draw[thick, <->](2,-4.5)--(3,-4.5) node[pos=0,below right] {\footnotesize $b_5 = 1$};
\node[below] at (0,-5) {\footnotesize $b_6 = 0$};
\end{tikzpicture}\]
\caption{\label{fig:bridges_example}Example of a skew diagram of thickness~$2$, with all bridges shaded and the bridge lengths $b_i$ written (see Definition~\ref{bridge_definition}). This diagram is also given with a $\{1, 2\}$-filling that satisfies Lemma~\ref{conditions_for_nice_filling_of_skew_diagram}.}
\end{figure}

\begin{lemma}
\label{conditions_for_nice_filling_of_skew_diagram}
%\begin{itemize}
Let $\mathcal{P}$ be a Young diagram of order~$n$, and let $\mathcal{Q}$ be a Young diagram contained in~$\mathcal{P}$. Then the skew diagram $\mathcal{P}/\mathcal{Q}$ admits an $\alpha_1$-codominant balanced semistandard $\{1, 2\}$-filling if and only if:
\begin{hypothenum}
\item \label{itm:balanced_implies_even} its total number of boxes $\# \mathcal{P}/\mathcal{Q}$ is even;
\item \label{itm:thickness_2} it has thickness at most~$2$;
\item \label{itm:no_majority_bridge} if we count the total number of boxes in all the bridges, no single bridge contains a majority of them:
\begin{equation}
\label{eq:b_no_majority_bridge}
\forall i = 1, \ldots, n,\quad b_i \leq \frac{1}{2} \sum_{j=1}^n b_j.
\end{equation}
\end{hypothenum}
%\item Formally, conditions \ref{itm:thickness_2} and~\ref{itm:no_majority_bridge} can be respectively expressed as the inequalities %TODO Need to somehow incorporate the condition that Q is contained in P.
%\begin{align}
%\label{eq:thickness_2}
%\forall i = 1, \ldots, n-2,\quad& \mathcal{P}_i \geq \mathcal{Q}_i \geq \mathcal{P}_{i+2}; \\
%\label{eq:no_majority_bridge}
%\forall i = 1, \ldots, n,\quad& \left( \sum_{\substack{j = 1 \\ j \;\neq\; i, i+1}}^n t_j \right) - \hat{t}_i + \hat{t}_{i+1} \;\geq\; 0,
%\end{align}
%where we set:
%\begin{align*}
%t_1 &:= \mathcal{P}_1 &\text{and}\qquad \hat{t}_1 &:= \mathcal{P}_1; \\
%\forall j = 2, \ldots, n,\quad
%t_j &:= - |\mathcal{Q}_{j-1} - \mathcal{P}_j| &\text{and}\qquad
%\hat{t}_j &:= \mathcal{Q}_{j-1} + \mathcal{P}_j; \\
%t_n &:= -\mathcal{P}_n &\text{and}\qquad \hat{t}_n &:= \mathcal{P}_n.
%\end{align*}

%\[
%t_j := \begin{cases}
%\mathcal{P}_1 &\text{ if } j = 1; \\
%- |\mu_{j-1} - \mathcal{P}_j| &\text{ if } 2 \leq j \leq n; \\
%-\mathcal{P}_n &\text{ if } j = n.
%\end{cases}
%\]
%and
%\[
%\hat{t}_j := \begin{cases}
%\mathcal{P}_1 &\text{ if } j = 1; \\
%\mathcal{Q}_{j-1} + \mathcal{P}_j &\text{ if } 2 \leq j \leq n; \\
%\mathcal{P}_n &\text{ if } j = n.
%\end{cases}
%\]
%\end{itemize}
\end{lemma}

\begin{proof}
Let $\mathcal{T}$ be any $\{1, 2\}$-filling of the skew-tableau $\mathcal{P}/\mathcal{Q}$. First of all, note that this filling is semistandard if and only if it satisfies the following properties:
\begin{itemize}
\item No columns of height more than $2$ exist (this is condition~\ref{itm:thickness_2}).
\item Each column of height~$2$ is filled with the symbols $1$ and~$2$ in that order.
\item For every $i$, there exists a number $c_i$ such that
\begin{equation}
\label{eq:c_bounds}
0 \leq c_i \leq b_i,
\end{equation}
with the $i$-th bridge of~$\mathcal{T}$ having the first $c_i$ boxes filled with~$1$ and the last $b_i - c_i$ boxes filled with~$2$.
\end{itemize}

%Suppose that $\mathcal{T}$ is an $\alpha_1$-codominant balanced semistandard $\{1, 2\}$-filling of~$\mathcal{P}/\mathcal{Q}$. Obviously this implies that $\# \mathcal{P}/\mathcal{Q}$ is even. Moreover, clearly columns of height more than~$2$ can not occur in~$\mathcal{P}/\mathcal{Q}$; and each column of height~$2$ must necessarily be filled with the symbols $1$ and~$2$ in that order (by semistandardness).
%
%Once these three conditions are met, the filling is semistandard if and only if, for every $i$, there exists a number $c_i$ such that
%\begin{equation}
%\label{eq:c_bounds}
%0 \leq c_i \leq b_i,
%\end{equation}
%with the $i$-th bridge of~$\mathcal{T}$ having the first $c_i$ boxes filled with~$1$ and the last $b_i - c_i$ boxes filled with~$2$.

Assume now that $\mathcal{T}$ is semistandard. Recall that $\alpha_1$-codominance (resp. balancedness) of~$\mathcal{T}$ means that the difference $\#\threeind{[1,j]}{2}{}\mathcal{T} - \#\threeind{[1,j]}{1}{}\mathcal{T}$ between the number of $2$'s and the number of $1$'s in the first $j$ columns is nonnegative for every $j = 1, \ldots, \#_1 \mathcal{P}$ (resp. is zero for $j = \#_1 \mathcal{P}$). Clearly columns of heights $0$ and~$2$ make no contribution to this difference, so it suffices to focus on the bridges. Within the bridge at height~$i$, this difference attains its minimum at the $c_i$-th column. Hence a semistandard filling~$\mathcal{T}$ is $\alpha_1$-codominant if and only if it satisfies
\begin{equation}
\label{eq:c_dominance_condition}
\forall i = 1, \ldots, n,\quad \sum_{j=1}^i c_j \leq \frac{1}{2} \left( c_i + \sum_{j=1}^{i-1} b_j \right)
\end{equation}
and balanced if and only if it satisfies
\begin{equation}
\label{eq:c_balancedness_condition}
\sum_{i=1}^n c_i = \frac{1}{2} \sum_{i=1}^n b_i.
\end{equation}

Finally, observe that, when condition~\ref{itm:thickness_2} holds, the total number of boxes in~$\mathcal{P}/\mathcal{Q}$ has the same parity as the sum $\sum_{i=1}^n b_i$: indeed, the difference between these two numbers simply counts all the boxes in columns of height~$2$.

The conclusion now follows from the following lemma.
\end{proof}

\begin{lemma}
Given a tuple of integers~$(b_1, \ldots, b_n)$, there exists a tuple of integers $(c_1, \ldots, c_n)$ satisfying conditions \eqref{eq:c_bounds}, \eqref{eq:c_dominance_condition} and~\eqref{eq:c_balancedness_condition} if and only if the $b_i$ have even sum and satisfy the system~\eqref{eq:b_no_majority_bridge}.
\end{lemma}
\begin{proof}
Suppose first that such a tuple $(c_1, \ldots, c_n)$ exists. Then \eqref{eq:c_balancedness_condition} directly implies that $\sum_{i=1}^n b_i$ is even. Furthermore, subtracting \eqref{eq:c_dominance_condition} from \eqref{eq:c_balancedness_condition}, we obtain, for all $i = 1, \ldots, n$:
\[
\sum_{j=i+1}^n c_j \geq \frac{1}{2} \left( c'_i + \sum_{j=i+1}^n b_j \right),
\]
where, for all $i$, we set $c'_i := b_i - c_i$. Subtracting both sides from twice the right-hand side, we see that the tuple $(c'_1, \ldots, c'_n)$ then satisfies a condition similar to~\eqref{eq:c_dominance_condition}, but with the order of the bridges reversed:
\begin{equation}
\label{eq:c_prime_dominance_condition}
\forall i = 1, \ldots, n,\quad \sum_{j=i}^n c'_j \leq \frac{1}{2} \left( c'_i + \sum_{j=i+1}^n b_j \right).
\end{equation}
Finally, by adding together \eqref{eq:c_dominance_condition} and \eqref{eq:c_prime_dominance_condition}, we obtain, for all $i = 1, \ldots, n$:
\begin{equation}
\label{eq:b_no_majority_bridge_developed}
\left( \sum_{j=1}^{i-1} c_j \right) + b_i + \left( \sum_{j=i+1}^n c'_j \right) \;\leq\; \frac{1}{2} \sum_{j=1}^n b_j.
\end{equation}
Since the left-hand side is not less than~$b_i$ (the other terms are all nonnegative), \eqref{eq:b_no_majority_bridge} follows.

Conversely, suppose that the tuple $(b_1, \ldots, b_n)$ has an even sum, that we shall denote by~$b$, and satisfies the system~\eqref{eq:b_no_majority_bridge}. Then consider the tuple $(c_1, \ldots, c_n)$ defined as follows:
\begin{equation}
c_i := \begin{cases}
0 &\text{ if } \sum_{j=1}^i b_j \leq \frac{1}{2}b; \\
b_i &\text{ if } \sum_{j=1}^{i-1} b_j \geq \frac{1}{2}b; \\
\frac{1}{2}b - \sum_{j=i+1}^n b_j &\text{ if } \sum_{j=1}^{i-1} b_j < \frac{1}{2}b < \sum_{j=1}^i b_j.
\end{cases}
\end{equation}
Informally, this corresponds to filling with $2$'s the leftmost $\frac{1}{2}b$ of all the boxes contained in bridges, and with $1$'s the rightmost $\frac{1}{2}b$ of them; and then, if the cut-off point happens to be inside a bridge (which would break row-standardness), we swap the $1$'s and the $2$'s within that bridge (see Figure~\ref{fig:bridges_example} for an example). Clearly this tuple satisfies~\eqref{eq:c_bounds} and~\eqref{eq:c_balancedness_condition}. Moreover:
\begin{itemize}
\item for all $i$ such that $\sum_{j=1}^i b_j \leq \frac{1}{2}b$, clearly $(c_1, \ldots, c_n)$ satisfies the condition~\eqref{eq:c_dominance_condition};
\item for all $i$ such that $\sum_{j=1}^{i-1} b_j \geq \frac{1}{2}b$, clearly $(c_1, \ldots, c_n)$ satisfies the condition~\eqref{eq:c_prime_dominance_condition}, which (given \eqref{eq:c_bounds} and~\eqref{eq:c_balancedness_condition}) is equivalent to~\eqref{eq:c_dominance_condition};
\item for the index $i$ such that $\sum_{j=1}^{i-1} b_j < \frac{1}{2}b < \sum_{j=1}^i b_j$ (if it exists), the condition~\eqref{eq:b_no_majority_bridge} implies the condition~\eqref{eq:b_no_majority_bridge_developed}, since all the additional terms on the left-hand side vanish. Now \eqref{eq:b_no_majority_bridge_developed}, being the sum of the two equivalent inequalities~\eqref{eq:c_dominance_condition} and~\eqref{eq:c_prime_dominance_condition}, is equivalent to both. \qedhere
\end{itemize}
%\item Checking that \ref{itm:thickness_2} is equivalent to~\eqref{eq:thickness_2} is straightforward. As for~\ref{itm:no_majority_bridge}, we start by observing that we have: %TODO "By gazing long enough at Figure ..., the reader can convince themselves that..."?
%\[
%\forall i = 1, \ldots, n,\quad b_i = \min(\mathcal{Q}_{i-1}, \mathcal{P}_i) - \max(\mathcal{Q}_i, \mathcal{P}_{i+1}),
%\]
%with the convention that $\mathcal{Q}_0 = +\infty$ and $\mathcal{Q}_n = -\infty$.
%%TODO Finish this proof... but do I really need the form of the inequality with the absolute values? maybe the form with max's and min's is sufficient?
%\qedhere
%\end{itemize}
\end{proof}

\subsubsection{The induction step}
\label{sec:induction_step}

This subsubsection is dedicated to proving the following result, which, when combined with Lemma~\ref{conditions_for_nice_filling_of_skew_diagram} from the previous subsubsection, provides the induction step for the proof of Proposition~\ref{slmH_combinatorics}. More precisely, it provides the equivalence (D) in the outline given in the final proof (Subsubsection~\ref{sec:slmH_conclusion}).

In this whole subsubsection, we assume that $m$ is an integer greater or equal than~$2$, and $n = 2m$.

\begin{lemma}
\label{induction_step}
Let $\mathcal{P}$ be a Young diagram of order $n = 2m \geq 4$. Then $\mathcal{P}$ satisfies the inequalities~\eqrefWithArgument{eq:slmH_lambda_condition}{$m$} and has integer offset (\ie $\# \mathcal{P}$ is divisible by~$n$) if and only if there exists a Young diagram $\mathcal{Q}$ of order~$n-2$ with the following properties:
\begin{hypothenum}
\item \label{itm:slmH_mu_condition} $\mathcal{Q}$ satsifies the system~\eqrefWithArgument{eq:slmH_lambda_condition}{$m-1$}, \ie the system \eqrefWithArgument{eq:slmH_lambda_condition}{$m$} with $m$ replaced by~$m-1$, which is explicitly:
\begin{equation}
\begin{cases}
\displaystyle -q_1 + \sum_{i=2}^{m} q_i - \sum_{i=m+1}^{2m-2} q_i \geq 0 \\
\displaystyle -\sum_{i=1}^{m-2} q_i + \sum_{i=m-1}^{2m-3} q_i - q_{2m-2} \leq 0,
\end{cases}
\tag{\ref{eq:slmH_lambda_condition}.$m-1$}
\end{equation}
with the notation shortcut $q_i := \#_i \mathcal{Q}$;
\item \label{itm:skew_diagram_even} the difference $\# \mathcal{P} - \# \mathcal{Q}$ is even;
\item \label{itm:thickness_2_bis} $\mathcal{Q} \subset \mathcal{P}$, and the skew diagram $\mathcal{P}/\mathcal{Q}$ has thickness at most~$2$;
\item \label{itm:no_majority_bridge_bis} no bridge in $\mathcal{P}/\mathcal{Q}$ contains more boxes than all the remaining bridges combined;
\item \label{itm:proportional_size} $\# \mathcal{Q} = \frac{n-2}{n} \# \mathcal{P}$.
\end{hypothenum}
\end{lemma}

The proof will require some preliminary work. Proving the ``if'' part will simply be a matter of rewriting the conditions on $\mathcal{P}$ and~$\mathcal{Q}$ as a system of inequalities, and then suitably combining some well-chosen inequalities from this system; the main difficulty lies in proving the ``only if'' part. Very roughly, the idea is to find such a $\mathcal{Q}$ for a few basic values of~$\mathcal{P}$, and then take advantage of additivity. However it will not quite work like this; a slight adaptation will be needed. We will present a more detailed outline after introducing a few basic definitions and notations.

\begin{definition}
An element of a commutative monoid is \emph{primitive} if it is not the sum of two nonzero elements of the monoid. Clearly, a subset of a monoid is a generating set if and only if it contains all the nonzero primitive elements. The set of nonzero primitive elements is called the \emph{basis} of the monoid.
\end{definition}

\begin{definition}
\label{Young_monoid_definition}
For each integer $x \geq 0$, we denote by~$\mathscr{M}^{(x)}$ the monoid of all the Young diagrams of order~$x$, with the addition operation defined by adding the numbers of boxes row-wise:
\begin{equation}
\forall i = 1, \ldots, x,\quad \#_i(\mathcal{P} + \mathcal{Q}) := \#_i \mathcal{P} + \#_i \mathcal{Q}.
\end{equation}
When $x = n$, we will usually omit the index, \ie we set $\mathscr{M} := \mathscr{M}^{(n)} = \mathscr{M}^{(2m)}$. We denote by~$\mathscr{M}_{\eqref{eq:slmH_lambda_condition}}$ the submonoid of~$\mathscr{M}$ determined by the system of linear inequalities~\eqrefWithArgument{eq:slmH_lambda_condition}{$m$}. For each integer $k > 0$, we denote by~$\mathscr{M}_{k | \#}$ the submonoid of diagrams whose number of boxes is divisible by~$k$:
\begin{equation}
\mathscr{M}_{k | \#} := \setsuch{\mathcal{P} \in \mathscr{M}}{\# \mathcal{P} \in k\ZZ}.
\end{equation}
For each $i = 0, \ldots, n$, we define $\mathcal{C}_i \in \mathscr{M}$ to be the diagram consisting of a single column of height~$i$. Thus $\mathcal{C}_0 = 0$ is the empty diagram, and $(\mathcal{C}_1, \ldots, \mathcal{C}_n)$ is a basis of the additive monoid~$\mathscr{M}$. (Note that the $\lie{sl}_n$-shape of~$\mathcal{C}_i$ is precisely $\varpi_i$ for $i = 1, \ldots, n-1$, and is zero for $i = n$.)
\end{definition}

In this terminology, in order to prove the ``only if'' part of Lemma~\ref{induction_step}, we need to construct, for every diagram~$\mathcal{P}$ lying in the monoid $\mathscr{M}_{\eqref{eq:slmH_lambda_condition}} \cap \mathscr{M}_{n | \#}$, a diagram~$\mathcal{Q}$ such that the pair $(\mathcal{P}, \mathcal{Q})$ satisfies the conditions \ref{itm:slmH_mu_condition} through~\ref{itm:proportional_size}. We shall soon see (in Lemma~\ref{good_pairs_closed_under_addition}, combined with the remark that follows it) that the set of such pairs is closed under addition; so it ``suffices'' to construct such diagrams~$\mathcal{Q}$ for the primitive elements~$\mathcal{P}$ of the monoid $\mathscr{M}_{\eqref{eq:slmH_lambda_condition}} \cap \mathscr{M}_{n | \#}$.

Unfortunately, the basis of the monoid $\mathscr{M}_{\eqref{eq:slmH_lambda_condition}} \cap \mathscr{M}_{n | \#}$ admits no simple description (for general~$n$). To bypass this difficulty, we extend our field of consideration to the (larger) monoid $\mathscr{M}_{\eqref{eq:slmH_lambda_condition}} \cap \mathscr{M}_{2 | \#}$, whose basis, on the contrary, can be readily described. The price to pay is that condition~\ref{itm:proportional_size} becomes impossible to satisfy: it could force $\mathcal{Q}$ to have a non-integer number of boxes.

We solve this difficulty by replacing the equation~\ref{itm:proportional_size} by a pair of inequalities: we construct, for every such~$\mathcal{P}$, two different diagrams $\mathcal{Q}^\pm$, that satisfy conditions \ref{itm:slmH_mu_condition} through~\ref{itm:no_majority_bridge_bis}, but whose (integer, and even) numbers of boxes bound $\frac{n-2}{n} \# \mathcal{P}$ from above and from below. This is the content of Lemma~\ref{horrible_technical_part} below.

\begin{lemma}
\label{good_pairs_closed_under_addition}
The set of pairs $(\mathcal{P}, \mathcal{Q})$ satisfying the conditions \ref{itm:slmH_mu_condition} through~\ref{itm:no_majority_bridge_bis} is a submonoid of $\mathscr{M}^{(n)} \oplus \mathscr{M}^{(n-2)}$.
\end{lemma}

Note that this is (obviously) also true for the condition~\ref{itm:proportional_size}; the discussion preceding this lemma explains why we did not include it.

\begin{proof}
Condition~\ref{itm:skew_diagram_even} is obviously stabe under addition. So is condition~\ref{itm:slmH_mu_condition}, as it is a system of (homogeneous) linear inequalities, \ie a system of the form
\begin{equation}
\label{eq:linear_ineq_syst}
\forall i \in I,\quad \phi_i(\mathcal{P}, \mathcal{Q}) \geq 0,
\end{equation}
where $(\phi_i)_{i \in I}$ is some family of linear forms, \ie linear maps from $\mathscr{M}^{(n)} \oplus \mathscr{M}^{(n-2)}$ to~$\RR$. Condition~\ref{itm:thickness_2_bis} is also of this form: indeed, it is equivalent to the system of inequalities
\[\forall i = 1, \ldots, n-2,\quad p_i \geq q_i \geq p_{i+2},\]
which can also be put into form~\eqref{eq:linear_ineq_syst}.

The slightly nontrivial part is the additivity of condition~\ref{itm:no_majority_bridge_bis}. Let us show that it can, in fact, also be put into form~\eqref{eq:linear_ineq_syst}. Indeed, it is given by the formula~\eqref{eq:b_no_majority_bridge}, that we may also rewrite as
\begin{equation}
\label{eq:b_no_bridge_larger_than_others}
\forall i = 1, \ldots, n-1,\quad \sum_{j=1}^{n-1} (-1)^{\delta_{ij}}b_j \geq 0,
\end{equation}
where $\delta_{ij}$ is the Kronecker delta symbol. We replaced here $n$ by~$n-1$, because the condition that $\mathcal{Q}$ has order $n-2$ forces $b_n = 0$. Furthermore, by using the identity $\#_i \mathcal{P} \geq j \iff \#^j \mathcal{P} \geq i$ (or just by gazing long enough at Figure~\ref{fig:bridges_example}), we can see that the $i$-th bridge length $b_i$ is given by the formula
\begin{equation}
\label{eq:bridge_length_formula}
\forall i = 1, \ldots, n-1,\quad b_i = \min(q_{i-1}, p_i) - \max(q_i, p_{i+1}),
\end{equation}
with the convention that $q_0 = +\infty$ and $q_{n-1} = -\infty$.

Plugging \eqref{eq:bridge_length_formula} into~\eqref{eq:b_no_bridge_larger_than_others}, and rearranging the sum so as to group the terms involving the same row-lengths of $\mathcal{P}$ and~$\mathcal{Q}$, we see that condition~\ref{itm:no_majority_bridge_bis} is equivalent to
\begin{gather}
\begin{flalign}
& \forall i = 1, \ldots, n-1, &
\end{flalign} \\
\begin{multlined}[\displaywidth]
\qquad (-1)^{\delta_{i1}} p_1 \;\;+\;\; \sum_{\smash{j=1}}^{n-1} \left( -(-1)^{\delta_{i,j-1}} \max(q_{j-1}, p_j)\right. +\\ \left. +(-1)^{\delta_{ij}} \min(q_{j-1}, p_j) \right) \;\;-\;\;  (-1)^{\delta_{i,n-1}}p_n \geq 0. \nonumber
\end{multlined}
\end{gather}
Using the identities $\max(x, y) + \min(x, y) = x+y$ and $\max(x, y) - \min(x, y) = |x-y|$, we can once again rephrase condition~\ref{itm:no_majority_bridge_bis} as
\begin{align}
\label{eq:no_majority_bridge}
\forall i = 1, \ldots, n &- 1, \\
&\Bigg( \sum_{\substack{j = 1 \\ j \;\neq\; i, i+1}}^n T_j \Bigg) - \hat{T}_i + \hat{T}_{i+1} \;\geq\; 0,
\tag{\theequation.$i$}
\end{align}
where we set:
\begin{align*}
T_1 &:= p_1 &\text{and}\qquad \hat{T}_1 &:= p_1; \\
\forall j = 2, \ldots, n-1,\quad
T_j &:= - |q_{j-1} - p_j| &\text{and}\qquad
\hat{T}_j &:= q_{j-1} + p_j; \\
T_n &:= -p_n &\text{and}\qquad \hat{T}_n &:= p_n.
\end{align*}
Each of these inequalities~\eqrefWithArgument{eq:no_majority_bridge}{$i$} is \emph{a priori} nonlinear, of the form
\[
\phi_0(\mathcal{P}, \mathcal{Q}) - \sum_{j=1}^N |\phi_j(\mathcal{P}, \mathcal{Q})| \geq 0,
\]
where $\phi_0, \phi_1, \ldots, \phi_N$ are some linear forms (depending on~$i$). But any such inequality can be rewritten as a system of $2^N$ linear inequalities: indeed it is equivalent to
\begin{equation}
\label{eq:absolute_values_expansion}
\forall (\sigma_1, \ldots, \sigma_N) \in \{\pm 1\}^N,\quad \phi_0(\mathcal{P}, \mathcal{Q}) - \sum_{j=1}^N \sigma_j \phi_j(\mathcal{P}, \mathcal{Q}) \geq 0. \qedhere
\end{equation}
\end{proof}

\begin{lemma}
\label{horrible_technical_part}
Let $\mathscr{B}'$ be the basis of the monoid $\mathscr{M}_{\eqref{eq:slmH_lambda_condition}} \cap \mathscr{M}_{2 | \#}$. Then for each Young diagram $\mathcal{P} \in \mathscr{B}'$, there exist two Young diagrams $\mathcal{Q}^+(\mathcal{P})$ and~$\mathcal{Q}^-(\mathcal{P})$ of order $n-2$ with the following properties:
\begin{itemize}
\item Both pairs $(\mathcal{P}, \mathcal{Q}^+)$ and $(\mathcal{P}, \mathcal{Q}^-)$ satisfy conditions \ref{itm:slmH_mu_condition} through~\ref{itm:no_majority_bridge_bis} from Lemma~\ref{induction_step}.
\item The total number of boxes in~$\mathcal{Q}^-$ is the largest even number not exceeding $\frac{n-2}{n} \# \mathcal{P}$, and symmetrically for $\mathcal{Q}^+$:
\begin{equation}
\label{eq:Q_pm_size}
\begin{cases}
\# \mathcal{Q}^- = 2 \left\lfloor \frac{1}{2} \frac{n-2}{n} \# \mathcal{P} \right\rfloor; \\
\# \mathcal{Q}^+ = 2 \left\lceil \frac{1}{2} \frac{n-2}{n} \# \mathcal{P} \right\rceil.
\end{cases}
\end{equation}
\end{itemize}
\end{lemma}

\begin{proof}
Let $\mathcal{P}$ be an element of~$\mathscr{M}$, and let $x_1, \ldots, x_n$ be its coordinates in the basis $(\mathcal{C}_1, \ldots, \mathcal{C}_n)$ introduced in Definition~\ref{Young_monoid_definition}. Then the inequalities~\eqrefWithArgument{eq:slmH_lambda_condition}{$m$} can be rewritten in terms of the $x_i$ as
\begin{equation}
\labelWithArgument{eq:slmH_lambda_condition_rewritten}{$m$}
\begin{cases}
\displaystyle \sum_{i=1}^{n-1} \min(i-2,\; n-i) x_i \geq 0; \\
\displaystyle \sum_{i=1}^{n-1} \min(i,\; n-i-2) x_i \geq 0
\end{cases}
\end{equation}
(recall that $n = 2m$). For example, \eqrefWithArgument{eq:slmH_lambda_condition_rewritten}{$4$} is
\begin{equation}
\tag{\ref{eq:slmH_lambda_condition_rewritten}.$4$}
\systeme{
  -x_1 + x_3 + 2x_4 + 3x_5 +2x_6 + x_7 \geq 0,
  x_1 + 2x_2 + 3x_3 + 2x_4 + x_5 - x_7 \geq 0.
}
\end{equation}

%Note that in the first (resp. second) inequality of~\eqrefWithArgument{eq:slmH_lambda_condition_rewritten}{$m$}, the coefficient in front of $x_i$:
%\begin{itemize}
%\item is equal to~$-1$ for $i=1$ (resp. for $i=n-1$);
%\item vanishes for $i=2$ (resp. for $i=n-2$), and for $i = n$;
%\item is positive for the remaining values of~$i$.
%\end{itemize}
It is then easy to see that the basis~$\mathscr{B}$ of the monoid~$\mathscr{M}_{\eqref{eq:slmH_lambda_condition}}$ is equal to
\begin{align}
\mathscr{B} = &\setsuch{\vphantom{\Big(}\mathcal{C}_i}{2 \leq i \leq n-2} \cup \{\mathcal{C}_n\} \cup {} \nonumber \\
              &\cup \setsuch{\vphantom{\Big(}\mathcal{C}_i + a\mathcal{C}_1}{0 < a \leq \min(i-2,\; n-i)} \cup \nonumber \\
              &\cup \setsuch{\vphantom{\Big(}\mathcal{C}_i + a\mathcal{C}_{n-1}}{0 < a \leq \min(i,\; n-i-2)}.
\end{align}

Now consider a diagram $\mathcal{P} \in \mathscr{M}_{\eqref{eq:slmH_lambda_condition}} \cap \mathscr{M}_{2|\#}$. Its decomposition as a sum of elements of~$\mathscr{B}$ will then involve an even number of odd-sized diagrams (where by ``size'' we mean the number of boxes). Denoting by $\mathscr{B}_{\text{even}}$ (resp.~$\mathscr{B}_{\text{odd}}$) the subset of~$\mathscr{B}$ comprising the diagrams of even (resp. odd) size, we obtain that the set
\begin{equation*}
\mathscr{B}_{\text{even}} \cup \Big( \mathscr{B}_{\text{odd}} + \mathscr{B}_{\text{odd}} \Big)
\end{equation*}
(where the ``$+$'' sign denotes the Minkowski, or elementwise, sum) generates the monoid $\mathscr{M}_{\eqref{eq:slmH_lambda_condition}} \cap \mathscr{M}_{2|\#}$.

It remains to eliminate the non-primitive elements. Clearly all elements of $\mathscr{B}_{\text{even}}$, being already primitive in~$\mathscr{M}_{\eqref{eq:slmH_lambda_condition}}$, are still primitive in~$\mathscr{M}_{\eqref{eq:slmH_lambda_condition}} \cap \mathscr{M}_{2|\#}$. Now let $\mathcal{P}$ be some element of $\Big( \mathscr{B}_{\text{odd}} + \mathscr{B}_{\text{odd}} \Big)$, \ie a sum of two elements of~$\mathscr{B}_{\text{odd}}$. Then necessarily it is of the form
\[
\mathcal{P} = \mathcal{C}_i + \mathcal{C}_j + a \mathcal{C}_1 + b \mathcal{C}_{n-1}
\]
where $a, b \geq 0$, and $i$ and $j$ satisfy $2 \leq i \leq j \leq n-2$. We claim that, if either $a$ or $b$ is nonzero, then this element is not primitive. Indeed:
\begin{itemize}
\item Suppose that $a > 0$ and $b > 0$. Then we have, possibly up to exchanging $i$ and~$j$:
\[\mathcal{C}_i + a \mathcal{C}_1 \in \mathscr{B}_{\text{odd}}
\quad\text{and}\quad
\mathcal{C}_j + b \mathcal{C}_{n-1} \in \mathscr{B}_{\text{odd}}.\]
From this, we deduce that, in the decomposition
\[
\mathcal{P} = \Big( \mathcal{C}_i + (a-1) \mathcal{C}_1 \Big) + \Big( \mathcal{C}_j + (b-1) \mathcal{C}_{n-1} \Big) + \Big( \mathcal{C}_1 + \mathcal{C}_{n-1} \Big),
\]
all three terms are still in~$\mathscr{B}$ (using the fact that $n \geq 4$), but have even number of boxes; \ie they are in $\mathscr{B}_{\text{even}}$. So $\mathcal{P}$ is not primitive.
\item Suppose that $a > 0$ and $b = 0$. Then necessarily the decomposition of~$\mathcal{P}$ as a sum of two elements of $\mathscr{B}_{\text{odd}}$ is of the form
\[\mathcal{P} = \Big( \mathcal{C}_i + a_i \mathcal{C}_1 \Big) + \Big( \mathcal{C}_j + a_j \mathcal{C}_1 \Big).\]
In particular this means that both of the sums $a_i + i$ and $a_j + j$ are odd. On the other hand, from the assumption $a_i + a_j = a > 0$ we get that at least one of $a_i$ or $a_j$ must be positive.

Exchanging if necessary $i$ and~$j$, assume that $a_i > 0$. Then we can rewrite $\mathcal{P}$ as
\[\mathcal{P} = \Big( \mathcal{C}_i + (a_i - 1) \mathcal{C}_1 \Big) + \Big( \mathcal{C}_j + (a_j + 1) \mathcal{C}_1 \Big).\]
In this new decomposition, clearly both summands have even number of boxes; let us justify that they are both still in~$\mathscr{B}$. For the first summand, this is obvious. As for the second summand, it suffices to see that, since $\min(j-2,n-j)$ always has the same parity as~$j$ but $a_j$ has opposite parity, the inequality $a_j \leq \min(j-2,n-j)$ is in fact necessarily strict. We conclude that $\mathcal{P}$ is not primitive.
\item The case $a = 0$ and $b > 0$ is analogous.
\end{itemize}
On the other hand, it is easy to see that the remaining elements of $\Big( \mathscr{B}_{\text{odd}} + \mathscr{B}_{\text{odd}} \Big)$ are primitive. It follows that
\begin{equation}
\mathscr{B}' = \mathscr{B}_{\text{even}} \cup \setsuch{\mathcal{C}_i + \mathcal{C}_j}{2 \leq i \leq j \leq n-2 \text{ with } i, j \text{ odd}}.
\end{equation}

We now define, for each element $\mathcal{P} \in \mathscr{B}'$, two diagrams $\mathcal{Q}^-(\mathcal{P})$ and $\mathcal{Q}^+(\mathcal{P})$ as given in Table~\ref{tab:table_of_Ps_and_Qs}. It remains only to check that, for each $\mathcal{P}$, both $\mathcal{Q}^-(\mathcal{P})$ and $\mathcal{Q}^+(\mathcal{P})$ satisfy all of the required properties.

\begin{table}
  \caption{\label{tab:table_of_Ps_and_Qs} Table listing the Young diagrams $\mathcal{Q}^-(\mathcal{P})$ and $\mathcal{Q}^+(\mathcal{P})$ claimed to exist in Lemma~\ref{horrible_technical_part}, for each primitive diagram~$\mathcal{P}$ lying in the monoid $\mathscr{M}_{\eqref{eq:slmH_lambda_condition}} \cap \mathscr{M}_{2 | \#}$.}
  \centering\bigskip
  \aboverulesep=0ex
  \belowrulesep=0ex
  \renewcommand{\arraystretch}{2}
  \makebox[\textwidth][c]{
  \begin{tabular}[t]{ll@{}r|c|c|}
    $\mathcal{P}$ & Parameter range & \multicolumn{1}{c}{Subrange} & \multicolumn{1}{c}{$\mathcal{Q}^-(\mathcal{P})$} & \multicolumn{1}{c}{$\mathcal{Q}^+(\mathcal{P})$} \\
    \midrule
    \multirow{2}{*}{$\mathcal{C}_i$} & \multirow{2}{*}{$\begin{cases} 2 \leq i \leq 2m \\ i \text{ even} \end{cases}$} & $i < 2m$ & \faded{as below} & $\mathcal{C}_i$ \\ \cmidrule{3-3}\cmidrule{5-5}
    & & $i = 2m$ & \multicolumn{2}{c|}{$\mathcal{C}_{i-2}$} \\ \midrule
    \multirow{2}{*}{$\mathcal{C}_i + a \mathcal{C}_1$} & \multirow{2}{*}{$\begin{cases} 0 < a \leq \min(i-2,\; 2m-i) \\ i+a \text{ even} \end{cases}$} & $a < 2m-i$ & \faded{as below} & $\mathcal{C}_i + a\mathcal{C}_1$ \\ \cmidrule{3-3}\cmidrule{5-5}
    & & $a = 2m-i$ & \multicolumn{2}{c|}{$\mathcal{C}_{i-1} + (a-1)\mathcal{C}_1$} \\ \midrule
    \multirow{2}{*}{$a \mathcal{C}_{2m-1} + \mathcal{C}_i$} & \multirow{2}{*}{$\begin{cases} 0 < a \leq \min(i,\; 2m-2-i) \\ i+a \text{ even} \end{cases}$} & $a = i$ & \multicolumn{2}{c|}{$\mathcal{C}_{2m-2} + (a-1)\mathcal{C}_{2m-3} + \mathcal{C}_{i-1}$} \\ \cmidrule{3-3}\cmidrule{4-4}
    & & $a < i$ & $a\mathcal{C}_{2m-3} + \mathcal{C}_{i-2}$ & \faded{as above} \\ \midrule
    \multirow{3}{*}{$\mathcal{C}_j + \mathcal{C}_i$} & \multirow{3}{*}{$\begin{cases} 2 \leq i < j \leq 2m-2 \\ i, j \text{ odd} \end{cases}$} & $i+j < 2m$ & \faded{as below} & $\mathcal{C}_j + \mathcal{C}_i$ \\ \cmidrule{3-3}\cmidrule{5-5}
    & & $i+j = 2m$ & \multicolumn{2}{c|}{$\mathcal{C}_{j-1} + \mathcal{C}_{i-1}$} \\ \cmidrule{3-3}\cmidrule{4-4}
    & & $i+j > 2m$ & $\mathcal{C}_{j-2} + \mathcal{C}_{i-2}$ & \faded{as above} \\ \midrule
    \multirow{3}{*}{$2\mathcal{C}_i$} & \multirow{3}{*}{$\begin{cases} 2 \leq i \leq 2m-2 \\ i \text{ odd} \end{cases}$} & $i < m$ & \faded{as below} & $2\mathcal{C}_i$ \\ \cmidrule{3-3}\cmidrule{5-5}
    & & $i = m$ & \multicolumn{2}{c|}{$\mathcal{C}_{i} + \mathcal{C}_{i-2}$} \\ \cmidrule{3-3}\cmidrule{4-4}
    & & $i > m$ & $2\mathcal{C}_{i-2}$ & \faded{as above} \\ \bottomrule
  \end{tabular}
  }
\end{table}

\begin{itemize}
\item Checking that $\mathcal{Q}^-$ and $\mathcal{Q}^+$ have the correct number of boxes, \ie that they satisfy~\eqref{eq:Q_pm_size}, is an immediate computation. In particular $\# \mathcal{Q}^-$ and~$\# \mathcal{Q}^+$ are even; by assumption, so is $\# \mathcal{P}$; this yields condition~\ref{itm:skew_diagram_even}.
\item The fact that both $\mathcal{Q}^-$ and $\mathcal{Q}^+$ are well-defined and of order~$n-2$ (\ie that all the terms $\mathcal{C}_k$ that comprise them satisfy $0 \leq k \leq 2m-2$), that they are contained in~$\mathcal{P}$, and that the skew diagrams $\mathcal{P}/\mathcal{Q}^-$ and~$\mathcal{P}/\mathcal{Q}^+$ have thickness at most~$2$ (condition~\ref{itm:thickness_2_bis}) is apparent by inspection.
\item We also notice that all of these skew diagrams have either no bridges at all, or exactly $2$ bridges of length~$1$. In particular they satisfy condition~\ref{itm:no_majority_bridge_bis}.
\item Finally, for condition~\ref{itm:slmH_mu_condition} \ie~\eqrefWithArgument{eq:slmH_lambda_condition}{$m-1$}, it is helpful to rewrite it  as~\eqrefWithArgument{eq:slmH_lambda_condition_rewritten}{$m-1$}: explicitly, for a diagram $\mathcal{Q} = \sum_{i=1}^{n-2} y_i \mathcal{C}_i$, the system~\eqrefWithArgument{eq:slmH_lambda_condition}{$m-1$} is equivalent to
\begin{equation}
\tag{\ref{eq:slmH_lambda_condition_rewritten}.$m-1$}
\begin{cases}
\displaystyle \sum_{i=1}^{n-3} \min(i-2,\; n-i-2) y_i \geq 0; \\
\displaystyle \sum_{i=1}^{n-3} \min(i,\; n-i-4) y_i \geq 0.
\end{cases}
\end{equation}
Checking this for all the diagrams $\mathcal{Q}^\pm$ is somewhat tedious, but straightforward. \qedhere
\end{itemize}
\end{proof}

We are now ready to conclude this subsubsection.

\begin{proof}[Proof of Lemma~\ref{induction_step}]~
\begin{itemize}
\item Assume first that a diagram $\mathcal{Q}$ satisfying conditions \ref{itm:slmH_mu_condition} through~\ref{itm:proportional_size} exists.

Then condition~\ref{itm:skew_diagram_even} combined with~\ref{itm:proportional_size} implies that $\# \mathcal{P}$ is divisible by~$n$.

Now consider condition~\ref{itm:no_majority_bridge_bis}: we have seen that it is equivalent to the system of inequalities~\eqref{eq:no_majority_bridge}. Consider specifically \eqrefWithArgument{eq:no_majority_bridge}{1}, \ie the first inequality of that system:
\[-p_1 + q_1 + p_2 - \sum_{i=2}^{n-2}|q_i - p_{i+1}| - p_n \geq 0.\]
We may then expand this into a system of the form~\eqref{eq:absolute_values_expansion}. That system contains among others the inequality
\[-p_1 + q_1 + p_2 + \sum_{i=2}^{m} (-q_i + p_{i+1}) + \sum_{i=m+1}^{2m-2} (q_i - p_{i+1}) - p_{2m} \geq 0,\]
that we may rewrite as
\begin{equation}
-p_1 + \sum_{i=2}^{m+1} p_i - \sum_{i=m+2}^{2m} p_i \;\;\geq\;\; -q_1 + \sum_{i=2}^{m} q_i - \sum_{i=m+1}^{2m-2} q_i;
\end{equation}
and the first part of~\eqrefWithArgument{eq:slmH_lambda_condition}{$m$} becomes a consequence of the first part of~\eqrefWithArgument{eq:slmH_lambda_condition}{$m-1$}.

Similarly, by using the inequality~\eqrefWithArgument{eq:no_majority_bridge}{$n-1$}, we deduce the second part of~\eqrefWithArgument{eq:slmH_lambda_condition}{$m$} from the second part of~\eqrefWithArgument{eq:slmH_lambda_condition}{$m-1$}.

\item Conversely, let $\mathcal{P}$ be a Young diagram satisfying the assumptions, \ie let $\mathcal{P}$ be in $\mathscr{M}_{\eqref{eq:slmH_lambda_condition}} \cap \mathscr{M}_{n | \#}$, which (since $n$~is even) is a submonoid of $\mathscr{M}_{\eqref{eq:slmH_lambda_condition}} \cap \mathscr{M}_{2 | \#}$. As announced, we take advantage of additivity by decomposing it as
\[\mathcal{P} = \sum_{l=1}^N \mathcal{P}_{l},\]
with each $\mathcal{P}_{l}$ lying in the basis~$\mathscr{B}'$ of the latter monoid. We then set, for each $k = 0, \ldots, N$:
\begin{equation}
\mathcal{Q}_{k} := \sum_{l=1}^k \mathcal{Q}^-(\mathcal{P}_{l}) + \sum_{l=k+1}^N \mathcal{Q}^+(\mathcal{P}_{l}),
\end{equation}
where $\mathcal{Q}^\pm(\mathcal{P}_l)$ are the diagrams constructed in Lemma~\ref{horrible_technical_part}. By construction of these diagrams and by Lemma~\ref{good_pairs_closed_under_addition}, it follows that each of the pairs $(\mathcal{P}, \mathcal{Q}_{k})$ satisfies conditions \ref{itm:slmH_mu_condition} through~\ref{itm:no_majority_bridge_bis}. On the other hand, also by construction, the numbers
\[\# \mathcal{Q}_{0},\; \# \mathcal{Q}_{1},\; \ldots, \# \mathcal{Q}_{N}\]
are all even, form a nondecreasing sequence with consecutive terms differing by at most~$2$, and satisfy
\[\# \mathcal{Q}_{0} \leq \frac{n-2}{n} \# \mathcal{P} \leq \# \mathcal{Q}_{N}.\]
This implies that for a suitable choice of~$k$, we have
\[\# \mathcal{Q}_{k} = \frac{n-2}{n} \# \mathcal{P}\]
as required. \qedhere
\end{itemize}
\end{proof}

\subsubsection{The case $\lie{sl}_m(\HH)$: conclusion}
\label{sec:slmH_conclusion}

It remains to put everything together.

\begin{proof}[Proof of Proposition~\ref{slmH_combinatorics}]
As announced, we proceed by induction on~$m$.

For $m = 1$, we have $\Theta = \{\alpha_1\} = \Pi$, and \eqrefWithArgument{eq:slmH_lambda_condition}{$m$} reduces to the condition $p_1 = p_2$. The result is then a particular case of Lemma~\ref{rectangle_characterization} (for $k = 2$).

Assume now that $m \geq 2$, and that the result is true for $m-1$. Let $\mathcal{P}$ be a Young diagram of order~$2m$. The result for $\mathcal{P}$ then follows by combining the lemmas proved so far, along the following outline:
\begin{center}
\bgroup
\def\arraystretch{1.5}
\makebox[\textwidth][c]{
\begin{tabular}{@{}l@{}c@{}c@{}l@{}c@{}c@{}l@{}c@{}c@{}l@{}}
 & & & \multirow{2}{*}{$\exists \mathcal{Q}$ of ord. $2m-2$,} & & & $\exists \mathcal{Q}$ of ord. $2m-2$, & & & \\[2mm]
 & & & \multirow{2}{*}{$\mathcal{Q}$ has $\Theta_{m-1}$-cbsf} & \multirow{2}{*}{$\smash{\overset{(B)}{\iff}}$} & \ldelim\{{2}{*} & $\#\mathcal{Q} \in (2m-2)\ZZ$ \rlap{\quad $\longrightarrow$ \quad redundant} & & & \\
\multirow{5}{*}{$\mathcal{P}$ has $\Theta_m$-cbsf\,\,} & \multirow{5}{*}{$\smash{\overset{(A)}{\iff}}$} & & & & & $\mathcal{Q} \vdash$ \eqrefWithArgument{eq:slmH_lambda_condition}{$m-1$} & \rdelim\}{5}{*} & \multirow{5}{*}{$\smash{\overset{(D)}{\iff}}$} & \multirow{5}{*}{$\left\lbrace\aligned &\#\mathcal{P} \in 2m\ZZ \\ &\mathcal{P} \vdash \eqrefWithArgument{eq:slmH_lambda_condition}{$m$} \endaligned\right.\ignorespacesafterend$} \\
 & & & \multirow{3}{*}{$\begin{rcases} \mathcal{Q} \subset \mathcal{P} \\ \text{$\mathcal{P}/\mathcal{Q}$ has $\Theta_1$-cbsf} \end{rcases}$} & \multirow{3}{*}{$\smash{\overset{(C)}{\iff}}$} & \ldelim\{{3}{*} & $\#\mathcal{P}/\mathcal{Q} \in 2\ZZ$ & & & \\
 & & & & & & $\forall j, 0 \leq \#^j \mathcal{P} - \#^j \mathcal{Q} \leq 2$ & & & \\
 & & & & & & $\forall i, b_i \leq \frac{1}{2}\sum_i b_i$ & & & \\
 & & \ldelim\{{-5.5}{*} & $\#\mathcal{Q} = \frac{m-1}{m}\#\mathcal{P}$ & & & $\#\mathcal{Q} = \frac{m-1}{m}\#\mathcal{P}$ & & &
\end{tabular}
}
\egroup
\end{center}
Here ``has $\Theta_m$-cbsf'' is shorthand for ``has a $\{\alpha_1, \alpha_3, \ldots, \alpha_{2m-1}\}$-codominant balanced semistandard $\{1, 2, \ldots, 2m\}$-filling'', and the symbol $\vdash$ is taken to mean ``satisfies''. Naturally, $b_i$ here stands for the length of the $i$-th bridge (see Definition~\ref{bridge_definition}) of the skew-diagram $\mathcal{P}/\mathcal{Q}$. The ingredients of the proof are then as follows:
\begin{itemize}
\item equivalence (A) is the (obvious) ``divide-and-conquer'' Lemma~\ref{dominance_decomposition}, applied to $k = 2m-2$;
\item equivalence (B) is the induction hypothesis;
\item equivalence (C) is Lemma~\ref{conditions_for_nice_filling_of_skew_diagram} (the main result of Subsubsection~\ref{sec:skew_thickness_two});
\item equivalence (D) is Lemma~\ref{induction_step} (the main result of Subsubsection~\ref{sec:induction_step}).
\end{itemize}
Also the condition ``$\#\mathcal{Q} \in (2m-2)\ZZ$'' is marked as redundant, as it follows from $\#\mathcal{P}/\mathcal{Q} \in 2\ZZ$ together with $\#\mathcal{Q} = \frac{m-1}{m}\#\mathcal{P}$.
\end{proof}

\subsection{The case $\lie{g} = A_r$: conclusion}
\label{sec:An_conclusion}

We are now ready to complete the proof of the Main Theorem when $\lie{g}$ is of type $A_r$, where $r = n-1$ is some positive integer. Almost all of the substantial work has been done in the previous subsections; it just remains to put the pieces together.

\begin{proof}[Proof of Main Theorem for $\lie{g}$ of type $A_{r \geq 1}$.]~
\begin{itemize}
\item For the real form $\lie{g}_\RR = \lie{sl}_n(\RR)$ which is split, the Main Theorem follows from Proposition~\ref{basic_results}.\ref{itm:main_split}, as we have already noted in the introduction to this section.

Note that the same argument goes for $\lie{g}_\RR = \lie{su}(p,n-p)$ when $n = 2p$ or $n = 2p+1$: indeed, we then have $\Theta(\lie{g}_\RR) = \emptyset$, which is easily seen to be equivalent to $\lie{g}_\RR$ being quasi-split. But these cases are also covered in the next point.

\item For $\lie{g}_\RR = \lie{su}(p,n-p)$ with arbitrary $p$, the result follows from Proposition~\ref{supq_combinatorics}, applied to $k = n-2p$ if $p < \frac{n}{2}$ or $k = 1$ otherwise. Indeed, given the formula~\eqref{eq:lambda_i_and_P_i} linking $\lambda_i$ and $\#_i \mathcal{P}$, we see that:
\begin{itemize}
\item $\mathcal{P}$ has integer offset if and only if its $\lie{sl}_n$-shape $\lambda$ has integer coordinates $\lambda_i$ (in the basis $(e_1, \ldots, e_n)$), which is equivalent (see Table~\ref{tab:root_lattice_congruences}) to $\lambda \in Q$;
\item $\lambda$ satisfies the inequalities listed in the appropriate line of Table~\ref{tab:conditions_for_classical_algebras} if and only if $\mathcal{P}$ satisfies the inequalities \eqref{eq:su_pq_diagram_inequalities}, for the given value of~$k$.
\end{itemize}
On the other hand, the condition $V_\lambda^\lie{l} \neq 0$ can be translated via Corollary~\ref{characterization_for_sln}.

It remains to explain how we pass from the set $\Theta(\lie{g}_\RR)$ to the set $\Pi_{[1,k-1]}$. If $p = \frac{n}{2}$, then both sets are empty, hence equal. Otherwise, we have $\Theta(\lie{g}_\RR) = \Pi_{[p+1,\; n-p-1]}$ and $\Pi_{[1,k-1]} = \Pi_{[1,\;n-2p-1]}$; and these two sets are mapped to each other by a certain element of the Weyl group (which acts by permutation of the indices $1, \ldots, n$). Hence the corresponding Levi subalgebras, say $\lie{l}_1$, $\lie{l}_2$, are conjugate in $G_\RR = \SU(p, n-p)$, so that the two spaces $V_\lambda^{\lie{l}_{1,2}}$ have the same dimension.

\item For $\lie{g}_\RR = \lie{sl}_m(\HH)$, the result follows from Proposition~\ref{slmH_combinatorics}. Indeed, the condition $V_\lambda^\lie{l} \neq 0$ has simply been translated via Corollary~\ref{characterization_for_sln}. Moreover, as in the previous point, we have $a \in \ZZ \iff \lambda \in Q$. Finally, the inequalities~\eqrefWithArgument{eq:slmH_lambda_condition}{$m$} are simply a homogeneous version of the inequalities appearing in the line $\lie{g}_\RR = \lie{sl}_m(\HH)$ of Table~\ref{tab:conditions_for_classical_algebras}: we substituted $\lambda_i = p_i - a$, and then expanded the $a$~terms in terms of the $p_i$ (instead of passing $a$ to the right-hand side, as we did for $\lie{su}(p,n-p)$). \qedhere
\end{itemize}
\end{proof}

\section{Types $B_r$, $C_r$ and~$D_r$: the setup}
\label{sec:BCD}

The goal of this section is to obtain Corollary~\ref{combinatorial_characterization}, which is a purely combinatorial characterization of the weights $\lambda$ such that $V^\lie{l}_\lambda \neq 0$ in the case when $\lie{g}$ is of type $B_r$, $C_r$ or~$D_r$, analogous to Corollary~\ref{characterization_for_sln} from the previous section. It will allow us, in the next section, to actually classify these weights $\lambda$. 

This criterion relies on so-called ``$\lie{g}$-standard doubled Young tableaux'', which play in types $B$, $C$ and~$D$ the same role as ordinary semistandard Young tableaux in type~$A$. More generally, these tableaux lead to a combinatorial character formula (Proposition~\ref{BCD_character_formula}) in types $B$, $C$ and $D$, analogous to Proposition~\ref{classical_character_formula} in type~$A$, which may be of independent interest. In types $B$ and $C$, this character formula already appears (without proof) in \cite[Appendix~A.2]{Lit90}. In type~$D$, a similar formula appears (also without proof) in \cite[Appendix~A.3]{Lit90}, but our formula constitutes a slight improvement, as discussed in the introduction (Subsubsection~\ref{sec:intro_dYt}).

All of this work is based on the Littelmann path model (that gives a character formula for any semisimple Lie algebra~$\lie{g}$), whose construction we briefly recall in Subsection~\ref{sec:Littelmann}.

In Subsection~\ref{sec:Bruhat}, we explain how to describe the path model based on a long starting path in terms of the path models based on its segments, using the Bruhat order. This part is also essentially due to Littelmann.

Starting from this point, we specialize to the case where $\lie{g}$ is of type $B$, $C$ or $D$. In Subsection~\ref{sec:Young_order}, we present a characterization of the Bruhat order in terms of Young tableaux, given some (reasonable) assumptions. This simple characterization is the key point that allows us to simplify the definition of the ``doubled Young tableaux'' in type~$D$.

In Subsection~\ref{sec:admissible}, we describe the path model on a ``short'' starting path of the form $e_1 + \ldots + e_{k-1} \pm e_k$, in terms of so-called ``admissible pairs'' (a notion due to Lakshmibai-Seshadri and Littelmann); and we give an explicit combinatorial description of these admissible pairs.

Finally, in Subsection~\ref{sec:dYt}, we define a $\lie{g}$-standard doubled Young tableau, and give the announced character formula and Levi branching rule in terms of these tableaux.

\subsection{The Littelmann path model}
\label{sec:Littelmann}

In this subsection, we briefly recall Littelmann's path technique, that provides a character formula for representations of an arbitrary semisimple Lie algebra~$\lie{g}$ (Proposition~\ref{Littelmann_character_formula}), as well as a generalization of the Littlewood-Richardson rule and its multiple avatars, including a branching rule from~$\lie{g}$ to any Levi subalgebra (Proposition~\ref{Vl_in_terms_of_paths}).

\begin{definition}[\cite{Lit95}]
\label{path_notations}
Let $\mathscr{P}$ be the set of continuous piecewise-linear paths in~$\lie{h}_{(\RR)}$ starting at~$0$, \ie maps $\pi: [0,1] \to \lie{h}_{(\RR)}$ such that $\pi(0) = 0$, considered up to reparametrization (by any increasing homeomorphism $[0,1] \to [0,1]$). We denote by $\mathscr{P}^+$ the subset of~$\mathscr{P}$ formed by paths lying entirely within the Weyl chamber $\lie{h}^+$.

For all $\nu \in \lie{h}_{(\RR)}$, we identify~$\nu$ with the linear path
\[\fundef{\nu:}{[0,1]}{\lie{h}_{(\RR)}}{t}{t\nu;}\]
and, given two paths $\pi, \rho \in \mathscr{P}$, we define the concatenated path $\pi * \rho$ by
\[\fundef{\pi * \rho:}{[0,1]}{\lie{h}_{(\RR)}}
{t}{\begin{cases}
\pi(2t)    & \text{ for } t \leq \frac{1}{2} ; \\
\pi(1) + \rho(2t-1) & \text{ for } t \geq \frac{1}{2}. \end{cases}}\]

For every simple root $\alpha \in \Pi$, Littelmann introduces two functions $e_\alpha$ and $f_\alpha$ from $\mathscr{P} \sqcup \{0\}$ to itself; we refer to~\cite[Section 1]{Lit95} for their definition. Here $0$ denotes a special element, that can be considered as the zero of the free $\ZZ$-module generated by~$\mathscr{P}$; it is not to be confused with the constant zero path, about which we will never need to talk.
\end{definition}

\begin{definition}[\cite{Lit95}]
For every $\pi \in \mathscr{P}$, we define the \emph{path model} corresponding to~$\pi$ as the smallest subset $B_{\pi} \subset \mathscr{P}$ containing $\pi$ and such that $B_{\pi} \sqcup \{0\}$ is closed under all the operators $e_\alpha$ and $f_\alpha$, for all simple roots $\alpha \in \Pi$. %TODO Notation for $B_{\pi}$?
\end{definition}

\begin{proposition}[\cite{Lit95}, Character Formula]
\label{Littelmann_character_formula}
Let $\lambda \in P \cap \lie{h}^+$ be a dominant integral weight of~$\lie{g}$. Choose any starting path $\pi^+ \in \mathscr{P}^+$ having endpoint $\pi^+(1) = \lambda$. Then the endpoints of the paths in $B_{\pi^+}$ describe the character of the representation with highest weight~$\lambda$:
\[\charact(V_\lambda) = \sum_{\pi \in B_{\pi^+}} e^{\pi(1)}.\]
\end{proposition}
%(Recall~\eqref{eq:char_defn} for the definition of $\charact(V)$.)

For every set $\Theta \subset \Pi$ of simple roots, let us define the ``$\Theta$-dominant Weyl chamber''
\begin{equation}
\label{eq:Theta_dom_Weyl_chamber}
\lie{h}^{\Theta, +} := \setsuch{X \in \lie{h}_{(\RR)}}{\forall \alpha \in \Theta,\; \alpha(X) \geq 0},
\end{equation}
which is just the dominant Weyl chamber of the reductive algebra~$\lie{l}(\Theta)$. Then we have:

\begin{proposition}[\cite{Lit95}, Restriction Rule]
Let $\lambda$ and $\pi^+$ be as before, and let $\Theta \subset \Pi$ be some set of simple roots. Then the subset $B_{\pi^+}^{\Theta}$ of~$B_{\pi^+}$ formed by paths lying entirely within $\lie{h}^{\Theta, +}$ parametrizes the decomposition of the restriction of the representation $V_\lambda(\lie{g})$ to $\lie{l}(\Theta)$ into irreducibles, in the following way:
\[\restr{V_\lambda(\lie{g})}{\lie{l}(\Theta)} = \bigoplus_{\pi \in B_{\pi^+}^{\Theta}} V_{\pi(1)}(\lie{l}(\Theta)).\]
\end{proposition}

As a corollary, this allows us to compute the dimension of $V_\lambda^{\lie{l}(\Theta)}$, which is just the multiplicity of the trivial representation of~$\lie{l}(\Theta)$ in that decomposition:

\begin{corollary}
\label{Vl_in_terms_of_paths}
Let $\lambda$, $\pi^+$ and $\Theta$ be as before. Then we have
\[\dim V_\lambda^{\lie{l}(\Theta)} = \# \setsuch{\pi \in B_{\pi^+}^{\Theta}}{\pi(1) = 0}.\]
\end{corollary}

\subsection{The Bruhat order}
\label{sec:Bruhat}

In this subsection, we give a partial characterization (essentially due to Littelmann) of the path model~$B_{\pi^+}$, given some fairly natural assumptions on the starting path~$\pi^+$ (Proposition~\ref{path_model_loc_int_concat}). These assumptions are in particular satisfied by all starting paths of the form
\begin{equation}
\label{eq:dominant_integral_segments}
\pi^+ = \nu^+_1 * \cdots * \nu^+_N
\end{equation}
where each $\nu^+_i$ is a dominant integral weight; both Littelmann's~\eqref{eq:Littelmann_starting_path} and our~\eqref{eq:basic_path_definition} choice of a starting path for $\lie{g}$ of types $B_r$, $C_r$ and~$D_r$ follow this pattern. For such paths~$\pi^+$, we shall then decompose this result into two subresults.
\begin{itemize}
\item The first part (Corollary~\ref{long_path_model_characterization}) is a characterization of the path model~$B_{\pi^+}$ in terms of the path models~$B_{\nu^+_i}$ corresponding to its segments.
\item The second part (Corollary~\ref{short_path_model_characterization}) will be given only later, in subsection~\ref{sec:admissible}. It consists of a description of each of these path models~$B_{\nu^+_i}$, assuming that $\nu^+_i$ is ``small enough''.
\end{itemize}

The main tool for this characterization is the so-called Bruhat order, whose definition we now recall.

\begin{definition}
The \emph{Bruhat order}~$\preceq_B$ is the partial order on~$W$ defined as the transitive closure of the relations
\[\setsuch{w \preceq s_\alpha w}{w \in W,\; \alpha \in \Delta \text{ such that } \ell(w) < \ell(s_\alpha w)},\]
where $\ell(w)$ stands for the length of~$w$ as a word on the generators $\setsuch{s_\alpha}{\alpha \in \Pi}$.
\end{definition}

The following classical characterization (see \eg Proposition~3.2.14.(4) in~\cite{CS}) of such pairs $(\alpha, w)$ is useful to have in mind:
\begin{lemma}
\label{Bruhat_order_characterization}
For all $w \in W$ and $\alpha \in \Delta^+$, we have $\ell(s_\alpha w) > \ell(w)$ if and only if $\alpha \in w \Delta^+$, or equivalently if and only if
\begin{equation}
\label{eq:Bruhat_order_characterization}
\forall X \in \lie{h}^+,\quad \alpha(w X) \geq 0.
\end{equation}
\end{lemma}

We now use the Bruhat order to define the notion of a ``Bruhat-nondecreasing'' tuple of elements on~$\lie{h}_{(\RR)}$. Note however that such a tuple can not, in general, be thought of as a sequence that is nondecreasing for some partial order on~$\lie{h}_{(\RR)}$: indeed, the relationship of forming a Bruhat-nondecreasing pair is not transitive (see Example~\ref{Bruhat_not_transitive}).

\begin{definition}
\label{Bruhat_tuple_definition}
Let $\nu_1, \ldots, \nu_N \in \lie{h}_{(\RR)}$ be some weights. We say that the tuple $(\nu_1, \ldots, \nu_N)$ is \emph{Bruhat-nondecreasing} (once again, this is \emph{not} a transitive relation, see Example~\ref{Bruhat_not_transitive}) if there exist some elements $w_i \in W$ such that:
\begin{itemize}
\item for every $i = 1, \ldots, N$, the weight~$\nu_i$ lies in the Weyl chamber $w_i \lie{h}^+$;
\item we have $w_1 \preceq_B \ldots \preceq_B w_N$.
\end{itemize}
We say that a path $\pi$ is \emph{Bruhat-nonincreasing} if the segments $(\nu_N, \ldots, \nu_1)$ of its subdivision $\pi = \nu_1 * \cdots * \nu_N$ into linear segments form a Bruhat-nondecreasing tuple.
\end{definition}
(The order inversion between tuples and paths is explained by the order inversion that we will see in Definition~\ref{BCD_tableaux_and_weights}, more precisely in~\eqref{eq:tableau_to_path}.)

\begin{example}
\label{Bruhat_not_transitive}
For $\lie{g} = \lie{so}_6(\CC)$, take $\nu_1 = \frac{1}{2}(- e_1 - e_2 + e_3)$, $\nu_2 = \frac{1}{2}(e_1 - e_2 + e_3)$ and $\nu_3 = e_1$. Then:
\begin{itemize}
\item the pair $(\nu_1, \nu_2)$ is Bruhat-nondecreasing: take for example $w_1 = w_2 = w$ with $w: (e_1, e_2, e_3) \mapsto (e_3, -e_2, -e_1)$;
\item the pair $(\nu_2, \nu_3)$ is Bruhat-nondecreasing: take for example $w_2 = w_3 = w'$ with $w': (e_1, e_2, e_3) \mapsto (e_1, e_3, e_2)$;
\item the pair $(\nu_1, \nu_3)$, and \emph{a fortiori} the triple $(\nu_1, \nu_2, \nu_3)$, is \emph{not} Bruhat-nondecreasing. We will be easily able to check this once we obtain Proposition~\ref{Bruhat_vs_Young}.
\end{itemize}
\end{example}

We also need one final short definition.

\begin{definition}
Let $\pi$ be a path, $\pi = \nu_1 * \cdots * \nu_N$ its subdivision into linear segments. We define the \emph{multishape} of~$\pi$ to be the path
$\pi^+ = \nu^+_1 * \cdots * \nu^+_N$,
where, for each~$i$, $\nu^+_i$ is the unique dominant element of the Weyl orbit of~$\nu_i$:
\[\{\nu^+_i\} := W \nu_i \cap \lie{h}^+.\]
\end{definition}

Here is now the announced result.

\begin{proposition}
\label{path_model_loc_int_concat}
Let $\nu^+_1, \ldots, \nu^+_N$ be some dominant weights such that the path $\pi^+ = \nu^+_1 * \cdots * \nu^+_N$ is a locally integral concatenation (see \cite[Definition~5.3]{Lit97}). Then a path $\pi$ lies in $B_{\pi^+}$ if and only if it is a locally integral concatenation and has multishape $\pi^+$.
\end{proposition}
\begin{proof}
By Proposition~5.9 in~\cite{Lit97}, local integrality is preserved by the root operators. The multishape is obviously invariant by the root operators. Denoting by $\hat{B}_{\pi^+}$ the set of locally integral concatenations with multishape~$\pi^+$, it then follows from Lemma~6.11 in~\cite{Lit97} that
\[\hat{B}_{\pi^+} = \bigcup_{\pi \in \hat{B}_{\pi^+} \cap \mathscr{P}^+} B_{\pi}.\]
It remains to check that the only path with multishape~$\pi^+$ that is a locally integral concatenation and lies entirely within~$\lie{h}^+$ is $\pi^+$ itself. Indeed, let $\pi$~be such a path; replacing if necessary $(\nu^+_1, \ldots, \nu^+_N)$ by a finer subdivision (whose concatenation then remains locally integral), we may assume that $\pi$ is of the form
\[\pi = w_1 \nu^+_1 * \cdots * w_N \nu^+_N,\]
with $w_1 \succeq_B \cdots \succeq_B w_N$. It is then easy to verify, by induction on~$k$ (and using Lemma~\ref{Bruhat_order_characterization}), that we have $w_k \nu^+_l = \nu^+_l$ for all $k, l$ such that $k \geq l$.
%
%Indeed, fix some~$k$, and assume that this is true for all $k' < k$. Then for $l < k$, we know that $w_l \in \operatorname{Stab}_W \nu^+_l$ and $w_k \preceq w_l$, from which, by well-known properties of the Bruhat order, it follows that we still have $w_k \in \operatorname{Stab}_W \nu^+_l$.
%
%Now consider $l = k$. Take any simple root~$\alpha \in \Pi$: we want to show that $\langle \alpha, w_k \nu^+_k \rangle \geq 0$. We distinguish two cases.
%\begin{itemize}
%\item Either $\langle \alpha, \sum_{l=1}^{k-1} w_l \nu^+_l \rangle = 0$. Then we may simply subtract this equality from the inequality $\langle \alpha, \sum_{l=1}^k w_l \nu^+_l \rangle \geq 0$, which reflects the assumption that the path $\pi$ lies entirely within~$\lie{h}^+$.
%\item Or $\langle \alpha, \sum_{l=1}^{k-1} w_l \nu^+_l \rangle > 0$. We already noted that for each $l < k$, the vector $w_l \nu^+_l = \nu^+_l$ is invariant by~$w_k$; hence so is their sum. In particular this means that the dominant vector $\sum_{l=1}^{k-1} w_l \nu^+_l$ is also contained in the Weyl chamber $w_k \lie{h}^+$. Since every Weyl chamber lies on one side or the other of the wall $\alpha^\perp$, this means that, for every $X \in w_k \lie{h}^+$, we have $\langle \alpha, X \rangle \geq 0$. In particular this holds for $X = w_k \nu^+_k$.
%\end{itemize}
\end{proof}

And here, as announced, is the interpretation of this result as a ``divide-and-conquer'' strategy.

\begin{corollary}
\label{long_path_model_characterization}
Let $\nu^+_1, \ldots, \nu^+_N \in P \cap \lie{h}^+$ be some dominant integral weights, and let $\pi^+ = \nu^+_1 * \cdots * \nu^+_N$. Then a path $\pi$ lies in $B_{\pi^+}$ if and only if it is Bruhat-nonincreasing and of the form $\pi = \pi_1 * \cdots * \pi_N$, with $\pi_k \in B_{\nu^+_k}$ for each~$k$.
\end{corollary}

\begin{proof}
The ``only if'' part is an immediate consequence of the previous proposition (Proposition~\ref{path_model_loc_int_concat}) and of the combination of Lemma~6.12 and Theorem~6.13 from~\cite{Lit97}.

For the ``if'' part, we only need to remark that if each of the paths $\pi_1, \ldots, \pi_N$ is a weakly locally integral concatenation and ends at an integral weight, then their concatenation is automatically weakly locally integral. Here by ``weakly locally integral concatenation'' we mean a path that satisfies all of the conditions from Definition~5.3 in~\cite{Lit97}, except possibly Bruhat-nonincreasingness.
\end{proof}

\subsection{The Young order}
\label{sec:Young_order}

\begin{center}
\fbox{\begin{minipage}{0.87\textwidth}
For the remainder of the paper, we assume that $\lie{g}$ is either of type $B_r$ for some $r \geq 1$, or of type $C_r$ for some $r \geq 1$, or of type $D_r$ for some $r \geq 3$.

In this setting, we have $n = r$, so we drop the notation~$n$.
\end{minipage}}
\end{center}

For these values of $\lie{g}$, we will consider Littelmann paths whose segments lie (up to occasional $\frac{1}{2}$ factors) in the set
\begin{equation}
\label{eq:infinity_norm_one}
\mathcal{X} := \{-1,0,1\}^r \setminus \{0\}
\end{equation}
of vectors with integer coordinates (in the basis $(e_1, \ldots, e_r)$) that have $\|\cdot\|_\infty$-norm $1$. We will encode these vectors as ``strongly-standard'' columns (\ie Young tableaux of width~$1$) on a certain alphabet: this is the object of Definitions \ref{alphabet}, \ref{strong_stand_col} and~\ref{cols_vs_weights}.

We then introduce (Definition~\ref{Young_order}) a ``Young order'' on the set~$\mathscr{C}$ of such columns, with an additional parity condition when $\lie{g}$ is of type $D_r$. This leads us to the central result of this subsection: Proposition~\ref{Bruhat_vs_Young}, which says that, for a sequence of elements of~$\mathcal{X}$ whose Weyl orbits are ordered in some natural way, being Bruhat-nondecreasing is equivalent to being nondecreasing for the Young order (or Young order with parity). In other terms, this assumption on the ordering of the Weyl orbits gets rid of the nontransitivity issues outlined in Example~\ref{Bruhat_not_transitive}.

\begin{definition}[The alphabet]
\label{alphabet}
We introduce the alphabet
\[\mathcal{A}_r := \{1, \ldots, r, \overline{r}, \ldots, \overline{1}\};\]
we also set $\mathcal{A} := \bigcup_{r \in \NN} \mathcal{A}_r = \NN \cup \overline{\NN}$. We adopt the convention $\overline{\overline{s}} = s$, and we define an \emph{absolute value} function $|\cdot|: \mathcal{A} \to \NN$ and a \emph{sign} function $\sgn: \mathcal{A} \to \{\pm 1\}$ by identifying $\overline{s}$ with $-s$.

We introduce on~$\mathcal{A}_r$ two very similar orders:
\begin{itemize}
\item the total order~$\preceq_{\mathcal{A}}$ given by
\begin{equation}
1 \prec \cdots \prec r \prec \overline{r} \prec \cdots \prec \overline{1};
\end{equation}
\item the not quite total order~$\preceq'_{\mathcal{A}}$ given by
\begin{equation}
1 \prec' \cdots \prec' r-1  \prec' r, \overline{r} \prec' \overline{r-1}  \prec' \cdots \prec' \overline{1},
\end{equation}
\ie $r$ and~$\overline{r}$ are both larger than $r-1$ and both smaller than~$\overline{r-1}$ for this order, but neither is larger than the other.
\end{itemize}
The order that we shall use will depend on~$\lie{g}$: more precisely, we set
\[\preceq^{\lie{g}}_{\mathcal{A}} := \begin{cases}
\preceq_{\mathcal{A}} & \text{ if $\lie{g}$ is of type $B_r$ or $C_r$;} \\
\preceq'_{\mathcal{A}} & \text{ if $\lie{g}$ is of type $D_r$.}
\end{cases}\]
However the total order $\preceq_{\mathcal{A}}$ will occasionally be useful even in type~$D_r$ (see Remark~\ref{semistandard_with_parity_is_semistandard}).
\end{definition}

\begin{definition}[Strongly standard columns]
\label{strong_stand_col}
Let~$\mathcal{C}$ be a column (\ie a Young tableau of width~$1$) filled with this alphabet. We say that~$\mathcal{C}$ is \emph{strongly standard} if, for each~$s$, it contains at most one of the symbols $s$ and~$\overline{s}$, and the symbols read from top to bottom form a strictly increasing sequence for the order~$\prec_{\mathcal{A}}$ (or equivalently for the order~$\prec'_{\mathcal{A}}$):
\[\forall i < i',\quad
\begin{cases}
\overline{\threeind{}{}{i} \mathcal{C}} \neq \threeind{}{}{i'} \mathcal{C}; \\
\threeind{}{}{i} \mathcal{C} \prec_{\mathcal{A}} \threeind{}{}{i'} \mathcal{C}.
\end{cases}\]
We denote by $\mathscr{C}$ the set of all strongly standard columns.
\end{definition}

\begin{definition}[Identification of columns and weights]
\label{cols_vs_weights}
We define the \emph{weight} of a strongly standard column~$\mathcal{C}$ to be the vector
\begin{equation}
\nu(\mathcal{C}) := \sum_{i=1}^{\# \mathcal{C}} \nu\left( \threeind{}{}{i} \mathcal{C} \right),
\end{equation}
with the function $\nu$ defined on~$\mathcal{A}_r$ by
\begin{equation}
\forall s = 1, \ldots, r,\quad
\begin{cases}
\nu(s) := e_s; \\
\nu(\overline{s}) := -e_s.
\end{cases}
\end{equation}
This map $\nu$ then induces a bijection between $\mathscr{C}$ and the set $\mathcal{X}$ introduced in \eqref{eq:infinity_norm_one}, that we shall henceforth identify with~$\mathscr{C}$.
\end{definition}

We now introduce an order on the set~$\mathscr{C}$ of strongly standard columns. In types $B_r$ and $C_r$, we use the ``Young order'', which simply encodes the notion of a semistandard Young tableau (a Young tableau is semistandard if and only if its columns form a nondecreasing sequence for this order); in type $D_r$, we use the Young order with an additional parity condition.

\begin{definition}
\label{Young_order}
We endow the set~$\mathscr{C}$ (and, using the identification~$\nu$, also the set~$\mathcal{X}$) with an order~$\preceq^{\lie{g}}_Y$, that once again depends on~$\lie{g}$: we set
\[\preceq^{\lie{g}}_Y := \begin{cases}
\preceq_Y & \text{ if $\lie{g}$ is of type $B_r$ or $C_r$,} \\
\preceq'_Y & \text{ if $\lie{g}$ is of type $D_r$;}
\end{cases}\]
it remains to explain what $\preceq_Y$ and $\preceq'_Y$ are.
\begin{enumerate}
\item We define the \emph{Young order}~$\preceq_Y$ by saying that $\mathcal{C} \preceq_Y \mathcal{C}'$ if and only if the two columns set next to each other form a semistandard Young tableau for the order~$\preceq_{\mathcal{A}}$. Formally:
\begin{equation}
\label{eq:Young_order_definition}
\mathcal{C} \preceq_Y \mathcal{C}' :\iff
\begin{cases}
\# \mathcal{C} \geq \# \mathcal{C}'; \\
\forall i = 1, \ldots, \# \mathcal{C}',\quad \threeind{}{}{i} \mathcal{C} \preceq_A \threeind{}{}{i} \mathcal{C}'.
\end{cases}
\end{equation}

\item We define the \emph{Young order with parity}~$\preceq'_Y$ by saying that $\mathcal{C} \preceq'_Y \mathcal{C}'$ if and only if the two columns set next to each other form a semistandard Young tableau for the order~$\preceq'_{\mathcal{A}}$, and this tableau satisfies the following additional condition: whenever it contains a rectangle of width~$2$ and height~$k$ that contains only symbols with absolute value greater than~$r-k$, the total number of symbols in that rectangle that are in~$\NN$ (\ie are without bars) must be even:
\begin{gather}
\text{if } \exists i_0, k \text{ with } 1 \leq i_0 \leq i_0 + k-1 \leq \# \mathcal{C}' \text{ such that } \nonumber \\
\label{eq:parity_condition} \left\{ \left| \threeind{}{}{i_0}\mathcal{C} \right|,\; \ldots,\; \left| \threeind{}{}{i_0+k-1}\mathcal{C} \right| \right\} = \left\{ \left| \threeind{}{}{i_0}\mathcal{C}' \right|,\; \ldots,\; \left| \threeind{}{}{i_0+k-1}\mathcal{C}' \right| \right\} = \{r-k+1,\; \ldots,\; r\}, \\
\text{then } \#\threeind{}{\NN}{}\mathcal{C} \equiv \#\threeind{}{\NN}{}\mathcal{C}' \pmod{2}. \nonumber
\end{gather}
One easily checks that this relation is transitive.
\end{enumerate}
\end{definition}
\begin{remark}
\label{semistandard_with_parity_is_semistandard}
Note that the case $k = 1$ of the condition \eqref{eq:parity_condition} tells us that in a tableau whose columns form a $\preceq'_Y$-nondecreasing sequence, $r$ and~$\overline{r}$ can never occur next to each other. So such a tableau will in particular be semistandard, not only for the partial order~$\preceq'_{\mathcal{A}}$, but also for the total order~$\preceq_{\mathcal{A}}$; and, for that matter, also for the total order~$\preceq''_{\mathcal{A}}$ in which $r$ and~$\overline{r}$ are swapped:
\begin{equation}
1 \prec'' \cdots \prec'' r-1 \prec'' \overline{r} \prec'' r \prec'' \overline{r-1} \prec'' \cdots \prec'' \overline{1}.
\end{equation}
\end{remark}

Finally, as announced, we explain how the Young order (in types $B_r$ and~$C_r$) or the Young order with parity (in type~$D_r$) is related to the Bruhat order. The remainder of this subsection is dedicated to proving the following proposition.

\begin{proposition}
\label{Bruhat_vs_Young}
Let $\mathcal{C}_1, \ldots, \mathcal{C}_N \in \mathscr{C}$, and let $\nu_i := \nu(\mathcal{C}_i)$ be the corresponding weights.
Then:
\begin{hypothenum}
\item for $\lie{g}$ of type $B_r$ or~$C_r$, the sequence $(\mathcal{C}_1, \ldots, \mathcal{C}_N)$ is Young-nondecreasing if and only if:
\[\begin{cases}
(\nu_1, \ldots, \nu_N) \text{ is Bruhat-nondecreasing}; \\
\|\nu_1\|^2 \geq \cdots \geq \|\nu_N\|^2.
\end{cases}\]
\item for $\lie{g}$ of type $D_r$, the sequence $(\mathcal{C}_1, \ldots, \mathcal{C}_N)$ is Young-nondecreasing with parity if and only if:
\[\begin{cases}
(\nu_1, \ldots, \nu_N) \text{ is Bruhat-nondecreasing}; \\
\|\nu_1\|^2 \geq \cdots \geq \|\nu_N\|^2; \\
\text{all $\nu_i$ such that $\|\nu_i\|^2 = r$ lie in the same $W$-orbit.}
\end{cases}\]
\end{hypothenum}
\end{proposition}

In order to prove this proposition, we need some preliminary work. Recall that a partially ordered set $(X, \preceq)$ can be characterized by its \emph{Hasse diagram}, \ie the oriented graph whose vertices are the elements of~$X$, with two vertices $x, y$ connected by an edge if and only if $y$ ``covers'' $x$, \ie
\[x \preceq y \text{ and } \setsuch{z}{x \preceq z \preceq y} = \{x, y\}.\]
We then have the following description of the Hasse diagram of the order $\preceq^{\lie{g}}_Y$, for both $\lie{g} = B_r$ or~$C_r$ and $\lie{g} = D_r$.

\begin{lemma}
\label{hasse_edges_description}
Let $\mathcal{C}, \mathcal{C}'$ be two strongly standard columns. Then the pair $(\mathcal{C}, \mathcal{C}')$ is an edge of the Hasse diagram for $\preceq^{\lie{g}}_Y$ if and only if it has one of the following forms:
\begin{itemize}
\item $\mathcal{C}$ and~$\mathcal{C}'$ have the same height, and differ in only one box:
\begin{subequations}
\label{eq:hasse_edges_description}
\begin{equation}
\label{eq:hasse_skip}
(\mathcal{C}, \mathcal{C}') = \left(\,
  \diagramfontsize
  \Yboxdim{12pt}
  \begin{tikzpicture}[baseline={([yshift=-.5ex]current bounding box.center)}]
  \tgyoung(0cm,0cm,\vdts,s,\vdts)
  \end{tikzpicture}
  \normalsize
\, ,\;
  \diagramfontsize
  \Yboxdim{12pt}
  \begin{tikzpicture}[baseline={([yshift=-.5ex]current bounding box.center)}]
  \tgyoung(0cm,0cm,\vdts,t,\vdts)
  \end{tikzpicture}
  \normalsize
\,\right)
\end{equation}
with $s \prec_\mathcal{A}^{\lie{g}} t$ and such that for every symbol $x$ satisfying $s \prec_\mathcal{A}^{\lie{g}} x \prec_\mathcal{A}^{\lie{g}} t$, the value $\overline{x}$ is contained in some box of $\mathcal{C}$ (and of $\mathcal{C}'$).
\item $\mathcal{C}$ and~$\mathcal{C}'$ have the same height, and differ in only two boxes:
\begin{equation}
\label{eq:hasse_swap}
(\mathcal{C}, \mathcal{C}') = \left(\,
  \diagramfontsize
  \Yboxdim{12pt}
  \begin{tikzpicture}[baseline={([yshift=-.5ex]current bounding box.center)}]
  \tgyoung(0cm,0cm,\vdts,s,\vdts,<\overline{t}>,\vdts)
  \end{tikzpicture}
  \normalsize
\, ,\;
  \diagramfontsize
  \Yboxdim{12pt}
  \begin{tikzpicture}[baseline={([yshift=-.5ex]current bounding box.center)}]
  \tgyoung(0cm,0cm,\vdts,t,\vdts,<\overline{s}>,\vdts)
  \end{tikzpicture}
  \normalsize
\,\right)
\end{equation}
with $s$ and~$t$ such that $(s, t)$ is an edge of the Hasse diagram for the order $\preceq_\mathcal{A}^{\lie{g}}$.
\item $\mathcal{C}'$ is equal to $\mathcal{C}$ with the last box removed:
\begin{equation}
\label{eq:hasse_truncate}
(\mathcal{C}, \mathcal{C}') = \left(\,
  \diagramfontsize
  \Yboxdim{12pt}
  \begin{tikzpicture}[baseline={([yshift=-.5ex]current bounding box.center)}]
  \tgyoung(0cm,0cm,\vdts,s)
  \end{tikzpicture}
  \normalsize
\, ,\;
  \diagramfontsize
  \Yboxdim{12pt}
  \begin{tikzpicture}[baseline={([yshift=-.5ex]current bounding box.center)}]
  \tgyoung(0cm,0cm,\vdts,:~)
  \end{tikzpicture}
  \normalsize
\,\right)
\end{equation}
\end{subequations}
with $s$ a maximal element (for the order $\preceq_\mathcal{A}^{\lie{g}}$) among the symbols that do not occur in $\mathcal{C}'$.
\end{itemize}
\end{lemma}

We omit the proof, which is a somewhat tedious but elementary exercise in combinatorics.

Before proving Proposition~\ref{Bruhat_vs_Young}, we also need to decompose the $W$-invariant set~$\mathcal{X}$ into $W$-orbits, which is equivalent to describing the set $\mathcal{X}^+ := \mathcal{X} \cap \lie{h}^+$, as we have $\mathcal{X} = \bigcup_{v \in \mathcal{X}^+} W v$. Setting, for all~$k = 1, \ldots, r$,
\begin{equation}
\label{eq:model_weight_definitions}
c_k := \sum_{i=1}^k e_i;\quad c^+_r := c_r;\quad c^-_r := c_{r-1} - e_r,
\end{equation}
we have
\begin{equation}
\label{eq:model_weights}
\mathcal{X}^+ = \begin{cases}
\{c_1, \ldots, c_r\} &\text{ if $\lie{g}$~is of type $B_r$ or~$C_r$}; \\
\{c_1, \ldots, c_{r-1}, c_r^+, c_r^-\} &\text{ if $\lie{g}$~is of type $D_r$}.
\end{cases}
\end{equation}
This set is almost totally ordered by the restriction of the order~$\preceq^{\lie{g}}_Y$, except in type~$D_r$ where $c_r^+$ and~$c_r^-$ are incomparable. More precisely, we have:
\begin{equation}
\forall v, v' \in \mathcal{X}^+,\quad v \preceq^{\lie{g}}_Y v' \quad:\iff\quad v = v' \text{ or } \|v\|^2 > \|v'\|^2.
\end{equation}

We are now ready for the proof.

\begin{proof}[Proof of Proposition~\ref{Bruhat_vs_Young}.]~
\begin{itemize}
\item The ``if'' part is now equivalent to proving that the application map
\[\fundef{\pi:}{W \times \mathcal{X}^+}{\mathcal{X}}
{(w, v)}{w v}\]
is order-preserving, where $W \times \mathcal{X}^+$ is endowed with the product order $\preceq_B \times \preceq^{\lie{g}}_Y$ and $\mathcal{X}$~is endowed with the order $\preceq^{\lie{g}}_Y$. This further reduces to proving the two identities:
\begin{align}
&\forall (w, v) \in W \times \mathcal{X}^+,\;\; \forall \alpha \in \Delta^+, &&\ell(w) < \ell(s_\alpha w) \implies w v \preceq^{\lie{g}}_Y s_\alpha w v; \\
&\forall w \in W,\;\; \forall k = 1, \ldots, r-1, &&w c^\pm_{k+1} \preceq^{\lie{g}}_Y w c_k.
\end{align}
The first identity easily follows from the characterization~\eqref{eq:Bruhat_order_characterization}; and the second identity is straightforward.

\item Conversely, assume now that the sequence $(\nu_1, \ldots, \nu_N)$ is nondecreasing for the order $\preceq^{\lie{g}}_Y$. Then clearly the Young ordering ensures that the heights of the columns $\|\nu_i\|^2 = \# \mathcal{C}_i$ form a nonincreasing sequence, and (in type~$D_r$) the parity condition ensures that all the columns of height~$r$ lie in the same Weyl orbit. It remains to prove that the sequence $(\nu_1, \ldots, \nu_N)$ is Bruhat-nondecreasing.

This can be proved by exhibiting a section
\[\xi: \mathcal{X} \to W,\]
\ie a map such that every vector $\nu \in \mathcal{X}$ lies in the Weyl chamber $\xi(\nu) \lie{h}^+$, which is also order-preserving.

We construct $\xi$ as follows. Let $\nu \in \mathcal{X}$, and let $\mathcal{C}$ be the corresponding strongly standard column; let $k = \# \mathcal{C} = \|\nu\|^2$. We define $\xi(\nu)$ as the unique element of~$W$ whose action on~$\{\pm e_1, \ldots, \pm e_r\}$, that we identify with~$\mathcal{A}_r$ as usual, satisfies:
\begin{itemize}
\item for all $i \leq k$ (except possibly $i = r$ if $\lie{g}$ is of type $D_r$), we have $\xi(\nu) \cdot i = \threeind{}{}{i} \mathcal{C}$;
\item $|\xi(\nu) \cdot (k+1)| < |\xi(\nu) \cdot (k+2)| < \cdots < |\xi(\nu) \cdot r|$;
\item for all $i > k$ (except possibly $i = r$ if $\lie{g}$ is of type $D_r$), we have $\sgn(\xi(\nu) \cdot i) = -1$.
\end{itemize}

It is then straightforward to verify that $\xi$ is indeed a section. As for the fact that it is order-preserving, it suffices to check it on the edges of the Hasse diagram, which we have described in Lemma~\ref{hasse_edges_description}. Let $(\mathcal{C}, \mathcal{C}')$ be such a pair. We then easily check that:
\begin{itemize}
\item If $(\mathcal{C}, \mathcal{C}')$ is of the form~\eqref{eq:hasse_skip} or~\eqref{eq:hasse_swap}, then we have
\begin{equation}
\label{eq:hasse_step_further_description}
\xi(\mathcal{C}') = s_\alpha \circ \xi(\mathcal{C}) \quad\text{with}\quad \alpha = e_s - e_t,
\end{equation}
with the obvious convention $e_s := -e_{\overline{s}}$ if $s \in \overline{\mathbb{N}}$. By the characterization~\eqref{eq:Bruhat_order_characterization}, we then immediately get that $\xi(\mathcal{C}) \preceq_B \xi(\mathcal{C}')$ as desired.
\item If $(\mathcal{C}, \mathcal{C}')$ is of the form~\eqref{eq:hasse_truncate}, then we simply have $\xi(\mathcal{C}') = \xi(\mathcal{C})$. \qedhere
\end{itemize}
\end{itemize}
\end{proof}

\subsection{Admissible pairs}
\label{sec:admissible}

%\begin{definition}
%Let $\nu_1, \nu_2
%%, \nu_3, \nu_4
%\in \lie{h}^+$ be dominant weights, and let $w_1, w_2
%%, w_3, w_4
%$ be elements of~$W$.
%
%We say that the weights $(w_1 \nu_1, w_2 \nu_2)$ form an \emph{admissible pair} if $\nu_1$ and~$\nu_2$ have the same stabilizer $W_{\nu}$ in~$W$, and if the pair $(\tilde{w}_2, \tilde{w}_1)$ formed by the projections of $w$ and~$w'$ modulo $W_{\nu}$ is admissible in the sense of \cite{Lit90} (definition given in Remark~3.4, and originally due to \cite[Definition~2.4]{LS86}).
%\end{definition}

We now give a reformulation of Proposition~\ref{path_model_loc_int_concat} when applied to ``short'' starting paths: this is Corollary~\ref{short_path_model_characterization} (in its notations, ``short'' means that $k$ is small).  We then further specify it to starting paths that lie in the set~$\mathcal{X}^+$ introduced~\eqref{eq:model_weights} in the previous section.

\begin{corollary}
\label{short_path_model_characterization}
Let $\nu^+ \in P \cap \lie{h}^+$ be a dominant integral weight, and consider the integer $k := \max_{\alpha \in \Delta} | \langle \nu^+, \alpha^\vee \rangle |$ (recall the notation $\alpha^\vee := \frac{2\alpha}{\|\alpha\|}$). Then:
\begin{hypothenum}
\item \label{itm:minuscule} If $k = 1$ (\ie $\nu^+$ is minuscule), then $B_{\nu^+}$ is just the $W$-orbit of~$\nu^+$.
\item \label{itm:seconscule} If $k \leq 2$, then $B_{\nu^+}$ is the set of paths $\pi$ of the form
\[\pi = \textstyle (\frac{1}{2} \nu_1) * (\frac{1}{2} \nu_2)\]
with $\nu_1, \nu_2$ two elements of the $W$-orbit of $\nu^+$ that form an admissible pair, in the sense of \cite{Lit90} (definition given in Remark~3.4, and originally due to \cite[Definition~2.4]{LS86}).
%\item If $k \leq 3$, then $B_{\nu^+}$ is the set of paths $\pi$ of the form
%\[\pi = \textstyle (\frac{1}{3} w_4 \nu^+) * (\frac{1}{6} w_3 \nu^+) * (\frac{1}{6} w_2 \nu^+) * (\frac{1}{3} w_1 \nu^+)\]
%where $w_4, w_3, w_2, w_1 \in W$ are four elements whose classes modulo the stabilizer of~$\nu^+$ form an \emph{admissible quadruple} in the sense of \cite[Definition~3.4]{Lit90}.
\end{hypothenum}
\end{corollary}

Note that this result is actually true for arbitrary~$\lie{g}$; and similar statements can be obtained for $k \leq 3$ (using the notion of an ``admissible quadruple'', see \cite[Definition~3.4]{Lit90}), $k \leq 4$ (leading to some notion of ``admissible sextuple'') and higher values of~$k$. But we go back to our assumption that $\lie{g}$ is of type $B_r$, $C_r$ or $D_r$, for which these cases do not occur. 

More specifically, we now consider the case where $\nu^+$ lies in the set~$\mathcal{X}^+$: clearly, all of its elements satisfy $k \leq 2$. In order to characterize the path model for these starting paths, it remains to give an explicit combinatorial characterization of admissible pairs (in terms of strongly standard columns). This is the subject of the following proposition, whose proof is the main goal of this subsection.

\begin{proposition}
\label{admissible_pair_characterization}
Let $\mathcal{C}$ and $\mathcal{C}'$ be two strongly standard columns. Let $0 \leq a_1 < b_1 < \ldots < a_k < b_k \leq r$ be integers such that
\begin{equation}
\setsuch{\left| \threeind{}{}{i} \mathcal{C} \right|}{i = 1, \ldots, \# \mathcal{C}} = \bigcup_{i=1}^k \{a_i+1, a_i+2, \ldots, b_i\};
\end{equation}
and, for every $i$, let $x_i$ denote the number of symbols in~$\mathcal{C}$ whose absolute value lies in the interval $\{a_i+1, \ldots, b_i\}$ and that have sign~$+1$. Define similarly integers $a'_1 < b'_1 < \ldots < a'_{k'} < b'_{k'}$ and $x'_i$ for $\mathcal{C}'$.

Then the pair of weights $(\nu(\mathcal{C}), \nu(\mathcal{C}'))$ is admissible if and only if all of the following conditions are satisfied:
\begin{enumerate}[label=(A\arabic*)]
\item \label{itm:adm_Young} $\mathcal{C} \succeq^{\lie{g}}_Y \mathcal{C}'$;
\item \label{itm:adm_same_height} $\# \mathcal{C} = \# \mathcal{C}'$;
\item \label{itm:adm_same_decomp} $k = k'$ and, for all $i = 1, \ldots, k$, $a_i = a'_i$ and $b_i = b'_i$;
\item \label{itm:adm_x_condition} for all $i = 1, \ldots, k$, the integers $x_i$ and~$x'_i$ satisfy the following condition:
\[\begin{cases}
\text{no restriction} &\text{if $b_i = r$ and $\lie{g}$ is of type $B_r$}; \\
x_i \equiv x'_i \pmod{2} &\text{if $b_i = r$ and $\lie{g}$ is of type $D_r$}; \\
x_i = x'_i &\text{otherwise.}
\end{cases}\]
\end{enumerate}
\end{proposition}

\begin{remark}
\label{admissible_pairs_of_height_n}
An important particular case of this proposition is the case of two columns of height~$r$. For a column~$\mathcal{C}$ of height~$r$, we necessarily have $k = 1$, $(a_1, b_1) = (0, r)$, and $x_1 = \# \threeind{}{\NN}{} \mathcal{C}$. Now consider a pair $(\mathcal{C}, \mathcal{C}')$ of columns of height~$r$, and let us additionally assume that they satisfy the condition~\ref{itm:adm_Young}. Then conditions~\ref{itm:adm_same_height} and~\ref{itm:adm_same_decomp} are automatically true; as for condition~\ref{itm:adm_x_condition}:
\begin{itemize}
\item if $\lie{g}$ is of type $B_r$, it is also automatically true.
\item if $\lie{g}$ is of type $D_r$, it reduces to $\# \threeind{}{\NN}{} \mathcal{C} \equiv \# \threeind{}{\NN}{} \mathcal{C}' \pmod{2}$, which is a consequence of the parity condition~\eqref{eq:parity_condition}.
\item if $\lie{g}$ is of type $C_r$, it reduces to $\# \threeind{}{\NN}{} \mathcal{C} = \# \threeind{}{\NN}{} \mathcal{C}'$, which together with \ref{itm:adm_Young} forces $\mathcal{C} = \mathcal{C}'$.
\end{itemize}
To summarize, in types $B_r$ and~$D_r$, any pair of columns of height~$r$ that satisfies \ref{itm:adm_Young} is admissible; whereas in type~$C_r$, the only admissible pairs of columns of height~$r$ are of the form $(\mathcal{C}, \mathcal{C})$.
\end{remark}

\begin{proof}
Unpacking the definition (and taking into account Lemma~\ref{Bruhat_order_characterization}), we see that $\nu(\mathcal{C})$ and~$\nu(\mathcal{C}')$ form an admissible pair if and only if one can pass from $\mathcal{C}'$ to~$\mathcal{C}$ by a series of steps of the form $s_\alpha$ for some $\alpha \in \Pi$, where:
\begin{itemize}
\item An operation of the form $s_{e_i - e_{i+1}}$ (for $1 \leq i \leq r-1$) is admissible only if both of the symbols $i$ and~$\overline{i+1}$ occur somewhere; it then replaces them by $i+1$ and~$\overline{i}$ respectively.
\item The operation $s_{e_{r-1} + e_r}$ is admissible only if both of the symbols $r-1$ and~$r$ occur somewhere; it then replaces them by $\overline{r}$ and $\overline{r-1}$ respectively.
\item The operation $s_{e_r}$ is admissible only if the symbol $r$ occurs somewhere; it then replaces it by~$\overline{r}$.
\item The operation $s_{2e_r}$ is never admissible.
\end{itemize}
Clearly each of these operations satisfies the conditions \ref{itm:adm_Young} through \ref{itm:adm_x_condition}, which are transitive; this proves the ``only if'' part.

Conversely, suppose that a pair of columns ($\mathcal{C}$, $\mathcal{C}'$) satisfies these four conditions. Let us then find a path going from $\mathcal{C}'$ to~$\mathcal{C}$ in the Hasse diagram of the order~$\preceq^{\lie{g}}_Y$. Using Lemma~\ref{hasse_edges_description} and the description~\eqref{eq:hasse_step_further_description}, we then easily check that each step of this path is an admissible operation (as described above).
\end{proof}

\subsection{Doubled Young tableaux}
\label{sec:dYt}

We are now ready to define a $\lie{g}$-standard doubled Young tableau, and to show (Proposition~\ref{BCD_path_model_description}) that these tableaux describe the path model with starting path $\pi^+_0(\lambda)$ as given in \eqref{eq:basic_path_definition}. This yields the announced character formula (Proposition~\ref{BCD_character_formula}). We also give a combinatorial characterization (Corollary~\ref{combinatorial_characterization}) of representations $V_\lambda$ satisfying $V^{\lie{l}(\Theta)}_\lambda \neq 0$ for any Levi subalgebra $\lie{l}(\Theta)$ (where $\Theta \subset \Pi$), accompanied by a slightly modified version of this result (Corollary~\ref{combinatorial_characterization_without_sign}) that exploits the outer automorphism of~$D_r$, and will save us some work in the next section.

As we already mentioned in the introduction (Subsubsection~\ref{sec:intro_dYt}), the following object is similar to the object defined in the appendix of \cite{Lit90}, but is not identical: in type~$D_r$, we replace the complicated condition~(3) by the simpler condition~\ref{itm:dsYt_Young}.

\begin{definition}
\label{tableau_definition}
A \emph{$\lie{g}$-standard doubled Young tableau} is a Young tableau $\mathcal{T}$ on the alphabet~$\mathcal{A}_r$ with the following properties:
\begin{enumerate}[label=(H\arabic*)]
\item \label{itm:dsYt_stst} All columns of~$\mathcal{T}$ are strongly standard:
\[\forall j = 1, \ldots, \#_1 \mathcal{T},\quad \threeind{j}{}{} \mathcal{T} \in \mathscr{C}.\]
\item \label{itm:dsYt_Young} The sequence formed by the columns of~$\mathcal{T}$ is Young-nondecreasing if $\lie{g}$ is of type $B_r$ or~$C_r$, Young-nondecreasing with parity if $\lie{g}$ is of type $D_r$:
\[\forall j = 2, \ldots, \#_1 \mathcal{T},\quad \threeind{j-1}{}{} \mathcal{T} \preceq^{\lie{g}}_Y \threeind{j}{}{} \mathcal{T}.\]
\item \label{itm:dsYt_adm} The columns form admissible pairs when grouped two by two starting from the right, \ie for all $j$ such that $1 < j \leq \#_1 \mathcal{T}$ and $j \equiv \#_1 \mathcal{T} \pmod{2}$, the pair $(\threeind{j}{}{} \mathcal{T}, \threeind{j-1}{}{} \mathcal{T})$ is admissible.
\end{enumerate}
In order to verify~\ref{itm:dsYt_adm}, in practice, it suffices to check that every such pair satisfies conditions~\ref{itm:adm_same_decomp} (which implies~\ref{itm:adm_same_height}) and~\ref{itm:adm_x_condition} from Proposition~\ref{admissible_pair_characterization}, since condition~\ref{itm:adm_Young} is already covered by~\ref{itm:dsYt_Young}. Note that the two columns are taken here ``in the wrong order'', because the order of the columns in the Young tableau is backwards compared to the direction of the corresponding path (see~\eqref{eq:tableau_to_path} below).
\end{definition}

\begin{remark}
\label{why_doubled}
It is possible to extend this definition, and all the work done in the previous two subsections, also to the case where $\lie{g} = A_r$, so that Proposition~\ref{BCD_path_model_description} (suitably modified) remains true. However in this case, it turns out that two columns form an admissible pair only if they are equal (essentially because all the fundamental weights of~$A_r$ are minuscule). So an $A_r$-standard doubled Young tableau is just an ordinary semistandard Young tableau, with every column repeated twice. We would then recover the character formula of Proposition~\ref{classical_character_formula} as a particular case of Proposition~\ref{BCD_character_formula}.
\end{remark}

\begin{definition}[Passing from doubled diagrams and tableaux to weights, in types $B_r$, $C_r$ and $D_r$]
\label{BCD_tableaux_and_weights}
Given a $\lie{g}$-standard doubled Young tableau~$\mathcal{T}$, we define:
\begin{hypothenum}
\item the \emph{corresponding path}
\begin{equation}
\label{eq:tableau_to_path}
\pi(\mathcal{T}) := \textstyle \left(\frac{1}{2} \nu(\threeind{N}{}{} \mathcal{T})\right) * \cdots * \left(\frac{1}{2} \nu(\threeind{1}{}{} \mathcal{T})\right)
\end{equation}
(where $N = \#_1 \mathcal{T}$), whose segments are the weights corresponding to the columns of~$\mathcal{T}$ scaled by~$\frac{1}{2}$ and taken in the reverse order.
\item the \emph{total weight} of~$\mathcal{T}$ as
\begin{equation}
\label{eq:BCD_total_weight_definition}
\nu(\mathcal{T}) := \pi(\mathcal{T})(1) = \frac{1}{2} \sum_{j=1}^{\#_1 \mathcal{T}} \nu(\threeind{j}{}{} \mathcal{T}) = \frac{1}{2} \sum_{i, j} \nu(\threeind{j}{}{i} \mathcal{T}).
\end{equation}
Note that, in comparison with Definition~\ref{sln_tableau_concepts}.\ref{itm:total_weight}, there is an extra factor~$\frac{1}{2}$: in fact, it is reasonable to think of doubled Young tableaux as having columns ``of width~$\frac{1}{2}$''.

We say that $\mathcal{T}$ is \emph{null} if $\nu(\mathcal{T}) = 0$.
\item the \emph{sign} $\epsilon$ of~$\mathcal{T}$ as follows:
\begin{itemize}
\item If $\lie{g}$ is of type $B_r$ or $C_r$, we adopt the convention that $\epsilon$ is always equal to~$+1$.
\item If $\lie{g}$ is of type $D_r$ and $\#_r \mathcal{T} > 0$, we take $\epsilon = (-1)^x$, where $x$ is the number of symbols with bars in any column of height~$r$. (Note that the parity condition~\eqref{eq:parity_condition}, whose prerequisite is automatically satisfied with $(i_0, k) = (1, r)$ for columns of height~$r$, ensures that $\epsilon$ does not depend on the choice of the column).
\item If $\lie{g}$ is of type $D_r$ but $\#_r \mathcal{T} = 0$, we adopt the convention that $\epsilon = 0$.
\end{itemize}
\item for $\alpha \in \Pi$ or $\Theta \subset \Pi$, we define \emph{$\alpha$-(co)dominance} and \emph{$\Theta$-(co)dominance} for doubled Young tableaux in the same way as for ordinary Young tableaux (see Definition~\ref{sln_tableau_concepts}.\ref{itm:dominance}). The factor~$\frac{1}{2}$ does not change anything here, as this definition only involves signs of total weights. Clearly a doubled Young tableau $\mathcal{T}$ is $\Theta$-dominant if and only if the path~$\pi(\mathcal{T})$ lies entirely within the $\Theta$-dominant Weyl chamber~$\lie{h}^{\Theta, +}$, as defined in~\eqref{eq:Theta_dom_Weyl_chamber}.
\end{hypothenum}

\begin{remark}
\label{dominance_explicit}
We can of course rephrase this last property in purely Young-tableau-theoretic terms. For example for $\alpha = e_i - e_{i+1}$, a doubled Young tableau is $\alpha$-codominant if and only if it satisfies
\begin{equation}
\label{eq:codominance_BCD}
\forall j \geq 0,\quad \# \threeind{[1,j]}{\overline{i}}{} \mathcal{T} + \threeind{[1,j]}{i+1}{} \mathcal{T} \geq \# \threeind{[1,j]}{\overline{i+1}}{} \mathcal{T} + \threeind{[1,j]}{i}{} \mathcal{T}
\end{equation}
(compare this with \eqref{eq:codominance_Ar}), and similar formulas exist for $\alpha = e_r$, $2e_r$ or $e_{r-1} + e_r$.
\end{remark}

Finally, we introduce a correspondence $\Psi$ between the set~$P \cap \lie{h}^+$ of dominant integral weights $\lambda$ and the set~$\mathscr{M}^{(r)}$ of Young diagrams of height~$r$. It is given by the formula
\begin{equation}
\label{eq:psi_definition}
\left(\#_1 \Psi(\lambda), \ldots, \#_r \Psi(\lambda)\right) := (2\lambda_1, \ldots, 2\lambda_{r-1}, 2|\lambda_r|),
\end{equation}
where, as usual, we decompose $\lambda = \sum_{i=1}^r \lambda_i e_i$.
\end{definition}

\begin{proposition}
\label{BCD_path_model_description}
Let $\lambda \in P \cap \lie{h}^+$ be a dominant integral weight of~$\lie{g}$. Then the set of paths $\pi(\mathcal{T})$, where $\mathcal{T}$ runs over all $\lie{g}$-standard doubled Young tableaux of shape~$\Psi(\lambda)$ with the same sign as~$\lambda_r$, is equal to the path model $B_{\pi^+_0(\lambda)}$.
\end{proposition}
Recall \eqref{eq:basic_path_definition} that the starting path $\pi^+_0(\lambda)$ we are using here is defined as
\[\pi^+_0(\lambda) := (\lambda_2 - \lambda_1) c_1 * \cdots * (\lambda_r - \lambda_{r-1}) c_{r-1} * |\lambda_r| c_r^{\sgn(\lambda_r)}.\]
\begin{proof}
Let $\mathcal{T}$ be any doubled Young tableau, $\pi(\mathcal{T})$ the corresponding path. Our goal is to apply Corollary~\ref{long_path_model_characterization}. Note that:
\begin{itemize}
\item It is straightforward to check that $\mathcal{T}$ has shape~$\Psi(\lambda)$ and sign~$\sgn(\lambda_r)$ if and only if $\pi(\mathcal{T})$ has multishape $\pi^+_0(\lambda)$. %It remains to check that such a tableau is $\lie{g}$-standard if and only if the corresponding path satisfies the conditions of Corollary~\ref{long_path_model_characterization}.

\item Assuming that this is the case, it follows from Proposition~\ref{Bruhat_vs_Young} that the columns of~$\mathcal{T}$ form a nondecreasing sequence for the order~$\preceq_Y^{\lie{g}}$ if and only if the path $\pi(\mathcal{T})$ is Bruhat-nonincreasing.
\end{itemize}
On the other hand, from the integrality of~$\lambda$, it follows that all the coefficients in the decomposition~\eqref{eq:basic_path_definition} are integer, except possibly $|\lambda_r|$ which can be half-integer when $\lie{g}$ is of type $B_r$ or $D_r$. So let us decompose $\pi^+_0(\lambda)$ into a concatenation that first involves $\lfloor \lambda_1 \rfloor$~segments chosen among elements of the set~$\mathcal{X}^+ = \mathcal{X} \cap \lie{h}^+$ (recall~\eqref{eq:model_weights}) and then possibly ends with a segment equal to $\frac{1}{2} c_r^\pm$; and then apply Corollary~\ref{long_path_model_characterization} to this decomposition.

We now conclude by Corollary~\ref{short_path_model_characterization}. Indeed, we have already remarked that all weights $\nu \in \mathcal{X}^+$ satisfy $\max_{\alpha \in \Delta} |\langle \nu, \alpha^\vee \rangle| \leq 2$; as for the weights $\frac{1}{2} c_r$ (if $\lie{g}$ is of type $B_r$) and $\frac{1}{2} c_r^\pm$ (if $\lie{g}$ is of type $D_r$), they are minuscule.
\end{proof}

\begin{remark}
\begin{enumerate}
\item If $\lie{g}$ is of type $B_r$ or $D_r$, then the way we have cut the subpath $|\lambda_r| c_r^{\pm}$ into segments is somewhat arbitrary: we could have just as well decomposed it into any other combination of segments $\frac{1}{2} c_r^{\pm}$ and $c_r^{\pm}$, or, for that matter, exclusively into segments $\frac{1}{2} c_r^{\pm}$. We could have used any of these decompositions to write an alternative definition of a $\lie{g}$-standard doubled Young tableau. But all of these definitions would have been equivalent, thanks to Remark~\ref{admissible_pairs_of_height_n}: indeed, for pairs of columns of height~$r$ (no matter the parity), the condition~\ref{itm:dsYt_adm} automatically follows from the first two conditions.
\item If $\lie{g}$ is of type $C_r$, then the weight~$c_r$ is in fact minuscule. This is also consistent with Remark~\ref{admissible_pairs_of_height_n}: in type~$C_r$, two columns of height~$r$ form an admissible pair if and only if they coincide. 
\end{enumerate}
\end{remark}

\begin{remark}
\label{no_unpaired_column}
From the characterization of~$P$ given in Table~\ref{tab:root_lattice_congruences}, one can easily see that \emph{every} $\lie{g}$-standard doubled Young tableau~$\mathcal{T}$ is in fact of shape $\Psi(\lambda)$ for some integral weight $\lambda \in P \cap \lie{h}^+$. In particular its weight $\nu(\mathcal{T})$ is then one of the weights of the representation~$V_\lambda$, hence the difference $\lambda - \nu(\mathcal{T})$ lies in the root lattice~$Q$.

If the tableau~$\mathcal{T}$ is null (which will always be the case in the sequel), this condition reduces to $\lambda \in Q$, which implies (we refer once again to Table~\ref{tab:root_lattice_congruences}) that all the $\lambda_i$ are integer. Hence all the row lengths of~$\mathcal{T}$ are even; and in condition~\ref{itm:dsYt_adm}, the admissible pairs form a partition of all columns of~$\mathcal{T}$, without any unpaired column.
\end{remark}

By combining this with the Littelmann character formula (Proposition~\ref{Littelmann_character_formula}), we obtain the following, purely combinatorial character formula.

\begin{proposition}[Character formula with $B_r$, $C_r$, $D_r$-standard doubled Young tableaux]
\label{BCD_character_formula}
Let $\lambda \in P \cap \lie{h}^+$ be a dominant integral weight of~$\lie{g}$. Then the character of the representation with highest weight~$\lambda$ is given by:
\[\charact(V_\lambda) = \sum_{\mathcal{T}} e^{\nu(\mathcal{T})},\]
where $\mathcal{T}$ runs over all $\lie{g}$-standard doubled Young tableaux of shape~$\Psi(\lambda)$ that have the same sign as~$\lambda_r$. \end{proposition}
This is of course very similar to Proposition~\ref{classical_character_formula}; keep in mind, though, that the definition of~$\nu(\mathcal{T})$ has now slightly changed \eqref{eq:BCD_total_weight_definition}.

We can also combine this with Corollary~\ref{Vl_in_terms_of_paths} to obtain a purely combinatorial characterization of representations having $\lie{l}$-invariant vectors:
\begin{corollary}
\label{combinatorial_characterization}
Let $\lambda \in P \cap \lie{h}^+$ and $\Theta \subset \Pi$. Then $V_\lambda^{\lie{l}(\Theta)} \neq 0$ if and only if there exists a doubled Young tableau~$\mathcal{T}$ satisfying the following seven conditions:
\begin{enumerate}[label = (H\arabic*), start=4]
\item[\ref{itm:dsYt_stst}] through \ref{itm:dsYt_adm} as in Definition~\ref{tableau_definition}, \ie the tableau~$\mathcal{T}$ is $\lie{g}$-standard.
\item \label{itm:null_total_weight} The tableau is null: $\nu(\mathcal{T}) = 0$.
\item \label{itm:shape} The tableau has shape $\Psi(\lambda) = (2\lambda_1, \ldots, 2\lambda_{r-1}, 2|\lambda_r|)$.
\item \label{itm:sign} The tableau has the same sign as~$\lambda_r$ (if $\lie{g}$ is of type~$D_r$).
\item \label{itm:l-codominant} The tableau is $\Theta$-codominant (recall that, for null tableaux, codominance is equivalent to dominance).
\end{enumerate}
\end{corollary}
We end this section with an additional small simplification: we can in fact get rid of condition~\ref{itm:sign}. Indeed, when we will study the properties satisfied by such tableaux, we will simply make no use of this property. When we will try to construct such tableaux, we will get around having to check this condition by way of the following (obvious) remark. Let $\sigma$ denote the outer automorphism of~$D_r$: it acts on~$\lie{h}$ by
\[\sigma(e_i) = \begin{cases}
e_i &\text{if } i < r; \\
-e_r &\text{if } i = r,
\end{cases}\]
and correspondingly on~$\mathcal{A}_r$ by exchanging $r$ and~$\overline{r}$.
\begin{remark}
\label{Dn_automorphism}
The set of $D_r$-standard doubled Young tableaux is invariant by~$\sigma$. If $\mathcal{T}$ is such a tableau, then:
\begin{itemize}
\item $\sigma(\mathcal{T})$ has the same shape, but opposite sign compared to $\mathcal{T}$;
\item $\sigma(\mathcal{T})$ is null if and only if $\mathcal{T}$ is null;
\item for all $\alpha \in \Pi$, $\sigma(\mathcal{T})$ is $\alpha$-(co)dominant if and only if $\mathcal{T}$ is $\sigma(\alpha)$-(co)dominant.
\end{itemize}
%
%\[\fundef{\sigma:}{\lie{h}}{\lie{h}}
%{e_i}{\begin{cases}e_i &\text{for } i = 1, \ldots, r-1\\ -e_r &\text{for } i = r.\end{cases}}\] 
\end{remark}

This gives us the following variation on the ``if'' part of Corollary~\ref{combinatorial_characterization}, with condition~\ref{itm:sign} gone, at the expense of replacing $\Theta$ by a slightly larger set.
\newcommand{\labellcodominantsym}{(H7')}
\begin{corollary}
\label{combinatorial_characterization_without_sign}
Let $\lambda \in P \cap \lie{h}^+$ and $\Theta \subset \Pi$. Suppose that there exists a doubled Young tableau~$\mathcal{T}$ satisfying the conditions:
\begin{enumerate}
\item[\ref{itm:dsYt_stst}] through~\ref{itm:shape} as in Corollary~\ref{combinatorial_characterization}.
\item[\labellcodominantsym] \label{itm:l-codominant-sym} The tableau $\mathcal{T}$ is $(\Theta \cup \sigma(\Theta))$-codominant (with the convention $\sigma = \Id$ if $\lie{g}$ is of type $B_r$ or~$C_r$).
\end{enumerate}
Then we have both $V_\lambda^{\lie{l}(\Theta)} \neq 0$ and $V_{\sigma{\lambda}}^{\lie{l}(\Theta)} \neq 0$.
\end{corollary}

\section{Types $B_r$, $C_r$ and~$D_r$: the proof}
\label{sec:BCD_proof}

In this section, we prove the Main Theorem, \ie the equality $\mathcal{M}_{\linv} = \mathcal{M}_{\theor}$, for $\lie{g}$ of types $B_r$ ($r \geq 1$), $C_r$ ($r \geq 1$) and $D_r$ ($r \geq 3$). Thanks to the work done in the previous section, it suffices, in order to do this, to prove the following two things:
\begin{itemize}
\item that every doubled Young tableau $\mathcal{T}$ that satisfies conditions \ref{itm:dsYt_stst}--\ref{itm:shape} and \ref{itm:l-codominant} above has a shape that satisfies the conditions from Table~\ref{tab:conditions_for_classical_algebras};
\item that every doubled Young diagram whose shape satisfies these conditions admits a filling that satisfies conditions \ref{itm:dsYt_stst}--\ref{itm:shape} and \labellcodominantsym{} above.
\end{itemize}
In Subsection~\ref{sec:linv_in_theor} we accomplish the first task. In subsection~\ref{sec:theor_in_linv}, we accomplish the second task for diagrams corresponding to primitive elements of~$\mathcal{M}_{\theor}$; and we conclude by additivity.

\subsection{The inclusion $\mathcal{M}_{\linv} \subset \mathcal{M}_{\theor}$}
\label{sec:linv_in_theor}

Let $\lie{g}_\RR$ be some real form of~$\lie{g} = B_r$, $C_r$ or $D_r$, and let $\lambda$ be a dominant integral weight such that the doubled Young diagram~$\Psi(\lambda)$ admits a $\Theta(\lie{g}_\RR)$-codominant null $\lie{g}$-standard filling. We must prove that $\lambda$ satisfies the corresponding condition from Table~\ref{tab:conditions_for_classical_algebras}.

From \cite{OV90}, Reference Chapter, Table~9, we obtain the values of the sets $\Theta(\lie{g}_\RR)$ for all such real forms $\lie{g}_\RR$; for quicker reference, we have reproduced them here in Table~\ref{tab:Theta_classical}.

Note that each of these sets contains the ``tail'' $\Pi_{[x+1, r]}$ (for some value of~$x$) of the Dynkin diagram. It turns out that the other simple roots contained in~$\Theta(\lie{g}_\RR)$ will not matter (except for two low-rank cases, that we treat by using exceptional isomorphisms). The bulk of this subsection is thus devoted to establishing a few inequalities (Proposition~\ref{height_restrictions} and Corollary~\ref{height_limitation}) satisfied by $\Pi_{[x+1, r]}$-codominant null $\lie{g}$-standard doubled Young tableaux. At the end of this subsection, we put the pieces together.

\begin{table}[h]
  \caption{\label{tab:Theta_classical}Values of $\Theta(\lie{g}_\RR)$ for all real forms of simple Lie algebras of types $B_r$, $C_r$ and~$D_r$. The (hopefully transparent) notations $\Pi_{[x,y]}$ and $\Pi_{\operatorname{odd}}$ are defined in \eqref{eq:Theta_interval_definition} and~\eqref{eq:Theta_odd_definition}.}
  \centering\bigskip
  \begin{tabular}[t]{llll}
    $\lie{g}$ & $\lie{g}_\RR$ & Parameter range & $\Theta(\lie{g}_\RR)$ \\
    \midrule
    $\underset{r \geq 1}{B_r}$
    & $\lie{so}(p,2r+1-p)$ & $0 \leq p \leq r$ & $\Pi_{[p+1,r]}$ \\
    \midrule
    \multirow{2}{*}{$\underset{r \geq 1}{C_r}$}
    & $\lie{sp}_\twice (p, r-p)$ & $0 \leq p \leq \frac{r}{2}$ & $\Pi_{\operatorname{odd}} \cup \Pi_{[2p+1,r]}$ \\
    & $\lie{sp}_{\twice r}(\RR)$ & & $\emptyset$ \\
    \midrule
    \multirow{3}{*}{$\underset{r \geq 3}{D_r}$}
    & $\lie{so}(p,2r-p)$ & $\begin{cases}0 \leq p \leq r \\ p \neq r-1\end{cases}$ & $\Pi_{[p+1,r]}$ \\
    & $\lie{so}(r-1,r+1)$ & & $\emptyset$ \\
    & $\lie{so}^*(2r)$ & & $\Pi_{\operatorname{odd}} \setminus \{\alpha_r\}$ \\
    \bottomrule
  \end{tabular}
\end{table}

\begin{proposition}
\label{height_restrictions}
Suppose that $\lie{g}$ is of type $B_r$, $C_r$ or~$D_r$; let $\mathcal{T}$ be any $\lie{g}$-standard doubled Young tableau. Let $h := \#^1 \mathcal{T}$~be the height of~$\mathcal{T}$, and let $t := \max_{i, j} \left| \threeind{j}{}{i} \mathcal{T} \right|$~be the largest number such that either $t$ or $\overline{t}$ appears somewhere in~$\mathcal{T}$; these numbers satisfy
\begin{equation}
\label{eq:height_upper_bound}
h \leq t.
\end{equation}
Moreover, for every integer~$x$ satisfying $0 \leq x \leq r$, we have the following inequalities.
\begin{hypothenum}
\item \label{itm:general_case} If $\mathcal{T}$ is $\Pi_{[x+1,r]}$-codominant and null, then we have:
\begin{equation}
\label{eq:height_lower_bound}
h \geq 2(t-x).
\end{equation}
\item \label{itm:odd_size} If moreover the (automatically integer) number $\frac{1}{2} \# \mathcal{T}$ is odd, then necessarily $\lie{g}$ is of type~$B_r$, $t=r$, and the inequality~\eqref{eq:height_lower_bound} becomes strict, \ie
\begin{equation}
\label{eq:height_better_lower_bound}
h \geq 2(r-x)+1.
\end{equation}
\end{hypothenum}
\end{proposition}
By rearranging \eqref{eq:height_lower_bound} as $t \leq \frac{h}{2} + x$ and combining it with \eqref{eq:height_upper_bound}, we also obtain the following consequence.
\begin{corollary}
\label{height_limitation}
Under the same assumptions on $\lie{g}$ and~$x$, every $\Pi_{[x+1,r]}$-codominant, null, $\lie{g}$-standard doubled Young tableau $\mathcal{T}$ has height at most~$2x$.
\end{corollary}

The proof of Proposition~\ref{height_restrictions} relies on the following lemma; in order to formulate it, we first need to introduce a notation.
\begin{definition}
Given a Young tableau~$\mathcal{T}$ and a symbol $s$, we define the numbers
\[
\operatorname{mincol}_{\mathcal{T}}(s) := \min \setsuch{j}{\exists i,\quad \threeind{j}{}{i} \mathcal{T} = s}
\]
and
\[
\operatorname{maxcol}_{\mathcal{T}}(s) := \max \setsuch{j}{\exists i,\quad \threeind{j}{}{i} \mathcal{T} = s},
\]
with the usual conventions $\min \emptyset = +\infty$ and $\max \emptyset = -\infty$.
\end{definition}

\begin{lemma}
\label{mincol_maxcol_inequalities}
Let $\lie{g}$, $x$ and $\mathcal{T}$ be as in Proposition~\ref{height_restrictions}.\ref{itm:general_case}. Then the inequalities
{\refstepcounter{equation}\label{eq:mincol_min}}
{\def\tagargument{$s$} %Here I had to mess around with tags manually. This is really dirty, but I couldn't make it work otherwise.
 \usetagform{argumenttag}
\begin{align}
\label{eq:mincol_plat} &\leq \operatorname{mincol}_{\mathcal{T}}(\overline{s+1}) \tag{\theequation{}a} \\
\label{eq:mincol_croise} \operatorname{mincol}_{\mathcal{T}}(\overline{s})
  \smash{\left.\begin{aligned}\vphantom{\overline{s}} \\ \vphantom{s+1} \\ \vphantom{s}\end{aligned}\right\{}
&\leq \operatorname{mincol}_{\mathcal{T}}(s+1) \tag{\theequation{}b} \\
\label{eq:mincol_embrasse} &\leq \operatorname{mincol}_{\mathcal{T}}(s) \tag{\theequation{}c} 
\end{align}
and
{\refstepcounter{equation}\label{eq:maxcol_max}}
\begin{align}
\label{eq:maxcol_embrasse} \operatorname{maxcol}_{\mathcal{T}}(\overline{s}) \leq& \tag{\theequation{}a} \\
\label{eq:maxcol_croise} \operatorname{maxcol}_{\mathcal{T}}(\overline{s+1}) \leq&
\smash{\left.\begin{aligned}\vphantom{\overline{s}} \\ \vphantom{s+1} \\ \vphantom{s}\end{aligned}\right\}}
  \operatorname{maxcol}_{\mathcal{T}}(s). \tag{\theequation{}b} \\
\label{eq:maxcol_plat} \operatorname{maxcol}_{\mathcal{T}}(s+1) \leq& \tag{\theequation{}c}
\end{align}}
hold for every $s$ such that $x < s \leq r - \mathbbm{1}_D$, where we set
\[
\mathbbm{1}_D := \begin{cases}
0 &\text{ if } \lie{g} = B_r \text{ or } C_r; \\
1 &\text{ if } \lie{g} = D_r.
\end{cases}
\]

%\begin{hypothenum} %TODO Maybe switch the order of the two items.
%\item \label{itm:minmaxcol_s_s+1} For each $s > x$, we have the inequalities
%\[
%\begin{cases}
%\operatorname{mincol}_{\mathcal{T}}(\overline{s}) \leq \operatorname{mincol}_{\mathcal{T}}(\overline{s+1}); \\
%\operatorname{maxcol}_{\mathcal{T}}(s+1) \leq \operatorname{maxcol}_{\mathcal{T}}(s).
%\end{cases}
%\]
%\item \label{itm:minmaxcol_s} For each $s > x$ except $s = r$ when $\lie{g}$ is of type $D_r$, we have the inequalities
%\[
%\begin{cases}
%\operatorname{mincol}_{\mathcal{T}}(\overline{s}) \leq \operatorname{mincol}_{\mathcal{T}}(s); \\
%\operatorname{maxcol}_{\mathcal{T}}(\overline{s}) \leq \operatorname{maxcol}_{\mathcal{T}}(s).
%\end{cases}
%\] %TODO Need actual equation numbering. The proof of the lemma would really become less awkward with the convention that "(?.?.r)" means equation (?.?) applied to value r. 
%\end{hypothenum}
\end{lemma}

The proof of this lemma relies on the following obvious remark.

\begin{remark}
\label{first_and_last_column}
In a null $\Theta$-codominant $\lie{g}$-standard doubled Young tableau~$\mathcal{T}$, the following statements hold for each $\alpha \in \Theta$:
\begin{hypothenum}
\item \label{itm:first_nonzero_is_positive} the first column of~$\mathcal{T}$ with a nonzero total $\alpha$-height has negative total $\alpha$-height;
\item \label{itm:last_nonzero_is_negative} the last column of~$\mathcal{T}$ with a nonzero total $\alpha$-height has positive total $\alpha$-height,
\end{hypothenum}
where we define the \emph{$\alpha$-height} of a column~$\mathcal{C}$ as the number $\langle \nu(\mathcal{C}), \alpha^\vee \rangle$.
\end{remark}

\begin{proof}[Proof of Lemma~\ref{mincol_maxcol_inequalities}]
We start by proving the inequalities \eqrefWithArgument{eq:mincol_min}{$s$}, for all $s$ within the given bounds. Assume that $\mathcal{T}$ contains at least one of the symbols $s$, $s+1$, $\overline{s+1}$ or~$\overline{s}$ (otherwise the inequalities are vacuously true), and let $j$ be the index of the first column where one of these four symbols occurs. Then the inequalities~\eqrefWithArgument{eq:mincol_min}{$s$} are equivalent to the statement that the $j$-th column of~$\mathcal{T}$ contains the symbol~$\overline{s}$.

We shall prove it by descending induction on~$s$.
\begin{itemize}
\item Let us first prove it for $s = r - \mathbbm{1}_D$. We distinguish two cases:
\begin{itemize}
\item Assume first that $\lie{g}$ is of type $B_r$ or~$C_r$, so that $s = r$, and $\alpha_r$~is equal to (possibly the double of)~$e_r$. In particular the symbols $r+1$ and~$\overline{r+1}$ do not occur anywhere in~$\mathcal{T}$, and we may ignore them. By Remark~\ref{first_and_last_column}.\ref{itm:first_nonzero_is_positive} applied to~$\alpha_r$, it then follows that the $j$-th column of~$\mathcal{T}$ contains~$\overline{r}$, as required.
\item Assume now that $\lie{g}$ is of type $D_r$, so that $s = r-1$. Both $\alpha_{r-1} = e_{r-1} - e_r$ and $\alpha_r = e_{r-1} + e_r$ lie in~$\Pi_{[x+1,r]}$; by applying Remark~\ref{first_and_last_column}.\ref{itm:first_nonzero_is_positive} to these two roots, we respectively obtain that:
\begin{itemize}
\item the $j$-th column must contain either $\overline{r-1}$ or~$r$;
\item the $j$-th column must contain either $\overline{r-1}$ or~$\overline{r}$.
\end{itemize}
However the $j$-th column is strongly standard: it cannot contain $r$ and~$\overline{r}$ simultaneously. Hence it contains $\overline{r-1}$ as required.
\end{itemize}
\item Now let $s$ be such that $x < s < r - \mathbbm{1}_D$, and assume that the three inequalities \eqrefWithArgument{eq:mincol_min}{$s+1$} are true. Since $s < r - \mathbbm{1}_D$, we have $\alpha_s = e_s - e_{s+1}$; so Remark~\ref{first_and_last_column}.\ref{itm:first_nonzero_is_positive} tells us that the $j$-th column contains either $s+1$ or~$\overline{s}$. If it contained $s+1$, then \eqrefWithArgument{eq:mincol_embrasse}{$s+1$} would force it to also contain $\overline{s+1}$, which is a contradiction; so it has to contain~$\overline{s}$ as required.
\end{itemize}

Similarly, the inequalities \eqrefWithArgument{eq:maxcol_max}{$s$} are equivalent to the statement that the last column of~$\mathcal{T}$ that contains one of the symbols $s$, $s+1$, $\overline{s+1}$ or~$\overline{s}$ must contain the symbol~$s$; and we can similarly prove them by descending induction on~$s$, using now Remark~\ref{first_and_last_column}.\ref{itm:last_nonzero_is_negative}.
\end{proof}

We are now ready to prove Proposition~\ref{height_restrictions}.

\begin{proof}[Proof of Proposition~\ref{height_restrictions}]
First of all, the inequality~\eqref{eq:height_upper_bound} immediately follows from the fact that the symbols occurring in the first column of~$\mathcal{T}$, which has height~$h$, must have pairwise distinct absolute values.
\begin{hypothenum}
\item Assume now that $\mathcal{T}$ is $\Pi_{[x+1,r]}$-codominant and null. Clearly \eqref{eq:height_lower_bound} holds if $t \leq x$; so assume that $t > x$. We introduce the numbers
\begin{equation}
\label{eq:js_definition}
\forall s = x+1,\; \ldots,\; t,\quad
\begin{cases}
j_{\overline{s}} := \operatorname{mincol}_{\mathcal{T}}(\overline{s}); \\
j_s := \operatorname{maxcol}_{\mathcal{T}}(s).
\end{cases}
\end{equation}
Since $\mathcal{T}$~is null, in fact, both $t$ and~$\overline{t}$ must appear somewhere in~$\mathcal{T}$. Now we apply Lemma~\ref{mincol_maxcol_inequalities}: from the inequalities~\eqrefWithArgument{eq:mincol_plat}{$s$} for $s$ running from $x+1$ to~$t-1$, it follows that
\begin{equation}
\label{eq:first_half_increasing}
j_{\overline{x+1}} \leq \cdots \leq j_{\overline{t-1}} \leq j_{\overline{t}} < +\infty;
\end{equation}
and from the inequalities~\eqrefWithArgument{eq:maxcol_plat}{$s$} for $s$ running from $x+1$ to~$t-1$, it follows that
\begin{equation}
\label{eq:second_half_increasing}
-\infty < j_t \leq j_{t-1} \leq \cdots \leq j_{x+1}.
\end{equation}
In particular, for each $s = x+1, \ldots, t$, the value $j_s$ (resp.~$j_{\overline{s}}$) is finite, \ie is the index of an actual column of~$\mathcal{T}$ that contains the symbol $s$ (resp.~$\overline{s}$). So let $i_s$ (resp.~$i_{\overline{s}}$) denote the (unique) index such that $\threeind{j_s}{}{i_s}\mathcal{T} = s$ (resp. $\threeind{j_{\overline{s}}}{}{i_{\overline{s}}}\mathcal{T} = {\overline{s}}$), for every such~$s$.

We will now establish some inequalities concerning the numbers $i_s$: either \eqref{eq:distinct_rows} or~\eqref{eq:distinct_rows_switched}, depending on the order between $j_{\overline{t}}$ and~$j_t$.
\begin{itemize}
\item Assume first that $t \leq r - \mathbbm{1}_D$. Then we have:
\begin{align*}
j_{\overline{t}} &= \operatorname{mincol}_{\mathcal{T}}(\overline{t}) \\
                 &\leq \operatorname{mincol}_{\mathcal{T}}(t) &\text{ by \eqrefWithArgument{eq:mincol_embrasse}{$t$}} \\
                 &\leq \operatorname{maxcol}_{\mathcal{T}}(t) &\text{ since $t$ actually occurs in $\mathcal{T}$} \\
                 &= j_t,
\end{align*}
so that we can combine \eqref{eq:first_half_increasing} and~\eqref{eq:second_half_increasing} into
\[
j_{\overline{x+1}} \leq \cdots \leq j_{\overline{t}} \leq j_t \leq \cdots \leq j_{x+1}.
\]
Now whenever we have two pairs of indexes $(i,j)$ and $(i',j')$ such that $j \leq j'$ but $\threeind{j}{}{i} \mathcal{T} \succ_\mathcal{A} \threeind{j'}{}{i'} \mathcal{T}$, we must necessarily have $i > i'$ (else this would contradict $\preceq_\mathcal{A}$-semistandardness of~$\mathcal{T}$, which holds by assumption when $\lie{g}$ is of type $B_r$ or~$C_r$ and by Remark~\ref{semistandard_with_parity_is_semistandard} when $\lie{g}$ is of type $D_r$). We conclude that
\begin{equation}
\label{eq:distinct_rows}
i_{\overline{x+1}} > \cdots > i_{\overline{t}} > i_t > \cdots > i_{x+1}
\end{equation}
as desired.
\item Assume now that $t > r - \mathbbm{1}_D$, which means that $\lie{g}$ is of type $D_r$ and $t = r$. If we still have $j_{\overline{r}} \leq j_r$, then the same proof works, and \eqref{eq:distinct_rows} still holds. So assume that $j_r \leq j_{\overline{r}}$; we then have
\[
j_{\overline{x+1}} \leq \cdots \leq j_{\overline{r-1}} \leq j_r \leq j_{\overline{r}} \leq j_{r-1} \leq \cdots \leq j_{x+1}
\]
(using now \eqrefWithArgument{eq:maxcol_croise}{$r-1$} and \eqrefWithArgument{eq:mincol_croise}{$r-1$} in addition to the chains of inequalities \eqref{eq:first_half_increasing} and~\eqref{eq:second_half_increasing}). We then claim that we have
\begin{equation}
\label{eq:distinct_rows_switched}
i_{\overline{x+1}} > \cdots > i_{\overline{r-1}} > i_r > i_{\overline{r}} > i_{r-1} > \cdots > i_{x+1}.
\end{equation}
Indeed, all of the inequalities except for the middle one once again follow from the semistandardness of~$\mathcal{T}$; and the inequality $i_r > i_{\overline{r}}$ follows from the semistandardness of~$\mathcal{T}$ for the alternative order~$\preceq''_{\mathcal{A}}$, as given in Remark~\ref{semistandard_with_parity_is_semistandard}.
\end{itemize}
No matter which one of \eqref{eq:distinct_rows} or~\eqref{eq:distinct_rows_switched} is true, we obtain that the integers $i_s$ are all distinct. Since there are $2(t-x)$ of them, and (being row numbers) they all lie between $1$ and~$h$, the inequality~\eqref{eq:height_lower_bound} follows.

\item Assume now that additionally $\frac{1}{2} \# \mathcal{T}$ is odd.
\begin{itemize}
\item Since $\mathcal{T}$ is null, we have $\# \threeind{}{\NN}{} \mathcal{T} = \# \threeind{}{\overline{\NN}}{} \mathcal{T}$, and (recall Remark~\ref{no_unpaired_column}) the width $\#_1 \mathcal{T} = 2\lambda_1$ of $\mathcal{T}$ is even. Hence
\begin{align*}
\frac{1}{2} \# \mathcal{T}
  &= \# \threeind{}{\NN}{} \mathcal{T} \\
  &= \sum_{j=1}^{2\lambda_1} \# \threeind{j}{\NN}{} \mathcal{T} \\
  &= \sum_{j'=1}^{\lambda_1} \left( \# \threeind{2j'-1}{\NN}{} \mathcal{T} + \# \threeind{2j'}{\NN}{} \mathcal{T} \right),
\end{align*}
so this last sum is odd. This means that there exists at least one index, let us call it~$j'_0$, such that
\begin{equation}
\label{eq:uneven_twins}
\# \threeind{2j'_0-1}{\NN}{} \mathcal{T} \not\equiv \# \threeind{2j'_0}{\NN}{} \mathcal{T} \pmod{2}.
\end{equation}

Now by assumption, we know that the $(2j'_0-1, 2j'_0)$-th pair of columns is admissible; in particular it satisfies condition~\ref{itm:adm_x_condition} from Proposition~\ref{admissible_pair_characterization}. Observe that, in the notations of that proposition, we have
\[\# \threeind{}{\NN}{} \mathcal{C} = \sum_{i=1}^k x_i\]
for every strongly standard column~$\mathcal{C}$. It follows that the inequality~\eqref{eq:uneven_twins} can only happen if $\lie{g}$~is of type~$B_r$ and both of the relevant columns have $b_i = r$ for some~$i$, \ie both of them contain a symbol with absolute value~$r$. In particular this means that $t = r$.

\item It remains to prove~\eqref{eq:height_better_lower_bound}. If $x = r$, then we only need to prove that $h \geq 1$ \ie that $\mathcal{T}$ is nonempty, which is obviously true (formally we can use \eqref{eq:uneven_twins} to say that $h \geq \# \threeind{2j'_0-1}{}{} \mathcal{T} \geq \# \threeind{2j'_0-1}{\NN}{} \mathcal{T} > \# \threeind{2j'_0}{\NN}{} \mathcal{T} \geq 0$). So assume that $x < r$, which also means that $x < t$, so that the inequalities \eqref{eq:first_half_increasing}, \eqref{eq:second_half_increasing} and~\eqref{eq:distinct_rows} from part~\ref{itm:general_case} still hold.

We have already observed that both the $2j'_0-1$-th column and the $2j'_0$-th column contain either $r$ or $\overline{r}$; in other terms
\[j_{\overline{r}} \leq 2j'_0-1 < 2j'_0 \leq j_r.\]
Now consider the function $j \mapsto \# \threeind{j}{\NN}{} \mathcal{T}$. Since $\mathcal{T}$ is a semistandard Young tableau, this function is nonincreasing; hence we have
\begin{equation}
\label{eq:barless_heights_inequality}
     \# \threeind{j_{\overline{r}}}{\NN}{} \mathcal{T}
\geq \# \threeind{2j'_0-1}{\NN}{} \mathcal{T}
>    \# \threeind{2j'_0}{\NN}{} \mathcal{T}
\geq \# \threeind{j_r}{\NN}{} \mathcal{T}
\end{equation}
(the middle inequality is strict because of \eqref{eq:uneven_twins}). On the other hand, by construction, we know that $\threeind{j_r}{}{i_r} \mathcal{T}$ (resp. $\threeind{j_{\overline{r}}}{}{i_{\overline{r}}} \mathcal{T}$) is equal to~$r$ (resp. to~$\overline{r}$), which, for the order~$\preceq_\mathcal{A}$, is the last symbol without a bar (resp. the first symbol with a bar). Hence we have, by column-standardness:
\begin{equation}
\label{eq:barless_heights_and_in}
\begin{cases}
i_{\overline{r}} = \# \threeind{j_{\overline{r}}}{\NN}{} \mathcal{T} + 1; \\
i_r = \# \threeind{j_r}{\NN}{} \mathcal{T}.
\end{cases}
\end{equation}
Plugging these identities into~\eqref{eq:barless_heights_inequality}, we obtain
\begin{equation}
i_{\overline{r}} > i_{\overline{r}}-1 > i_r.
\end{equation}
This allows us to insert an extra step in the middle of the chain of inequalities~\eqref{eq:distinct_rows} (remember that $t = r$), and thus to improve~\eqref{eq:height_lower_bound} to~\eqref{eq:height_better_lower_bound}. \qedhere
\end{itemize}
\end{hypothenum}
\end{proof}

\begin{proof}[Proof that $\mathcal{M}_{\linv} \subset \mathcal{M}_{\theor}$ for $\lie{g}$ of types $B_{r \geq 1}$, $C_{r \geq 1}$ or $D_{r \geq 3}$.]
Let $\lambda \in \mathcal{M}_{\linv} \subset P \cap \lie{h}^+$. Then by Proposition~\ref{basic_results}~\ref{itm:nontriv_implies_radical}, we get that $\lambda \in Q$. For $\lie{g}_\RR = \lie{sp}_{2\cdot r}(\RR)$ (for any rank~$r$) and $\lie{g}_\RR = \lie{so}^*(2r)$ with $r \geq 4$, this is all there is to check.

In the remaining cases, Corollary~\ref{combinatorial_characterization} tells us that the doubled Young diagram $\Psi(\lambda)$ admits a $\Theta(\lie{g}_\RR)$-codominant null $\lie{g}$-standard filling, with $\Theta(\lie{g}_\RR)$ given in Table~\ref{tab:Theta_classical}. To satisfy Table~\ref{tab:conditions_for_classical_algebras}, we need to check the following conditions.
\begin{itemize}
\item The condition $\lambda_{2p+1} = 0$, or equivalently $\#^1 \Psi(\lambda) \leq 2p$, for $\lie{g}_\RR = \lie{so}(p,q)$, no matter the parity of $p+q$. When $q \neq p+2$, we have $\Theta(\lie{g}_\RR) = \Pi_{[p+1,r]}$ and this condition follows from Corollary~\ref{height_limitation}. For $\lie{g}_\RR = \lie{so}(p,p+2)$, this condition is tautologically true, since $2p+1 = 2r-1 > r$.
\item The condition $\lambda_{4p+1} = 0$, or equivalently $\#^1 \Psi(\lambda) \leq 4p$, for $\lie{g}_\RR = \lie{sp}_{2\cdot}(p,q)$ (a real form of $C_{p+q}$). We then have $\Theta(\lie{g}_\RR) \supset \Pi_{[2p+1,r]}$, and similarly this follows from Corollary~\ref{height_limitation}. (Note that for $\lie{sp}_{2\cdot}(1,1)$ this condition still holds, but was omitted from Table~\ref{tab:conditions_for_classical_algebras} since it is trivial.)
\item The inequality $\lambda_{2r-2p+1} > 0$, or equivalently $\#^1 \Psi(\lambda) \geq 2r-2p+1$, if the sum $\sum_{i=1}^r \lambda_i = \frac{1}{2} \# \Psi(\lambda)$ is odd, for $\lie{g}_\RR = \lie{so}(p,q)$ with $p+q$ odd. This is given by Proposition~\ref{height_restrictions}~\ref{itm:odd_size}.
\item The congruence $\lambda_1 \in 2\ZZ$ for $\lie{g}_\RR = \lie{sp}_{2\cdot}(1,1)$. We prove this by noting that this algebra is isomorphic to $\lie{so}(1,4)$, which we have just treated. (It is also fairly easy to find a direct combinatorial proof for this.)
\item The inequality $|\lambda_3| \leq \lambda_1 - \lambda_2$ for $\lie{g}_\RR = \lie{so}^*(6)$. We prove this by noting that this algebra is isomorphic to $\lie{su}(1,3)$, which we have already treated in Section~\ref{sec:Ar}. (A direct combinatorial proof probably also exists, but seems fairly tedious on first approach.) \qedhere
\end{itemize}
\end{proof}

\subsection{The inclusion $\mathcal{M}_{\theor} \subset \mathcal{M}_{\linv}$}
\label{sec:theor_in_linv}

In this subsection, we prove that, conversely, all elements $\lambda \in \mathcal{M}_{\theor}$ satisfy $V^\lie{l}_\lambda \neq 0$. We rely for this on Proposition~\ref{closed_under_addition}, that reduces the problem to the basis of the monoid $\mathcal{M}_{\theor}$ (which, in contrast to the $A_r$ case, can be easily described). For each $\lambda$ lying in this basis, thanks to the work done in Section~\ref{sec:BCD}, our goal is to construct a doubled Young tableau of shape $\Psi(\lambda)$ satisfying conditions \ref{itm:dsYt_stst}--\ref{itm:shape} and \labellcodominantsym{} from Corollary~\ref{combinatorial_characterization_without_sign}.

We start by presenting (Definition~\ref{example_tableaux_definition}) nine infinite families of doubled Young tableaux, and checking their properties (Proposition~\ref{prop_example_tableaux_properties}). All the required doubled Young tableaux will then be picked from this pool, sometimes with the symbols all shifted by some constant~$x$. This ``shift'' operation will be rigorously defined in Definition~\ref{shift_definition}.

\begin{definition}
\label{example_tableaux_definition}
We introduce the doubled Young tableaux $\mathcal{T}_K$ and $\mathcal{T}'_K$ (of shape $2\mathcal{C}_K$), $\mathcal{T}_{K,L}$ and $\mathcal{T}'_{K,L}$ (of shape $2\mathcal{C}_K + 2\mathcal{C}_L$), $\mathcal{S}_{K,K}$ and~$\mathcal{S}'_{K,K}$ (of shape $4\mathcal{C}_K$) for some values of the parameters $K$ and~$L$, as given in Figure~\ref{fig:example_tableaux}.
\end{definition}

\newcommand\moins\overline
\newcommand\half{0.5}
\Yboxdimy{18pt}
\newlength{\boxwidth}

\begin{figure}[p]
\DeclareRobustCommand\minidiagramXtoY{\diagramfontsize
  \Yboxdim{12pt}
  \begin{tikzpicture}[baseline={([yshift=-.5ex]current bounding box.center)}]
  \Yfillcolor{black!20}
  \tgyoung(0cm,0cm,)
  \Yfillopacity{0}
  \Ylinecolor{black!30}
  \tgyoung(0cm,0cm,x,\vdts,y)
  \Ylinecolor{black}
  \tgyoung(0cm,0cm,|3)
  \end{tikzpicture}
  \normalsize}
\DeclareRobustCommand\minidiagramXtoXmo{\diagramfontsize
  \Yboxdimy{12pt}
  \settowidth{\boxwidth}{$\;x-1\;$}
  \Yboxdimx{\boxwidth}
  \begin{tikzpicture}[baseline={([yshift=-.5ex]current bounding box.center)}]
  \Yfillcolor{black!20}
  \tgyoung(0cm,0cm,)
  \Yfillopacity{0}
  \Ylinecolor{black!30}
  \tgyoung(0cm,0cm,x,\vdts,<x-1>)
  \Ylinecolor{black}
  \tgyoung(0cm,0cm,|3)
  \end{tikzpicture}
  \normalsize}
\DeclareRobustCommand\minidiagramXXmo{\diagramfontsize
  \Yboxdimy{12pt}
  \settowidth{\boxwidth}{$\;x-1\;$}
  \Yboxdimx{\boxwidth}
  \begin{tikzpicture}[baseline={([yshift=-.5ex]current bounding box.center)}]
  \Yfillcolor{black!20}
  \tgyoung(0cm,0cm,)
  \Yfillopacity{0}
  \Ylinecolor{black!30}
  \tgyoung(0cm,0cm,)
  \Ylinecolor{black}
  \tgyoung(0cm,0cm,x,<x-1>)
  \end{tikzpicture}
  \normalsize}
\DeclareRobustCommand\minidiagramYbartoXbar{\diagramfontsize
  \Yboxdim{12pt}
  \begin{tikzpicture}[baseline={([yshift=-.5ex]current bounding box.center)}]
  \Yfillcolor{black!20}
  \tgyoung(0cm,0cm,|3)
  \Yfillopacity{0}
  \Ylinecolor{black!30}
  \tgyoung(0cm,0cm,<\overline{y}>,\vdts,<\overline{x}>)
  \Ylinecolor{black}
  \tgyoung(0cm,0cm,|3)
  \end{tikzpicture}
  \normalsize}
\caption{\label{fig:example_tableaux} Doubled Young tableaux introduced in Definition~\ref{example_tableaux_definition}. Here it is understood that each block of the form \minidiagramXtoY is filled with consecutive symbols from $x$ to $y$ in increasing order, so that it has height~$y-x+1$. Of course we keep this convention even if $y = x-1$, so that \minidiagramXtoXmo represents the zero-height block (and not the block \smash{\minidiagramXXmo} with height~$2$). Similarly, each block of the form~\smash[t]{\minidiagramYbartoXbar} is filled with symbols from $\overline{y}$ to~$\overline{x}$ in increasing $\prec$ order, which corresponds to the decreasing order of their absolute values. (Boxes containing symbols with bars are shaded for better legibility.) Also note that in these pictures, whenever two boxes seem to be at the same height, they actually are at the same height, regardless of the values of $k$ and~$l$.} 
\centering\bigskip\bigskip
\addtocounter{figure}{-1} %I do not understand why I need this, but I do.
\begin{subfigure}{0.3\textwidth}
\centering
\diagramfontsize
\settowidth{\boxwidth}{$\;2k+1\;$}
\Yboxdimx{\boxwidth}
\begin{tikzpicture} %2k+1
\Yfillcolor{black!20}
\tgyoung(0cm,0cm,,,,:|4:,|3,:~)
\Yfillopacity{0}
\Ylinecolor{black!30}
\tgyoung(0cm,0cm,1<k+2>:,|2\vdts<\vdts>:,:;<2k+1>:,<k+1><\moins{k+1}>:,<\moins{2k+1}>|2\vdts:,<\vdts>:,<\moins{k+2}><\moins{1}>)
\Ylinecolor{black}
\tgyoung(0cm,0cm,|4|3,:|4:,|3)
\end{tikzpicture}
\caption{$\mathcal{T}_{2k+1}$, \\ for $k \geq 0$.}
\label{fig:T2k+1}
\end{subfigure}
~
\begin{subfigure}{0.3\textwidth}
\centering
\diagramfontsize
\settowidth{\boxwidth}{$\;k+1\;$}
\Yboxdimx{\boxwidth}
\begin{tikzpicture} %2k
\Yfillcolor{black!20}
\tgyoung(0cm,0cm,,,,,|4|4)
\Yfillopacity{0}
\Ylinecolor{black!30}
\tgyoung(0cm,0cm,1<k+1>:,|2\vdts|2\vdts,<k><2k>:,<\moins{2k}><\moins{k}>:,|2\vdts|2\vdts,<\moins{k+1}><\moins{1}>)
\Ylinecolor{black}
\tgyoung(0cm,0cm,|4|4,|4|4)
\end{tikzpicture}
\caption{$\mathcal{T}_{2k}$, \\ for $k \geq 1$.}
\label{fig:T2k}
\end{subfigure}
~
\begin{subfigure}{0.3\textwidth}
\centering
\diagramfontsize
\settowidth{\boxwidth}{$\;k+1\;$}
\Yboxdimx{\boxwidth}
\begin{tikzpicture} %2k'
\Yfillcolor{black!20}
\tgyoung(0cm,0cm,,,,,|4|4)
\Yfillopacity{0}
\Ylinecolor{black!30}
\tgyoung(0cm,0cm,1k,\vdts<k+2>,<k-1>\vdts,<k+1><2k>,<\moins{2k}><\moins{k+1}>,\vdts<\moins{k-1}>,<\moins{k+2}>\vdts,<\moins{k}><\moins{1}>)
\Ylinecolor{black}
\tgyoung(0cm,0cm,|3|1,:|3:,:,|1,|3|1,:|3:,:,|1)
\end{tikzpicture}
\caption{$\mathcal{T}_{2k}'$, \\ for $k \geq 2$.}
\label{fig:T2k'}
\end{subfigure}
\end{figure}

\begin{figure}\ContinuedFloat
\centering
\sbox0{
\begin{subfigure}{0.48\textwidth}
\centering
\diagramfontsize
\settowidth{\boxwidth}{$\;k+l+2\;$}
\Yboxdimx{\boxwidth}
\begin{tikzpicture} %2k+1,2l+1
\newcommand\tailheight{9.5}
\newcommand\vdtsheight{7.5}
\Yfillcolor{black!20}
\tgyoung(0cm,0cm,,,,,::|\tailheight|\tailheight:,,,,|7|7)
\Yfillopacity{0}
\Ylinecolor{black!30}
\tgyoung(0cm,0cm,1<k-l+1><k+2><k+2>,|6\vdts|3\vdts|2\vdts|2\vdts,::;<k+l+1><k+l+1>,:;<k+1><\moins{k+1}><\moins{k+1}>,:;<k+l+2>|\vdtsheight\vdts|\vdtsheight\vdts:,:;\vdts,<k+1><2k+1>,<\moins{2k+1}><\moins{k+l+1}>,|5\vdts|2\vdts,:;<\moins{k+2}>,:;<\moins{k-l}>/\half,::;<\moins{k-l+1}><\moins{k-l+1}>/\half,:;\vdts,<\moins{k+2}><\moins{1}>)
\Ylinecolor{black}
\tgyoung(0cm,0cm,|8|5|4|4,::|\tailheight|\tailheight:,:|3,|7|4,:|3)
\end{tikzpicture}
\caption{$\mathcal{T}_{2k+1,2l+1}$, \\ for $k \geq l \geq 0$.}
\label{fig:T2k+1_2l+1}
\end{subfigure}
}
\sbox1{
\begin{subfigure}{0.48\textwidth}
\centering
\diagramfontsize
\settowidth{\boxwidth}{$\;k+l+2\;$}
\Yboxdimx{\boxwidth}
\begin{tikzpicture} %2k+1,2l+1'
\newcommand\tailheight{10.5}
\newcommand\vdtsheight{6.5}
\Yfillcolor{black!20}
\tgyoung(0cm,0cm,,,,::|\tailheight|\tailheight:,,,,,,|6|6)
\Yfillopacity{0}
\Ylinecolor{black!30}
\tgyoung(0cm,0cm,1<k-l+1><k+3><k+3>,|7\vdts|2\vdts\vdts\vdts,::;<k+l+1><k+l+1>,:;<k><\moins{k+2}><\moins{k+2}>,:;<k+1><\moins{k+1}><\moins{k+1}>,:;<k+2><\moins{k}><\moins{k}>,:;<k+l+2>|\vdtsheight\vdts|\vdtsheight\vdts:,:;\vdts,<k+2><2k+1>,<\moins{2k+1}><\moins{k+l+1}>,|4\vdts\vdts,:;<\moins{k+3}>,:;<\moins{k-l}>/\half,::;<\moins{k-l+1}><\moins{k-l+1}>/\half,:;\vdts,<\moins{k+3}><\moins{1}>)
\Ylinecolor{black}
\tgyoung(0cm,0cm,|9|6|3|3,::|\tailheight|\tailheight:,,,:|3,|6|3,:|3)
\end{tikzpicture}
\caption{$\mathcal{T}_{2k+1,2l+1}'$, \\ for $k \geq l > 0$.}
\label{fig:T2k+1_2l+1'}
\end{subfigure}
}
\sbox2{
\begin{subfigure}{0.48\textwidth}
\centering
\diagramfontsize
\settowidth{\boxwidth}{$\;2k+1\;$}
\Yboxdimx{\boxwidth}
\begin{tikzpicture} %2k+1,1'
\Yfillcolor{black!20}
\tgyoung(0cm,0cm,::|1|1,,,,,|4|4)
\Yfillopacity{0}
\Ylinecolor{black!30}
\tgyoung(0cm,0cm,1<k+1><\moins{k+2}><\moins{k+2}>,|2\vdts|3\vdts:,,k,<k+2><2k+1>,<\moins{2k+1}><\moins{k}>,\vdts|2\vdts:,<\moins{k+3}>,<\moins{k+1}><\moins{1}>)
\Ylinecolor{black}
\tgyoung(0cm,0cm,|4|5|1|1,,,,|1,|3|4:,,,|1)
\end{tikzpicture}
\caption{$\mathcal{T}_{2k+1,1}'$, \\ for $k \geq 1$.}
\label{fig:T2k+1_1'}
\end{subfigure}
}
\sbox3{
\begin{subfigure}{0.48\textwidth}
\centering
\diagramfontsize
\settowidth{\boxwidth}{$\;k+l+2\;$}
\Yboxdimx{\boxwidth}
\begin{tikzpicture} %2k,2l+1
\newcommand\tailheight{9.5}
\newcommand\vdtsheight{7.5}
\Yfillcolor{black!20}
\tgyoung(0cm,0cm,,,,::|\tailheight|\tailheight:,,,,:|7:,|6)
\Yfillopacity{0}
\Ylinecolor{black!30}
\tgyoung(0cm,0cm,1<k-l+1><k+2><k+2>,|6\vdts|2\vdts\vdts\vdts,::;<k+l+1><k+l+1>,:;<k+1><\moins{k+1}><\moins{k+1}>,:;<k+l+2>|\vdtsheight\vdts|\vdtsheight\vdts:,::\vdts,:;<2k>,<k+1><\moins{k+l+1}>,<\moins{2k}>\vdts,|4\vdts<\moins{k+2}>,:;<\moins{k-l}>,:|2\vdts/\half,::;<\moins{k-l+1}><\moins{k-l+1}>/\half,,<\moins{k+2}><\moins{1}>)
\Ylinecolor{black}
\tgyoung(0cm,0cm,|8|4|3|3,::|\tailheight|\tailheight:,:|3,:|3:,|6:,,:|4)
\end{tikzpicture}
\caption{$\mathcal{T}_{2k,2l+1}$, \\ for $k > l \geq 0$.}
\label{fig:T2k_2l+1}
\end{subfigure}
}
\usebox0\usebox1

\vspace{\baselineskip}

\adjustbox{valign=T}{\usebox3}\adjustbox{valign=T}{\usebox2}
\end{figure}

\begin{figure}\ContinuedFloat
\centering
\begin{subfigure}{0.4\textwidth}
\centering
\diagramfontsize
\settowidth{\boxwidth}{$\;2k+1\;$}
\Yboxdimx{\boxwidth}
\begin{tikzpicture} %2k+1,2k+1''
\Yfillcolor{black!20}
\tgyoung(0cm,0cm,,,,,:|5|5|5:,,|3)
\Yfillopacity{0}
\Ylinecolor{black!30}
\tgyoung(0cm,0cm,11<k+1><k+2>,|3\vdts|2\vdts|2\vdts|2\vdts,:;k<2k><2k+1>,<k+1><\moins{2k+1}><\moins{2k+1}><\moins{k+1}>,<2k+1>|3\vdts<\moins{k}>|3\vdts:,<\moins{2k}>:|2\vdts:,\vdts,<\moins{k+2}><\moins{k+1}><\moins{1}><\moins{1}>)
\Ylinecolor{black}
\tgyoung(0cm,0cm,|5|4|4|4,:|5|1|5:,|1:|4:,|3)
\end{tikzpicture}
\caption{$\mathcal{S}_{2k+1,2k+1}$, \\ for $k \geq 1$.}
\label{fig:S2k+1_2k+1}
\end{subfigure}
\begin{subfigure}{0.4\textwidth}
\centering
\diagramfontsize
\settowidth{\boxwidth}{$\;2k+1\;$}
\Yboxdimx{\boxwidth}
\begin{tikzpicture} %2k+1,2k+1'''
\Yfillcolor{black!20}
\tgyoung(0cm,0cm,,::|8|8:,,,,,,:|2)
\Yfillopacity{0}
\Ylinecolor{black!30}
\tgyoung(0cm,0cm,11<2k-1><2k+1>,|4\vdts|4\vdts<\moins{2k+1}><\moins{2k}>,::;<\moins{2k}><\moins{2k-1}>,::;<\moins{2k-2}><\moins{2k-2}>,::|4\vdts|4\vdts:,<2k-2><2k-2>,<2k-1><2k>,<2k><\moins{2k+1}>,<2k+1><\moins{2k-1}><\moins{1}><\moins{1}>)
\Ylinecolor{black}
\tgyoung(0cm,0cm,|6|6|1|1,::|2|8:,,::|6:,,,|3|1,:|1,:|1)
\end{tikzpicture}
\caption{$\mathcal{S}_{2k+1,2k+1}'$, \\ for $k \geq 2$.}
\end{subfigure}
\end{figure}

\begin{table}
\caption{Properties of the tableaux introduced in Figure~\ref{fig:example_tableaux}.}
\label{tab:example_tableaux_properties}
\centering
\renewcommand{\arraystretch}{2}
\makebox[\textwidth][c]{
\begin{tabular}{lr|l|l} %TODO Typeset nicely...
Tableau~$\mathcal{T}$ && $\mathcal{T}$ is $\lie{g}$-standard for $\lie{g} = \ldots$ & \makecell[l]{$\mathcal{T}$ is $\alpha$-codominant for\\ all $\alpha$ except...} \\ \hline
$\mathcal{T}_{2k+1}$ &%for $k \geq 0$ 
 & $B_{2k+1}$
 & $e_{k+1}-e_{k+2}$ or $e_{k+1}$ \\ \hline
$\mathcal{T}_{2k}$ &%for $k \geq 1$
 & $B_r$, $C_r$, $D_r$ for $r \geq 2k$
 & $e_k-e_{k+1}$ \\ \hline
$\mathcal{T}_{2k}'$ &for $k \neq 1$ %$k \geq 2$
 & $C_r$, $D_r$ for $r \geq 2k$
 & $\begin{cases} e_{k-1}-e_k; \\ e_{k+1}-e_{k+2}\end{cases}$ \\ \hline
$\mathcal{T}_{2k+1,2l+1}$ &%for $k \geq l \geq 0$
 & \makecell[l]{$B_r$, $C_r$, $D_r$ for $r \geq 2k+1$,\\ \quad except $D_{2k+1}$ if $l=k$}
 & $\begin{cases}e_{k+1}-e_{k+2} \text{ or } e_{k+1} \mathrlap{\text{ or } 2e_{k+1};} \\ e_2+e_3 \quad \text{ if } k=1, l=0\end{cases}$ \\ \hline
$\mathcal{T}_{2k+1,2l+1}'$ &for $l \neq 0$ %$k \geq l > 0$
 & \makecell[l]{$C_r$, $D_r$ for $r \geq 2k+1$,\\ \quad except $D_{2k+1}$ if $l=k$}
 & $\begin{cases}e_{k+2}-e_{k+3} \text{ or } 2e_{k+2}; \\ e_4+e_5 \quad \text{ if } k=2, l=1\end{cases}$ \\ \hline
$\mathcal{T}_{2k+1,1}'$ &for $k \neq 0$ %$k \geq 1$
 & $C_r$, $D_r$ for $r \geq 2k+1$
 & $\begin{cases}e_k-e_{k+1}; \\ e_{k+2}-e_{k+3} \text{ or } 2e_{k+2}; \\ e_2+e_3; \\ e_4+e_5\end{cases}$ \\ \hline
$\mathcal{T}_{2k,2l+1}$ &%for $k > l \geq 0$
 & $B_{2k}$
 & $e_{k+1}-e_{k+2}$ or $e_{k+1}$ \\ \hline
$\mathcal{S}_{2k+1,2k+1}$ &for $k \neq 0$ %$k \geq 1$
 & $D_{2k+1}$
 & $\begin{cases}e_k-e_{k+1}; \\ e_{k+1}-e_{k+2}; \\ e_2+e_3\end{cases}$ \\ \hline
$\mathcal{S}_{2k+1,2k+1}'$ &for $k \neq 0,1$ %$k \geq 2$
 & $D_{2k+1}$
 & $\begin{cases}e_{2k-2}-e_{2k-1}; \\ e_{2k}-e_{2k+1} \text{ and } e_{2k}+e_{2k\mathrlap{{}+1}}\end{cases}$
\end{tabular}
}
\end{table}

\begin{proposition}
\label{prop_example_tableaux_properties}
Let $\mathcal{T}$ be any Young tableau from this list, and let $\lie{g}$ be any of the Lie algebras listed in the second column of Table~\ref{tab:example_tableaux_properties}. Then $\mathcal{T}$ is a $\lie{g}$-standard doubled Young tableau, is null, and is $\alpha$-codominant for all simple roots $\alpha$ of~$\lie{g}$ that do \emph{not} belong to the list given in the third column of Table~\ref{tab:example_tableaux_properties}. 
\end{proposition}

Beware that the value of~$r$ is now allowed to vary, while the tableaux are fixed.

\begin{proof}
Observe that all of these tableaux~$\mathcal{T}$ happen to satisfy the following condition:
\begin{itemize}
\item[(*)] Each column of~$\mathcal{T}$ is filled with symbols with consecutive absolute values, \ie $k = 1$ in the notations of Proposition~\ref{admissible_pair_characterization}.
\end{itemize}
Also recall (Definition~\ref{tableau_definition}) that $\lie{g}$-standardness involves the three conditions \ref{itm:dsYt_stst} (strong semistandardness of columns), \ref{itm:dsYt_Young} (that the columns form a $\preceq^{\lie{g}}_Y$-nondecreasing sequence) and \ref{itm:dsYt_adm} (that some pairs of consecutive columns are admissible). 
\begin{itemize}
\item Strong semistandardness of columns \ref{itm:dsYt_stst}, and the fact that these tableaux are all null, are completely straightforward to check.
\item Condition~\ref{itm:dsYt_Young} can be seen as encompassing two properties. First of all, we can check that these tableaux are all $\preceq_\mathcal{A}$-semistandard. This is fairly tedious, but straightforward. %TODO Add an example.
\item When $\lie{g} = B_r$ or $C_r$, this is all; but when $\lie{g} = D_r$, condition~\ref{itm:dsYt_Young} also involves the parity condition~\eqref{eq:parity_condition}. Given the property~(*), this parity condition can be rephrased as follows: whenever we have, for some $j$,
\begin{equation}
\label{eq:parity_precondition}
b_1 \left( \threeind{j}{}{} \mathcal{T} \right) = b_1 \left( \threeind{j+1}{}{} \mathcal{T} \right) = r,
\end{equation}
(where $b_1$, as per the notations of Proposition~\ref{admissible_pair_characterization}, stands for the largest absolute value of a symbol in the column), we need to have
\begin{equation}
\label{eq:parity_reformulated}
\# \threeind{j}{\NN}{} \mathcal{T} \equiv \# \threeind{j+1}{\NN}{} \mathcal{T} \pmod{2}.
\end{equation}
We observe moreover that $b_1$ never exceeds the height of the column, which is at most~$K$ (\ie $2k$ or~$2k+1$); so the parity condition is vacuously true for $r > K$. When $r = K$, the prerequisite~\eqref{eq:parity_precondition} is satisfied:
\begin{itemize}
\item For the first two columns only in the tableaux $\mathcal{T}_{2k}$, $\mathcal{T}_{2k}'$ and $\mathcal{T}_{2k+1,1}'$, and for all columns in the tableaux $\mathcal{S}_{2k+1,2k+1}$ and $\mathcal{S}_{2k+1,2k+1}'$. Condition~\eqref{eq:parity_reformulated} is then easily checked.
\item For the first two columns only in the tableaux $\mathcal{T}_{2k+1,2l+1}$ and $\mathcal{T}_{2k+1,2l+1}'$, as long as $k > l$ (and condition~\eqref{eq:parity_reformulated} is then easily checked). However when $k = l$, the prerequisite~\eqref{eq:parity_precondition} becomes satisfied also for the last two columns; but condition~\eqref{eq:parity_reformulated} fails between the 2nd and the 3rd column. This is the reason why we have to explicitly exclude these cases (and introduce the tableaux $\mathcal{S}$ and $\mathcal{S}'$ to replace them).
\end{itemize}
%\item Condition~\ref{itm:consecutive_column_parity} needs to be checked only in type~$D_r$, and only for $r = K$. It is easily verified by inspection for the tableaux $\mathcal{T}_{2k}$, $\mathcal{T}_{2k}'$ and $\mathcal{T}_{2k+1,1}'$, $\mathcal{S}_{2k+1,2k+1}$ and $\mathcal{S}_{2k+1,2k+1}'$. For the tableaux $\mathcal{T}_{2k+1,2l+1}$ and $\mathcal{T}_{2k+1,2l+1}'$, the same thing holds, as long as $k > l$. When $k = l$, the prerequisite $s(j) = r$ becomes satisfied not only for the first two, but also for the last two columns; but the condition $b^j \equiv b^{j+1} \pmod{2}$ fails for $j=2$. This is the reason why these cases have to be explicitly excluded.
\item Condition~\ref{itm:dsYt_adm} reduces, by Proposition~\ref{admissible_pair_characterization}, to checking the four properties \ref{itm:adm_Young} through \ref{itm:adm_x_condition} for the first and last pair of columns. In fact, condition~\ref{itm:adm_Young} is already part of~\ref{itm:dsYt_Young}; condition~\ref{itm:adm_same_height} is immediate by inspection; as for conditions \ref{itm:adm_same_decomp} and~\ref{itm:adm_x_condition}, they become immediate by inspection once we take into account property~(*).
\item Finally, verification of $\alpha$-codominance for all $\alpha$ except the listed exceptional values is very tedious, but straightforward. %TODO Add an example?
\qedhere
\end{itemize}
\end{proof}

\begin{definition}
\label{shift_definition}
Given a doubled Young tableau~$\mathcal{T}$ and an integer~$x \geq 0$, we define the \emph{shifted tableau} $\threeind{}{x+}{} \mathcal{T}$ to be the tableau with the same shape, with every symbol~$s$ replaced by $s+x$ and every symbol~$\overline{s}$ replaced by $\overline{s+x}$. Thus, formally, it is given by:
\[\forall i, j,\quad
\begin{cases}
\left| \threeind{j}{x+}{i} \mathcal{T} \right| &:= x + \left| \threeind{j}{}{i} \mathcal{T} \right|; \\
\sgn \left( \threeind{j}{x+}{i} \mathcal{T} \right) &:= \sgn \left( \threeind{j}{}{i} \mathcal{T} \right).
\end{cases}\]
\end{definition}
The following statement is then obvious:
\begin{lemma}
\label{shift_properties}
Keeping the same setup, let us also fix some integer $r \geq 1$ (resp. $r \geq 1$, $r \geq 3$);
and let $\lie{g} = B_r$ (resp. $C_r$, $D_r$) and $\lie{g}' = B_{r+x}$ (resp. $C_{r+x}$, $D_{r+x}$). For all $s$ within the appropriate bounds, we denote by $\alpha_s$ (resp. $\alpha'_s$) the $s$-th simple root of~$\lie{g}$ (resp. of~$\lie{g}'$) in the usual Bourbaki ordering. Then:
\begin{hypothenum}
\item $\threeind{}{x+}{} \mathcal{T}$ is $\lie{g}'$-standard if and only if $\mathcal{T}$ is $\lie{g}$-standard.
\item $\threeind{}{x+}{} \mathcal{T}$ is null if and only if $\mathcal{T}$ is null.
%\item The $\lie{g}'$-algebraic shape of $\threeind{}{x+}{} \mathcal{T}$ coincides with the $\lie{g}$-algebraic shape of $\mathcal{T}$.
\item For $s < x$, $\threeind{}{x+}{} \mathcal{T}$ is always $\alpha'_s$-codominant.
\item For $s = x$, $\threeind{}{x+}{} \mathcal{T}$ is always $\alpha'_x$-codominant, as soon as $\mathcal{T}$ is semistandard for the $\preceq_\mathcal{A}$~order.
\item For $s > x$, $\threeind{}{x+}{} \mathcal{T}$ is $\alpha'_s$-codominant if and only if $\mathcal{T}$ is $\alpha_{s-x}$-codominant.
\end{hypothenum}
\end{lemma}

\begin{proof}[Proof that $\mathcal{M}_{\theor} \subset \mathcal{M}_{\linv}$ for $\lie{g}$ of types $B_{r \geq 1}$, $C_{r \geq 1}$ or $D_{r \geq 3}$.]
By Proposition~\ref{closed_under_addition}, it suffices to prove that, for every $\lambda$ lying in the basis of the monoid~$\mathcal{M}_{\theor}$, we have $V_\lambda^\lie{l} \neq 0$. By Corollary~\ref{combinatorial_characterization_without_sign}, it suffices to find, for every such~$\lambda$, a $(\Theta \cup \sigma \Theta)$-codominant null $\lie{g}$-standard filling of the doubled Young diagram $\Psi(\lambda)$.

\begin{table}
  \caption{\label{tab:tableaux_for_monoid_basis} Images by $\Psi$ of the primitive elements of $\mathcal{M}_{\theor}$ for most real Lie algebras of types $B_r$, $C_r$ and~$D_r$, and their $(\Theta \cup \sigma \Theta)$-codominant null $\lie{g}$-standard fillings. The symbol "$\triangle$" stands for symmetric difference.
  }
  \centering
  \aboverulesep=.2ex
  \belowrulesep=.2ex
  \renewcommand{\arraystretch}{1.5}
  \makebox[\textwidth][c]{
  \begin{tabular}{ccll@{}rcl@{}l}
    \multicolumn{2}{c}{$\lie{g}_\RR$} & $\Psi(\lambda)$ & Parameter range & Subrange & $\mathcal{T}$ & Syndrome & ~ \\
    \midrule
    \multirow{5}{*}{\rotatebox[origin=c]{90}{$\underset{0 \leq p \leq r}{\lie{so}(p,2r+1-p)}$}}
    & \multirow{5}{*}{\rotatebox[origin=c]{90}{$\Theta = \Pi_{[p+1,r]}, \sigma = \Id$}}
    & $2\mathcal{C}_{2k}$ & $\begin{cases}1 \leq k \leq p \\ 2k \leq r\end{cases}$ & & $\mathcal{T}_{2k}$ & $\alpha_k$ \\ \cmidrule{3-7}
    & & $2(\mathcal{C}_{2k+1} + \mathcal{C}_{2l+1})$ & $\begin{cases}0 \leq l \leq k < p \\ 2k+1 \leq r\end{cases}$ & & $\mathcal{T}_{2k+1,2l+1}$ & $\alpha_{k+1}$ \\ \cmidrule{3-7}
    & & $2\mathcal{C}_{2k+1}$ & $r-p \leq k \leq \frac{r-1}{2}$ & & $\mathclap{\threeind{}{r-2k-1+}{} \mathcal{T}_{2k+1}}$ & $\alpha_{r-k}$ \\ \cmidrule{3-7}
    & & $2(\mathcal{C}_{2k} + \mathcal{C}_{2l+1})$ & $0 \leq l < r-p \mathrlap{{} < k \leq \frac{r}{2}}$ & & $\hspace{10.5mm}\mathclap{\threeind{}{r-2k+}{} \mathcal{T}_{2k, 2l+1}}\hspace{10.5mm}$ & $\alpha_{r-k+1}$ \\
    \midrule
    \multirow{7}{*}{\rotatebox[origin=c]{90}{$\underset{r \geq 3,\; 0 \leq p \leq \frac{r}{2}}{\lie{sp}_{2\cdot}(p,r-p)}$}}
    & \multirow{7}{*}{\rotatebox[origin=c]{90}{$\Theta = \Pi_{\operatorname{odd}} \cup \Pi_{[2p+1,r]}, \sigma = \Id$}}
    & \multirow{3}{*}{$2\mathcal{C}_{2k}$} & \multirow{3}{*}{$\begin{cases}1 \leq k \leq 2p \\ 2k \leq r\end{cases}$} & $k$ even & $\mathcal{T}_{2k}$ & $\alpha_k$ \\
    & & & & $k > 1$ odd & $\mathcal{T}_{2k}'$ & $\alpha_{k-1}, \alpha_{k+1}$ \\
    & & & & $k = 1$ & $\threeind{}{1+}{} \mathcal{T}_2$ & $\alpha_2$ \\ \cmidrule{3-7}
    & & \multirow{4}{*}{$2(\mathcal{C}_{2k+1} + \mathcal{C}_{2l+1})$} & \multirow{4}{*}{$\begin{cases}0 \leq l \leq k < 2p \\ 2k+1 \leq r\end{cases}$} & $k$ odd & $\mathcal{T}_{2k+1,2l+1}$ & $\alpha_{k+1}$ \\
    & & & & $k$ even, $l > 0$ & $\mathcal{T}_{2k+1,2l+1}'$ & $\alpha_{k+2}$ \\
    & & & & $k > 0$ even, $l = 0$ & $\mathcal{T}_{2k+1,1}'$ & $\alpha_k, \alpha_{k+2}$ \\
    & & & & $k = l = 0$ & $\threeind{}{1+}{} \mathcal{T}_{1,1}$ & $\alpha_2$ \\
    \midrule
    \multirow{4}{*}{\rotatebox[origin=c]{90}{$\underset{r \geq 3,\; 0 \leq p \leq r}{\lie{so}(p,2r-p)}$}}
    & \multirow{3}{*}{\rotatebox[origin=c]{90}{$\Theta = \sigma \Theta \subset \Pi_{[p+1,r]}$}}
    & $2\mathcal{C}_{2k}$ & $\begin{cases}1 \leq k \leq p \\ 2k \leq r\end{cases}$ & & $\mathcal{T}_{2k}$ & $\alpha_k$ \\ \cmidrule{3-7}
    & & \multirow{3}{*}{$2(\mathcal{C}_{2k+1} + \mathcal{C}_{2l+1})$} & \multirow{3}{*}{$\begin{cases}0 \leq l \leq k < p \\ 2k+1 \leq r\end{cases}$} & $2l+1 < r$ & $\mathcal{T}_{2k+1,2l+1}$ & \multicolumn{2}{l}{$\begin{cases} \alpha_{k+1} \\ \alpha_r \text{ for } \mathcal{T}_{3,1} \text{ if } r\mathrlap{{}=3}\end{cases}$} \\
    & & & & $\mathllap{2l+1 = {}} 2k+1 = r$ & $\mathcal{S}_{2k+1,2k+1}$ & \multicolumn{2}{l}{$\begin{cases} \alpha_k, \alpha_{k+1} \\ \alpha_r \text{ if } r\mathrlap{=3}\end{cases}$} \\
    \midrule
    \multirow{9}{*}{\rotatebox[origin=c]{90}{$\underset{r \geq 4}{\lie{so}^*(2r)}$}}
    & \multirow{8.5}{*}{\rotatebox[origin=c]{90}{$\Theta \cup \sigma \Theta = \Pi_{\operatorname{odd}} \,\triangle\, \{\alpha_r\}$}}
    & \multirow{3}{*}{$2\mathcal{C}_{2k}$} & \multirow{3}{*}{$1 \leq k \leq \frac{r}{2}$} & $k$ even & $\mathcal{T}_{2k}$ & $\alpha_k$ \\
    & & & & $k > 1$ odd & $\mathcal{T}_{2k}'$ & $\alpha_{k-1}, \alpha_{k+1}$ \\
    & & & & $k = 1$ & $\threeind{}{1+}{} \mathcal{T}_2$ & $\alpha_2$ \\ \cmidrule{3-7}
    & & \multirow{6}{*}{$2(\mathcal{C}_{2k+1} + \mathcal{C}_{2l+1})$} & \multirow{6}{*}{$0 \leq l \leq k \mathrlap{{} \leq \frac{r-1}{2}}$} & $k$ odd, $2l+1 < r$ & $\mathcal{T}_{2k+1,2l+1}$ & $\alpha_{k+1}$ \\
    & & & & \llap{$k > 0$ even,} $1 < 2l+1 < r$ & $\mathcal{T}_{2k+1,2l+1}'$ & \multicolumn{2}{l}{$\begin{cases} \alpha_{k+2} \\ \alpha_r \text{ for } \mathcal{T}_{5,3}' \text{ if } r\mathrlap{{}=5}\end{cases}$} \\
    & & & & $k > 0$ even, $l = 0$ & $\mathcal{T}_{2k+1,1}'$ & \multicolumn{2}{l}{$\begin{cases} \alpha_k, \alpha_{k+2} \\ \alpha_r \text{ for } \mathcal{T}_{5,1}' \text{ if } r\mathrlap{{}=5}\end{cases}$} \\
    & & & & $k = l = 0$ & $\threeind{}{1+}{} \mathcal{T}_{1,1}$ & $\alpha_2$ \\
    & & & & $\mathllap{2l+1 = {}} 2k+1 = r$ & $\mathcal{S}'_{2k+1,2k+1}$ & \multicolumn{2}{l}{$\alpha_{2k-2}, \alpha_{2k}, \alpha_{2k+1}$} \\ \bottomrule
  \end{tabular}
  }
\end{table}

For most values of $\lie{g}_\RR$, this is done in Table~\ref{tab:tableaux_for_monoid_basis}. First of all, it is straightforward to verify that, for each of the $\lie{g}_\RR$ mentioned in that table, the image of the basis of the monoid $\mathcal{M}_{\theor}$ by the map~$\Psi$ is as listed in the second and third column. (Recall from Definition~\ref{Young_monoid_definition} that $\mathcal{C}_i$ denotes the Young diagram comprising a single column of height~$i$: thus we have, for every~$i$,
\[2\mathcal{C}_i = \Psi(c^\pm_i),\]
where $c^\pm_i$ is as defined in~\eqref{eq:model_weight_definitions}.)

Moreover, for every such diagram $\Psi(\lambda)$, we deduce as a particular case of Proposition~\ref{example_tableaux_definition}, possibly using Lemma~\ref{shift_properties} when a shift is involved, that its filling~$\mathcal{T}$ listed in the fifth column has the following properties. %TODO We possibly need to make dependence on $\lie{g}_\RR$ explicit from time to time, as in the previous section we fix $\mathcal{T}$ and make $\lie{g}$ vary.
\begin{itemize}
\item $\mathcal{T}$ is null and $\lie{g}$-standard. This can be checked by a simple lookup in Table~\ref{tab:example_tableaux_properties}, and applying a shift as needed.
\item The set of simple roots~$\alpha \in \Pi(\lie{g})$ for which $\mathcal{T}$ is \emph{not} codominant --- let us call it the \emph{syndrome} of~$\mathcal{T}$ --- is as listed in the sixth column. Indeed the syndrome is obtained by taking the list given in the third column of Table~\ref{tab:example_tableaux_properties}, shifting it if needed, and intersecting it with $\Pi(\lie{g})$.
\item This syndrome is disjoint from $(\Theta \cup \sigma \Theta)$, whose value we have reminded in the first column. This usually easily follows from the inequalities and parity conditions on~$k$. It is maybe worth explaining why the root~$\alpha_r = \alpha_3$, which sometimes occurs in the syndrome of~$\mathcal{T}$ when $\lie{g}_\RR = \lie{so}(p,q)$ with $p+q = 6$, never lies in $\Theta \cup \sigma \Theta$. Indeed this happens only when $\mathcal{T}$ has height $2k+1 = 3$; since we always have $k < p$, this means that $p = 2$ or $3$. But in both cases, we actually have $\Theta = \emptyset$ (in other terms $\lie{g}_\RR$~is quasi-split).
\end{itemize}

It remains to take care of the remaining values of $\lie{g}_\RR$. Specifically:
\begin{itemize}
\item The algebras $\lie{g}_\RR = \lie{sp}_{2\cdot}(p,r-p)$ with $1 \leq r \leq 2$. There are only three of them. For the compact real forms $\lie{sp}_{2\cdot}(1)$ and~$\lie{sp}_{2\cdot}(2)$, we have $\mathcal{M}_{\theor} = \{0\}$ and the statement is trivial.

For $\lie{sp}_{2\cdot}(1,1)$, we have $\Theta = \{\alpha_1\}$ (and $\sigma$ is by convention the identity), and we easily check that the basis of $\mathcal{M}_{\theor}$ maps by $\Psi$ to $\{4\mathcal{C}_1, 4\mathcal{C}_2\}$. The respective fillings
\diagramfontsize
  \Yboxdim{12pt}
  \begin{tikzpicture}[baseline={([yshift=-.5ex]current bounding box.center)}]
  \Yfillcolor{black!20}
  \tgyoung(0cm,0cm,::_2)
  \Yfillopacity{0}
  \Ylinecolor{black!30}
  \tgyoung(0cm,0cm,22<\moins{2}><\moins{2}>)
  \Ylinecolor{black}
  \tgyoung(0cm,0cm,_2_2)
  \end{tikzpicture}
\normalsize
(also known as $\threeind{}{1+}{}\mathcal{T}_{1,1}$) and
\diagramfontsize
  \Yboxdim{12pt}
  \begin{tikzpicture}[baseline={([yshift=-.5ex]current bounding box.center)}]
  \Yfillcolor{black!20}
  \tgyoung(0cm,0cm,::`22)
  \Yfillopacity{0}
  \Ylinecolor{black!30}
  \tgyoung(0cm,0cm,11<\moins{2}><\moins{2}>,22<\moins{1}><\moins{1}>)
  \Ylinecolor{black}
  \tgyoung(0cm,0cm,`22`22)
  \end{tikzpicture}
\normalsize
of these two diagrams are then both $\lie{g}$-standard, null, and have syndrome $\{\alpha_2\}$. (Alternatively, we may of course simply invoke the exceptional isomoprhism $\lie{sp}_{2\cdot}(1,1) \simeq \lie{so}(1,4)$, like we did in the previous section.)
\item The algebras $\lie{sp}_{2\cdot r}(\RR)$ are split, so we conclude by Proposition~\ref{basic_results}~\ref{itm:main_split}. (Alternatively, we can of course use the same doubled Young tableaux as for $\lie{sp}_{2\cdot}(p,r-p)$ for any $p \geq \frac{r}{4}$.)
\item For $\lie{g}_\RR = \lie{so}^*(6)$, we have $\Theta = \sigma \Theta = \{\alpha_1\}$, and we easily check that the basis of $\mathcal{M}_{\theor}$ maps by $\Psi$ to $\{4\mathcal{C}_1, 2\mathcal{C}_2, 2\mathcal{C}_3 + 2\mathcal{C}_1\}$. In fact this almost follows the general pattern for $\lie{so}^*(2r)$ with $r \geq 4$, with only $4\mathcal{C}_3$ missing (which corresponds to the last line in Table~\ref{tab:tableaux_for_monoid_basis}). The respective fillings $\threeind{}{1+}{}\mathcal{T}_{1,1}$, $\threeind{}{1+}{}\mathcal{T}_2$ and $\mathcal{T}_{3,1}$ of these tableaux are then $\lie{g}$-standard, null, and have respective syndromes $\{\alpha_2\}$, $\{\alpha_2\}$ and $\{\alpha_2, \alpha_3\}$, all disjoint from $\Theta$ and from $\sigma \Theta$. (Alternatively, we may of course simply invoke the exceptional isomorphism $\lie{so}^*(6) \simeq \lie{su}(1,3)$, like we did in the previous section.) \qedhere
\end{itemize}
\end{proof}

\bibliographystyle{alpha}
\bibliography{/home/ilia/Documents/Travaux_mathematiques/mybibliography}
\end{document}